\documentclass[10pt]{amsart}
\usepackage[margin=1in]{geometry}
\usepackage{amsmath, amsthm, amssymb, amsfonts, color}
\usepackage{diagrams}
\newarrow{Eq}{=}{=}{=}{=}{=}
\usepackage{hyperref}
\hypersetup{
    colorlinks,
    citecolor=blue,
    filecolor=black,
    linkcolor=black,
    urlcolor=black
}

\usepackage[final]{showkeys}
\usepackage[disable]{todonotes}

\renewcommand{\sl}{\mathfrak{sl}}
\newcommand{\asl}{\widehat{\sl}}
\newcommand{\Tr}{\text{Tr}}
\newcommand{\hh}{\mathfrak{h}}
\newcommand{\bb}{\mathfrak{b}}
\newcommand{\cR}{\mathcal{R}}
\newcommand{\JJ}{\mathbb{J}}
\newcommand{\cJ}{\mathcal{J}}
\newcommand{\wh}{\widetilde{\hh}}

\newcommand{\Ad}{\text{Ad}}

\newcommand{\wt}{\text{wt}}
\newcommand{\QQ}{\mathbb{Q}}

\newcommand{\id}{\text{id}}
\newcommand{\wtilde}{\widetilde}
\renewcommand{\AA}{\mathcal{A}}
\newcommand{\CC}{\mathbb{C}}
\newcommand{\RR}{\mathbb{R}}

\newcommand{\eps}{\varepsilon}

\let\mgg\gg
\renewcommand{\gg}{\mathfrak{g}}
\newcommand{\abb}{\widehat{\bb}}
\newcommand{\agg}{\widehat{\gg}}
\newcommand{\wgg}{\widetilde{\gg}}
\newcommand{\Aut}{\text{Aut}}
\newcommand{\wrho}{\widetilde{\rho}}
\newcommand{\hotimes}{\widehat{\otimes}}
\newcommand{\Hom}{\text{Hom}}
\newcommand{\DD}{\mathcal{D}}
\newcommand{\MM}{\mathcal{M}}
\newcommand{\whh}{\widehat{h}}
\newcommand{\End}{\text{End}}

\newcommand{\ch}{h^\vee} 
\newcommand{\ahh}{\widehat{\hh}}
\newcommand{\wPhi}{\wtilde{\Phi}}
\newcommand{\wPsi}{\wtilde{\Psi}}
\newcommand{\wJ}{\wtilde{J}}

\newcommand{\Ind}{\text{Ind}}
\newcommand{\hV}{\widehat{V}}
\newcommand{\cL}{\mathcal{L}}
\newcommand{\LL}{\mathbb{L}}
\newcommand{\trig}{\text{trig}}
\newcommand{\GG}{\mathbb{G}}
\newcommand{\weta}{\wtilde{\eta}}
\newcommand{\wGG}{\wtilde{\GG}}

\theoremstyle{definition}
\newtheorem{theorem}{Theorem}[section]
\newtheorem{prop}[theorem]{Proposition}
\newtheorem{lemma}[theorem]{Lemma}
\newtheorem{corr}[theorem]{Corollary}
\newtheorem*{remark}{Remark}

\numberwithin{equation}{section}

\begin{document}
\title[Traces of intertwiners for quantum affine algebras and difference equations]{Traces of intertwiners for quantum affine algebras and difference equations \\
(after Etingof-Schiffmann-Varchenko)}
\author{Yi Sun}
\address{Y.S.: Department of Mathematics\\ Columbia University\\ 2990 Broadway\\ New York, NY 10027, USA}
\email{yisun@math.columbia.edu}
\date{\today}

\begin{abstract}
We modify and give complete proofs for the results of Etingof-Schiffmann-Varchenko in \cite{ESV} on traces of intertwiners of untwisted quantum affine algebras in the opposite coproduct and the standard grading.  More precisely, we show that certain normalized generalized traces $F^{V_1, \ldots, V_n}(z_1, \ldots, z_n; \lambda, \omega, \mu, k)$ for $U_q(\agg)$ solve four commuting systems of $q$-difference equations: the Macdonald-Ruijsenaars, dual Macdonald-Ruijsenaars, $q$-KZB, and dual $q$-KZB equations.  In addition, we show a symmetry property for these renormalized trace functions.  Our modifications are motivated by their appearance in the recent work \cite{Sun:qafftr}.
\end{abstract}

\maketitle
\setcounter{tocdepth}{1}
\tableofcontents

\section{Introduction}

This work presents complete proofs for modifications of the results of Etingof-Schiffmann-Varchenko in \cite{ESV} on traces of intertwiners of untwisted quantum affine algebras.  In that work, certain normalized generalized traces $F^{V_1, \ldots, V_n}(z_1, \ldots, z_n; \lambda, \omega, \mu, k)$ for $U_q(\agg)$ were shown to solve four commuting systems of $q$-difference equations: the Macdonald-Ruijsenaars, dual Macdonald-Ruijsenaars, $q$-KZB, and dual $q$-KZB equations.  In addition, these renormalized trace functions were shown to satisfy a symmetry property.  These results were generalizations to the quantum affine setting of the prior results \cite{EV, ES} of Etingof-Varchenko and Etingof-Schiffmann for finite-type quantum groups and \cite{Eti4, ES3} of Etingof and Etingof-Schiffmann for classical affine algebras.  The $q$-KZB equations which appear were previously studied by Felder-Tarasov-Varchenko in \cite{FTV, FTV2, FV}.

The purpose of the present work is twofold.  First, we provide an exposition of the proofs omitted from \cite{ESV}.  Second, we modify the statements of \cite{ESV} to use the opposite coproduct and use the standard grading instead of the principal grading.  The second modification in particular allows trigonometric limits of the results to be easily taken to recover the results of \cite{EV} for finite-type quantum groups.

These modifications were motivated by the recent work \cite{Sun:qafftr} of the author, where a trace function for $U_q(\asl_2)$ was explicitly computed and related to certain theta hypergeometric integrals appearing in \cite{FV} as part of Felder-Varchenko's solutions to the $q$-KZB heat equation.  In \cite{Sun:qafftr}, the opposite coproduct to that of \cite{ESV} was considered in order to use the bosonization of \cite{Mat}, and the standard grading was used both to compare with the Felder-Varchenko function and to enable comparison with the trigonometric limit.  While the techniques used in this paper are similar to those sketched in \cite{ESV}, to ensure that the appropriate modifications to the corresponding $q$-difference equations are made, the author has chosen to write complete proofs for these modified versions.

In the remainder of this introduction, we state our results more precisely and describe the organization of the paper. For the reader's convenience, all notations will be redefined in later sections.

\subsection{Normalized trace functions for $U_q(\agg)$}

Throughout this introduction, we adopt the conventions for the quantum affine algebra $U_q(\agg)$ from Section \ref{sec:qa} and for dynamical notation from Section \ref{sec:dyn}.  Let $V_1, \ldots, V_n$ be finite dimensional $U_q(\agg)$-representations, and let $M_{\mu, k, a}$ be the Verma module with highest weight $\mu + k\Lambda_0 + a \delta$.  By the results of \cite{EFK}, there is a unique universal intertwiner
\[
\Phi_{\mu, k, a}^{V_1, \ldots, V_n}(z_1, \ldots, z_n): M_{\mu, k, a} \to M_{\mu - \tau_1 - \cdots - \tau_n, k - k_1' - \cdots - k_n', a} \hotimes V_1[z_1^{\pm 1}] \hotimes \cdots \hotimes V_n[z_n^{\pm 1}] \otimes V_n^* \otimes \cdots \otimes V_1^*.
\]
In these terms, the universal trace function is defined by 
\[
\Psi^{V_1,\ldots, V_n}(z_1, \ldots, z_n; \lambda, \omega, \mu, k) := \Tr|_{M_{\mu, k}}\Big(\Phi^{V_1, \ldots, V_n}_{\mu, k}(z_1, \ldots, z_n)q^{2\lambda + 2 \omega d}\Big).
\]
Our main object of study will be the following normalized version of the universal trace function, where the fusion operator $\JJ$, its collapsed version $\QQ$, and the Weyl denominator $\delta_q$ are defined in Sections \ref{sec:fus-op}, \ref{sec:fus-col}, and \ref{sec:wd-sec}.  We define the normalized trace function for $U_q(\agg)$ by 
\begin{multline*}
F^{V_1, \ldots, V_n}(z_1, \ldots, z_n; \lambda, \omega, \mu, k) := \QQ_{V_n^*}((\mu, k) - h^{(* 1 \cdots * n)})^{-1} \cdots \QQ_{V_1^*}((\mu, k) - h^{(*1)})^{-1} \\
 \JJ^{1 \cdots n}_{V_1[z_1^{\pm 1}] \otimes \cdots \otimes V_n[z_n^{\pm 1}]}(z_1, \ldots, z_n; \lambda, \omega)^{-1} \Psi^{V_1, \ldots, V_n}(z_1, \ldots, z_n; \lambda, \omega, \mu - \rho, k - h) \delta_q(\lambda, \omega).
\end{multline*}

\subsection{Macdonald-Ruijsenaars and dual Macdonald-Ruijsenaars equations}

Our first class of main results concerns affine analogues of the Macdonald-Ruijsenaars and dual Macdonald-Ruijsenaars equations.  These equations state that certain infinite difference operators are diagonalized on the normalized trace functions.  Let $W$ be an integrable lowest weight $U_q(\wgg)$-module of non-positive integer level $k_W$, and let $V_1, \ldots, V_n$ be finite-dimensional $U_q(\agg)$-modules.  Let also $\RR$ denote the exchange operator defined in Section \ref{sec:ex-op-def}.  Define the difference operator 
\[
\DD_W(\omega, k) := \sum_{\nu \in \hh^*, a \in \CC} \Tr|_{W[\nu + a \delta + k_W \Lambda_0]}\Big(\RR_{WV_1}(1, z_1; (\lambda, \omega) - h^{(2\cdots n)}) \cdots \RR_{W V_n}(1, z_n; \lambda, \omega)\Big) q^{-2ka} T^{\lambda, \omega}_{\nu, k_W},
\]
where $T^{\lambda, \omega}_{\nu, k_W} f(\lambda, \omega) = f(\lambda - \nu, \omega - k_W)$.  Define also the dual operator
\[
\DD_W^\vee(\omega, k) = \sum_{\nu \in \hh^*, a \in \CC} \Tr|_{W[\nu + a \delta + k_W \Lambda_0]}\Big(\RR_{WV_n^*}(1, z_n; (\mu, k) - h^{(* 1 \cdots * (n-1))}) \cdots \RR_{WV_1^*}(1, z_1; \mu, k)\Big) q^{-2\omega a} T^{\mu, k}_{\nu, k_W},
\]
where $T^{\mu, k}_{\nu, k_W}f(\mu, k) = f(\mu - \nu, k - k_W)$.  The Macdonald-Ruijsenaars and dual Macdonald-Ruijsenaars equations state that $\DD_W$ and $\DD_W^\vee$ are diagonalized on the normalized trace functions.

{\renewcommand{\thetheorem}{\ref{thm:mr}}
\begin{theorem}[Macdonald-Ruijsenaars equation] 
For any integrable lowest weight representation $W$ of non-positive integer level $\eta$, we have
\[
\DD_W(\omega, k) F^{V_1, \ldots, V_n}(z_1, \ldots, z_n; \lambda, \omega, \mu, k) = \chi_W(q^{-2\mu - 2kd}) F^{V_1, \ldots, V_n}(z_1, \ldots, z_n; \lambda, \omega, \mu, k),
\]
where $\chi_W$ is the character of $W$.
\end{theorem}
\addtocounter{theorem}{-1}}

{\renewcommand{\thetheorem}{\ref{thm:dmr}}
\begin{theorem}[dual Macdonald-Ruijsenaars equation] 
For any integrable lowest weight representation $W$ of non-positive integer level $k_W$, we have
\[
\DD_W^\vee(\omega, k) F^{V_1, \ldots, V_n}(z_1, \ldots, z_n; \lambda, \omega, \mu, k) = \chi_W(q^{-2\lambda - 2 \omega d}) F^{V_1, \ldots, V_n}(z_1, \ldots, z_n; \lambda, \omega, \mu, k),
\]
where $\chi_W$ is the character of $W$.
\end{theorem}
\addtocounter{theorem}{-1}}

Define the function $F_\star^{V_n^*, \ldots, V_1^*}$ to be the result of interchanging $V_i$ and $V_{n + 1 - i}^*$ in the definition of $F^{V_1, \ldots, V_n}$.  We may also deduce from Theorems \ref{thm:mr} and \ref{thm:dmr} the following symmetry relation.

{\renewcommand{\thetheorem}{\ref{thm:mrs}}
\begin{theorem}[Macdonald symmetry identity]
The functions $F^{V_1, \ldots, V_n}$ and $F_\star^{V_n^*, \ldots, V_1^*}$ satisfy the symmetry relation
\[
F^{V_1, \ldots, V_n}(z_1, \ldots, z_n; \lambda, \omega, \mu, k) = F_\star^{V_n^*, \ldots, V_1^*}(z_n, \ldots, z_1; \mu, k, \lambda, \omega).
\]
\end{theorem}
\addtocounter{theorem}{-1}}

\subsection{$q$-KZB and dual $q$-KZB equations}

Our second class of main results concerns affine analogues of the $q$-KZ equations known as the $q$-KZB equations introduced by Felder in \cite{Fel} and studied by Felder-Tarasov-Varchenko in \cite{FTV, FTV2}.  These equations state that shifts in the spectral parameters by the modular parameters are implemented by certain difference operators in $(\lambda, \omega)$ and $(\mu, k)$.  More precisely, the $q$-KZB operators and dual $q$-KZB operators are defined by 
\begin{align*}
K_j(z_1, \ldots, z_n;& \lambda, \omega, k) := \RR_{V_{j+1}V_j}(z_{j+1}, q^{2k} z_j; (\lambda, \omega) - h^{((j + 2) \cdots n)})^{-1} \cdots \RR_{V_nV_j}(z_n, q^{2k}z_j; \lambda, \omega)^{-1} \Gamma_j\\
&\phantom{=} \RR_{V_jV_1}(z_j, z_1; (\lambda, \omega) - h^{(2 \cdots (j-1))} - h^{((j+1) \cdots n)}) \cdots \RR_{V_j V_{j-1}}(z_j, z_{j-1}; (\lambda, \omega) - h^{((j + 1) \cdots n)})\\
&\phantom{=}D_j(\mu) := q_{*j}^{-2\mu + \sum_i x_i^2} q^{\Omega_{*j, *(j-1)}} \cdots q^{\Omega_{*j*1}},
\end{align*}
where $\Gamma_j f(\lambda, \omega) = f\Big((\lambda, \omega) + h^{(j)}\Big)$ and
\begin{align*}
K_j^\vee(&z_1, \ldots, z_n; \mu, k, \omega) := \RR_{V_{j-1}^*, V_j^*}(z_{j-1}, q^{2\omega} z_j; (\mu, k) - h^{(*1 \cdots *(j-2))})^{-1} \cdots \RR_{V_1^*, V_j^*}(z_1, q^{2\omega}z_j; \mu, k)^{-1} \Gamma_{*j} \\
&\phantom{=} \RR_{V_j^* V_n^*}(z_j, z_n; (\mu, k) - h^{(*1 \cdots *(j-1))} - h^{(*(j+1) \cdots *(n-1))}) \cdots \RR_{V_j^* V_{j+1}^*}(z_j, z_{j+1}; (\mu, k) - h^{(*1  \cdots *(j-1))})\\
&\phantom{======!}D_j^\vee(\lambda) := q_j^{-2\lambda + \sum_i x_i^2} q^{\Omega_{j, j + 1}} \cdots q^{\Omega_{j, n}},
\end{align*}
where $\Gamma_{*j} f(\mu, k) = f\Big((\mu, k) + h^{(*j)}\Big)$.  We recall here that $\RR$ is the exchange operator defined in Section \ref{sec:ex-op-def}.

Note that $K_j(z_1, \ldots, z_n; \lambda, \omega, k)$ and $K_j^\vee(z_1, \ldots, z_n; \mu, k, \omega)$ are difference operators in $(\lambda, \omega)$ and $(\mu, k)$ whose coefficients are linear operators on $V$ and $V^*$ and that $D_j(\mu)$ and $D_j^\vee(\lambda)$ are linear operators on $V^*$ and $V$.  The $q$-KZB and dual $q$-KZB equations relate the actions of $K_j$ and $D_j$ and $K_j^\vee$ and $D_j^\vee$ on the normalized trace functions.

{\renewcommand{\thetheorem}{\ref{thm:qkzb}}
\begin{theorem}[$q$-KZB equation]
For $j = 1, \ldots, n$, we have
\begin{multline*}
F^{V_1, \ldots, V_n}(z_1, \ldots, q^{2k}z_j,\ldots, z_n; \lambda, \omega, \mu, k)\\ = \Big(K_j(z_1, \ldots, z_n;\lambda, \omega, k) \otimes D_j(\mu)\Big) F^{V_1, \ldots, V_n}(z_1, \ldots, z_j, \ldots, z_n; \lambda, \omega, \mu, k).
\end{multline*}
\end{theorem}
\addtocounter{theorem}{-1}}

{\renewcommand{\thetheorem}{\ref{thm:dual-qkzb}}
\begin{theorem}[dual $q$-KZB equation]
For $j = 1, \ldots, n$, we have
\begin{multline*}
F^{V_1, \ldots, V_n}(z_1, \ldots, q^{2\omega}z_j,\ldots, z_n; \lambda, \omega, \mu, k)\\ = \Big(D_j^\vee(\lambda) \otimes K_j^\vee(z_1, \ldots, z_n; \mu, k, \omega)\Big) F^{V_1, \ldots, V_n}(z_1, \ldots, z_j, \ldots, z_n; \lambda, \omega, \mu, k).
\end{multline*}
\end{theorem}
\addtocounter{theorem}{-1}}

\subsection{Organization of the paper}

The remainder of the paper is organized as follows.  In Section \ref{sec:qa}, we fix our conventions for the quantum affine algebra $U_q(\wgg)$ and its universal $\cR$-matrix.  In Section \ref{sec:inter}, we define and characterize intertwiners for $U_q(\wgg)$.  In Section \ref{sec:fus-ex}, we define fusion and exchange operators in the quantum affine setting and relate their universal versions to their evaluations in representations.  In Section \ref{sec:tr}, we define the normalized trace function.  In Section \ref{sec:mr}, we prove the Macdonald-Ruijsenaars equations (Theorem \ref{thm:mr}) by computing difference operators originating as the radial part of certain central elements for $U_q(\agg)$.  In Section \ref{sec:dmr}, we prove the dual Macdonald-Ruijsenaars equations (Theorem \ref{thm:dmr}) by computing an $U_q(\agg)$-intertwiner in two different ways.  In Section \ref{sec:sym}, we use these two equations to prove a symmetry identity (Theorem \ref{thm:mrs}) for renormalized trace functions.  In Section \ref{sec:kzb}, we prove the dual $q$-KZB equation (Theorem \ref{thm:dual-qkzb}) and use the symmetry identity to deduce the $q$-KZB equation (Theorem \ref{thm:qkzb}).

\subsection{Acknowledgments}

The author thanks P. Etingof for suggesting the writing of this paper and for many helpful discussions.  Y.~S. was partially supported by a NSF Graduate Research Fellowship (NSF Grant \#1122374) and a Junior Fellow award from the Simons Foundation.

\section{Quantum affine algebras and $\cR$-matrices} \label{sec:qa}

In this section we fix our conventions on the quantum affine algebra $U_q(\agg)$ and its central extension $U_q(\wgg)$.  We recall also Drinfeld's construction of a central element in a completion of $U_q(\wgg)$.

\subsection{Cartan subalgebras}

Let $\gg$ be a simple Lie algebra, $\agg$ its affinization, and $\wgg$ the central extension.  Let $\alpha_i$, $i = 1, \ldots, r$ be the simple roots, $r$ the rank of $\gg$, $\theta$ the highest root, $\rho = \frac{1}{2} \sum_{\alpha > 0} \alpha$, and $\ch = 1 + (\theta, \rho)$ the dual Coxeter number. Let $A = (a_{ij})_{i, j = 0}^r$ be the extended Cartan matrix of $\agg$ and $d_i$ relatively prime positive integers so that $(d_i a_{ij})$ is symmetric.  Let the Cartan and dual Cartan algebras be 
\[
\wh = \hh \oplus \CC c \oplus \CC d \text{ and } \wh^* = \hh^* \oplus \CC \Lambda_0 \oplus \CC \delta,
\]
with $\Lambda_0 = c^*$ and $\delta = d^*$. Take $\alpha_0 := \delta - \theta \in \wh$.  The algebra $\wgg$ admits a non-degenerate invariant form $(-,-)$ whose restriction to $\wh$ has non-trivial values given by
\[
(d, d) = 0 \qquad (c, d) = 1 \qquad (\alpha_i, \alpha_i) = 2 \text{ for $i > 0$}
\]
and agrees with the standard non-degenerate form on $\hh$. Fix an orthonormal basis $\{x_i\}$ of $\hh$ under $(, )$. Define $\wrho := \rho + \ch \Lambda_0$.

\subsection{Quantum affine algebra}

Let $q$ be a non-zero complex number with $|q| < 1$.  The quantum affine algebra $U_q(\agg)$ is the Hopf algebra generated as an algebra by $e_i, f_i, q^{\pm h_i}$ for $0 \leq i \leq r$ with relations
\begin{align*}
[q^{h_i}, q^{h_j}] &= 0 \qquad q^{h_i} e_j q^{-h_j} = q^{(h_i, \alpha_j)} e_j \qquad q^{h_i} f_j q^{-h_j} = q^{-(h_i, \alpha_j)} f_j \qquad [e_i, f_j] = \delta_{ij} \frac{q^{d_i h_i} - q^{-d_i h_i}}{q - q^{-1}} \\
\sum_{k = 0}^{1 - a_{ij}}& (-1)^k \binom{1 - a_{ij}}{k}_{q^{d_i}} e_i^{1 - a_{ij} - k} e_j e_i^k = 0 \qquad \sum_{k = 0}^{1 - a_{ij}} (-1)^k \binom{1 - a_{ij}}{k}_{q^{d_i}} f_i^{1 - a_{ij} - k} f_j f_i^k = 0,
\end{align*}
where we use the notations $[n] = \frac{q^n - q^{-n}}{q - q^{-1}}$, $[n]! = [n] \cdots [1]$, and $\binom{a}{b}_q = \frac{[a]_q!}{[b]_q! [a - b]_q!}$.  The coproduct of $U_q(\agg)$ is
\begin{align*}
\Delta(e_i) &= e_i \otimes 1 + q^{d_i h_i} \otimes e_i \qquad \Delta(f_i) = f_i \otimes q^{-d_i h_i} + 1 \otimes f_i \qquad 
\Delta(q^{h_i}) = q^{h_i} \otimes q^{h_i}.
\end{align*}
The antipode of $U_q(\agg)$ is 
\[
S(e_i) = -  q^{-d_i h_i}e_i \qquad S(f_i) = - f_iq^{d_i h_i} \qquad S(q^{h_i}) = q^{-h_i}
\]
and the counit is 
\[
\eps(e_i) = \eps(f_i) = 0 \qquad \eps(q^h) = 1.
\]
In what follows, we will often use the Sweedler notation
\[
\Delta(x) = \sum_{(x)} x^{(1)} \otimes x^{(2)}.
\]
When the context is clear, we will omit the summation sign, writing $\Delta(x) = x^{(1)} \otimes x^{(2)}$ to denote an implicit summation over the pure tensor summands of the coproduct.

\begin{remark}
This coproduct is the opposite of the one in \cite{FR, ESV} but agrees with those in the bosonizations of \cite{KQS, Kon, Mat}.  Our motivation for using it is to produce results compatible with the results of \cite{Sun:qafftr}, which uses the bosonization and coproduct of \cite{Mat}.
\end{remark}

Let the Hopf-subalgebras $U_q(\abb_+)$ and $U_q(\abb_-)$ of $U_q(\agg)$ be generated by $\{e_i, q^{\pm h_i}\}$ and $\{f_i, q^{\pm h_i}\}$, respectively.  We centrally extend $U_q(\agg)$ to $U_q(\wgg)$ by adding a generator $q^d$ which commutes with $q^{h_i}$ and interacts with $e_i$ and $f_i$ via
\[
[q^d, e_i] = [q^d, f_i] = 0 \text{ for $i \neq 0$} \qquad q^de_0 q^{-d} = q e_0 \qquad q^df_0q^{-d} = q^{-1} f_0
\]
and on which the coproduct, antipode, and counit are
\[
\Delta(q^d) = q^d \otimes q^d \qquad S(q^d) = q^{-d} \qquad \eps(q^d) = 1.
\]
For $z \in \CC^\times$, define the automorphism $D_z := \Ad(z^d) \in \Aut(U_q(\agg))$.  The action of $d$ gives a grading on $U_q(\wgg)$.  In \cite[Section 5]{Dri},the square of the antipode is shown to act by conjugation by an explicit Cartan element.

\begin{lemma}[{\cite[Section 5]{Dri}}] \label{lem:anti-square}
For any $x \in U_q(\wgg)$, we have $S^2(x) = q^{-2\wrho} x q^{2\wrho}$.
\end{lemma}

\subsection{Universal $\cR$-matrix and Drinfeld element for $U_q(\wgg)$}

Define $\Omega = c\otimes d + d \otimes c + \sum_i x_i \otimes x_i$, and let $\Omega^0 = c \otimes d + d \otimes c$ and $\Omega^1 = \sum_i x_i \otimes x_i$.  By the general construction of \cite{Dri}, there is a universal $\cR$-matrix $\cR$ for $U_q(\wgg)$ with
\[
\cR = q^{-\Omega} \cR_0 \text{ and } \cR_0 \in (1 + U_q(\abb_+)_{>0} \hotimes U_q(\abb_-)_{< 0}).
\]
It satisfies 
\[
\Delta^{21}(x) \cR = \cR \Delta(x) \text{ for $x \in U_q(\wgg)$} \qquad (\Delta \otimes 1) \cR = \cR^{13} \cR^{23} \qquad (1 \otimes \Delta) \cR = \cR^{13} \cR^{12}.
\]

\begin{lemma}[{\cite[Proposition 3.1]{Dri}}] \label{lem:r-mat-s}
The universal $\cR$-matrix satisfies:
\begin{itemize}
\item[(a)] $(S \otimes 1) \cR = (1 \otimes S^{-1})\cR = \cR^{-1}$;

\item[(b)] $(S \otimes S) \cR = \cR$.
\end{itemize}
\end{lemma}

\begin{remark}
Note that the sign of $\Omega$ is reversed from that of \cite{ESV} because we use the opposite coproduct.
\end{remark}

In \cite[Section 5]{Dri}, the Drinfeld element $u$ in a completion of $U_q(\wgg)$ is defined by
\[
u = m_{21}\Big((1 \otimes S) \cR\Big)
\]
and shown to satisfy the following properties.

\begin{lemma}[{\cite[Section 5]{Dri}}] \label{lem:drinfeld}
The Drinfeld element $u$ satisfies the following:
\begin{itemize}
\item [(a)] $u^{-1} = m_{21}\Big((1 \otimes S^{-1}) \cR^{-1}\Big)$

\item [(b)] $S^2(x) = u x u^{-1}$;

\item [(c)] $q^{2\wrho} u = u q^{2\wrho}$ and $S(u) q^{-2\wrho} = q^{-2\wrho}S(u)$ are central in different completions of $U_q(\wgg)$;
\end{itemize}
\end{lemma}

\section{Representations and intertwiners for $U_q(\agg)$ and $U_q(\wgg)$} \label{sec:inter}

In this section we fix conventions for Verma and evaluation modules for $U_q(\agg)$ and $U_q(\wgg)$ and characterize intertwiners between them, which will be the central object of study of this paper.  Throughout the paper, we work with evaluation representations valued in the ring of Laurent polynomials or formal Laurent series.

\subsection{Evaluation modules}

If $V$ is a $U_q(\agg)$-module, for $z \neq 0$ we define the $U_q(\agg)$-module $V(z)$ to be the vector space $V$ with action of $U_q(\agg)$ given by 
\[
\pi_{V(z)}(a) = \pi_V(D_z(a)).
\]
In type A, if $V$ is a $U_q(\sl_r)$-module treated as a $U_q(\asl_r)$-module via evaluation at $1$, then $V(z)$ is the evaluation module at $z$. 

If $V$ is a finite-dimensional or highest weight $U_q(\agg)$-module, define the $U_q(\wgg)$-modules $z^{-\Delta}V[z, z^{-1}]$ and $z^{-\Delta}V((z))$ so that generators of $U_q(\agg)$ act in the same way as on $V(z)$ and $d$ acts by $z \frac{\partial}{\partial z}$ if $V$ is finite-dimensional and by $z \frac{\partial}{\partial z} + d$ if $V$ is highest weight.  Notice that both $z^{-\Delta}V[z, z^{-1}]$ and $z^{-\Delta}V((z))$ are infinite-dimensional as vector spaces over $\CC$.

\subsection{Verma modules}

We denote by $M_{\mu, k}$ the Verma module for $U_q(\agg)$ with highest weight $\mu + k \Lambda_0$ and by $v_{\mu, k} \in M_{\mu, k}$ a canonically chosen highest weight vector.  Define the restricted dual of $M_{\mu, k}$ by
\[
M_{\mu, k}^\vee := \bigoplus_{\tau, a} M_{\mu, k}[\tau + k \Lambda_0 - a \delta]^*,
\]
where the action of $U_q(\agg)$ is given by $(u \cdot \phi)(m) := \phi(S(u) m)$. Define the representation $M_{\mu, k, a}$ to coincide with $M_{\mu, k}$ as a $U_q(\agg)$-representation, but with $U_q(\wgg)$-action given by letting $q^d$ act by $q^a$ on $v_{\mu, k}$.  If $a = 0$, we write $M_{\mu, k}$ for $M_{\mu, k, 0}$. Our convention is chosen to be consistent with the following explicit computation of the action of the Drinfeld element on $M_{\mu, k}$.

\begin{lemma}[{\cite[Section 5]{Dri}}] \label{lem:drinfeld-verma}
The action of $q^{2\wrho} u$ on $M_{\mu, k}$ is given by $q^{(\mu + k \Lambda_0, \mu + k \Lambda_0 + 2 \wrho)}$.
\end{lemma}

For generic $(\mu, k)$, the Verma module $M_{\mu, k, a}$ is irreducible.  On the other hand, if $\mu\rho + k \Lambda_0$ is dominant integral, then the singular vectors in $M_{\mu, k, a}$ may be determined explicitly.
\begin{prop} \label{prop:verma-sing-vect}
Suppose $\mu + k \Lambda_0$ is dominant integral.  For a reduced decomposition $s_{i_1} \cdots s_{i_l}$ of $w \in \wtilde{W}$, define $\alpha^l = \alpha_{i_l}$, $\alpha^j = (s_{i_l} \cdot s_{i_{j+1}})(\alpha_{i_j})$, and $n_j = 2 \frac{(\mu + k \Lambda_0 + \wrho, \alpha^j)}{(\alpha^j, \alpha^j)}$. Then the vectors
\[
v^{w}_{\mu, k} := \frac{f_{i_1}^{n_{i_1}}}{[n_{i_1}]!} \cdots \frac{f_{i_l}^{n_{i_l}}}{[n_{i_l}]!} v_{\mu, k}
\]
are the only singular vectors in $M_{\mu, k, a}$. 
\end{prop}

\begin{remark}
We give any highest weight $U_q(\agg)$-module $M$ the structure of a $U_q(\wgg)$-module by imposing that $q^d$ acts by $1$ on the highest weight vector.  This convention differs from that of \cite{FR} but is consistent with our notation for $M_{\mu, k}$ above.
\end{remark}

\subsection{$\cR$-matrices and intertwiners between representations}

We will use the following intertwining property of the universal $\cR$-matrix on evaluation representations of $U_q(\wgg)$.  Let $V_1, \ldots, V_n$ be finite-dimensional $U_q(\agg)$-representations, and let $W$ be a $U_q(\ahh)$-semisimple $U_q(\wgg)$-representation.  Define the tensor products $V$ and $\hV$ of evaluation representations by
\[
V := V_1[z_1^{\pm 1}] \otimes \cdots \otimes V_n[z_n^{\pm 1}] \qquad \text{ and } \qquad \hV := V_1((z_1) \otimes \cdots \otimes V_n((z_n)).
\]

\begin{lemma} \label{lem:r-inter}
The operator $P_{VW}\cR_{VW}$ gives an intertwiner
\[
P_{\hV W}\cR_{VW}: V_1[z_1^{\pm 1}] \otimes \cdots \otimes V_n[z_n^{\pm 1}] \otimes W \to W \otimes V_1((z_1)) \otimes \cdots \otimes V_n((z_n)).
\]
\end{lemma}
\begin{proof}
Because $\cR \in q^{-\Omega}\Big(1 + U_q(\abb_+) \hotimes U_q(\abb_-)\Big)$, $\cR_{VW}$ defines a linear map $V \otimes W \to \hV \otimes W$.  The composed map $P_{\hV W}\cR_{VW}$ is then an intertwiner of $U_q(\wgg)$-representations by the property $\Delta^{21}(x) \cR = \cR \Delta(x)$ for $x \in U_q(\wgg)$.
\end{proof}

\subsection{Intertwiners of $U_q(\wgg)$-representations}

For any $U_q(\wgg)$-module $W$ which is $U_q(\ahh)$-semisimple, define the completed tensor product by
\[
M_{\mu, k, a} \hotimes W := \Hom_{\CC}(M_{\mu, k, a}^\vee, W),
\]
where the $U_q(\wgg)$-action is given by $(u \cdot \phi)(m) = u^{(1)} \phi(S(u^{(2)}) m)$. In these terms, elements of $M_{\mu, k, a} \hotimes W$ are sums $\sum_{i = 0}^\infty m_i \otimes w_i$ with $m_i, w_i$ homogeneous and $\lim_{i \to \infty} \deg(m_i) = \infty$. 

A key construction in this paper will be of intertwiners between a Verma module and its completed tensor product with either a finite-dimensional or integrable module.  We begin by characterizing the space of such intertwiners when the weight is either generic or dominant integral. Denote the highest term of an intertwiner $\Phi: M_{\mu_1, k_1, a_1} \to M_{\mu_2, k_2, a_2} \otimes W$ by
\[
\langle \Phi \rangle := \langle v_{\mu_2, k_2}^*, \Phi v_{\mu_1, k_1}\rangle.
\]

\begin{prop} \label{prop:int-exist}
Let $M_{\lambda, k, a}$ and $M_{\mu, k - k', a}$ be Verma modules and $W$ a $U_q(\wgg)$-representation of level $k'$ on which $q^{h_i}$ and $q^d$ act diagonally.  Suppose that either (1) $(\mu, k)$ is generic or (2) $\mu + (k - k') \Lambda_0$ is dominant integral and $W[\lambda - \mu + n_i \alpha_i + k'\Lambda_0] = 0$ for all $i > 0$. We have an isomorphism
\[
\Hom_{U_q(\wgg)}(M_{\lambda, k, a}, M_{\mu, k - k', a} \hotimes W) \simeq W[\lambda - \mu + k' \Lambda_0]
\]
given by $\Phi \mapsto \langle \Phi \rangle$.
\end{prop}
\begin{proof}
The space of $U_q(\agg)$-intertwiners $M_{\lambda, k, a} \to M_{\mu, k - k', a} \hotimes W$ is given by
\begin{align*}
\Hom_{U_q(\agg)}(M_{\lambda, k}, M_{\mu, k - k'} \hotimes W) &\simeq \Hom_{U_q(\agg)}(\Ind^{U_q(\agg)}_{U_q(\abb)} \CC_{\lambda, k}, \Hom_{\CC}(M_{\mu, k - k'}^\vee, W))\\
&\simeq \Hom_{U_q(\abb)}(\CC_{\lambda, k}, \Hom_{\CC}(M_{\mu, k - k'}^\vee, W))\\
&\simeq \Hom_{U_q(\abb)}(\CC_{\lambda, k} \otimes_k M_{\mu, k - k'}^\vee, W)\\
&\simeq \{v \in W[\lambda - \mu + k' \Lambda_0] \mid I_{\mu, k - k'}v = 0\},
\end{align*}
where $I_{\mu, k - k'} := \{u \in U_q(\agg) \mid u \cdot v_{-\mu, -k}^* = 0\}$ is the annihilator ideal of the lowest weight vector of $M_{\mu, k - k'}^\vee$.  By Proposition \ref{prop:verma-sing-vect}, the $U_q(\abb_+)$-submodule of $M_{\mu, k}^\vee$ generated by $v_{-\mu, -k + k'}^*$ has relations 
\[
e_i^{n_i} v_{-\mu, -k + k'}^* = 0
\]
so that $I_{\mu, k - k'}$ is generated by $e_i^{n_i}$.  The given condition ensures any element of $W[\lambda - \mu + k'\Lambda_0]$ yields a map of $U_q(\agg)$-modules in both cases (1) and (2).  Computing the action of $q^d$ on the source and target and recalling our convention that $q^d$ acts by $q^a$ on the highest-weight vector shows that a valid $v$ must lie in the degree $0$ part of $W$, completing the proof.
\end{proof}

\begin{remark}
In what follows, we will apply Proposition \ref{prop:int-exist} for representations $W$ which are either integrable modules or tensor products of evaluation modules associated to finite-dimensional $U_q(\agg)$-modules.
\end{remark}

For a $U_q(\ahh)$-semisimple representation $V$ of level $k'$ which is either highest weight or finite-dimensional, $v\in V[\tau]$, and $(\mu, k)$ so that the conditions of Proposition \ref{prop:int-exist} hold, denote by 
\[
\Phi_{\mu, k, a}^v(z): M_{\mu, k, a} \to M_{\mu - \tau, k - k', a} \hotimes V[z, z^{-1}]
\]
the unique corresponding $U_q(\wgg)$-intertwiner.  Each $z$-coefficient of $\Phi_{\mu, k, a}^{v}(z)$ lies in the tensor product $M_{\mu - \tau, k - k', a} \otimes V$ without completion.  Similarly, if $W$ is a $U_q(\ahh)$-semisimple $U_q(\wgg)$-representation of level $k_W$, and $w \in W[\tau + k_W \Lambda_0]$, denote by 
\[
\Phi_{\mu, k, a}^w: M_{\mu, k, a} \to M_{\mu - \tau, k - k_W, a} \hotimes W
\]
the corresponding intertwiner given by Proposition \ref{prop:int-exist}.  For notational convenience, we will sometimes denote this by $\Phi_{\mu, k, a}^w(1) := \Phi_{\mu, k, a}^w$.  For $i = 1, \ldots, n$, let $V_i$ be a $U_q(\agg)$-representation of level $k_i'$ which is either finite-dimensional or a highest weight $U_q(\wgg)$-representation.  For $v_i \in V_i[\tau_i + k_i' \Lambda_0]$, define the iterated intertwiner
\[
\Phi_{\mu, k, a}^{v_1, \ldots, v_n}(z_1, \ldots, z_n): M_{\mu, k, a} \to M_{\mu - \tau_1 - \cdots - \tau_n, k - k_1' - \cdots - k_n', a} \hotimes V_1[z_1^{\pm 1}] \hotimes \cdots \hotimes V_n[z_n^{\pm 1}]
\]
by the composition
\begin{equation} \label{eq:comp-int-def}
\Phi_{\mu, k, a}^{v_1, \ldots, v_n}(z_1, \ldots, z_n) = \Phi_{\mu - \tau_2 - \cdots - \tau_n, k - k_2' - \cdots - k_n', a}^{v_1}(z_1) \circ \cdots \circ \Phi_{\mu, k, a}^{v_n}(z_n),
\end{equation}
where we adopt the convention that $z_i \equiv 1$ and $V_i[z_i^{\pm 1}] \equiv V_i$ if $V_i$ is a highest weight $U_q(\wgg)$-representation.  If all $V_i$ are finite-dimensional, define also the universal intertwiner
\begin{equation} \label{eq:uni-int-def}
\Phi^{V_1, \ldots, V_n}_{\mu, k, a}(z_1, \ldots, z_n) := \sum_{v_1, \ldots, v_n} \Phi^{v_1, \ldots, v_n}_{\mu, k, a}(z_1, \ldots, z_n) \otimes v_n^* \otimes \cdots \otimes v_1^*,
\end{equation}
where in the sum $\{v_i\}$ and $\{v_i^*\}$ range over dual bases of $V_1, \ldots, V_n$ and $V_1^*, \ldots, V_n^*$.  Finally, denote the single-step intertwiner associated to $v_1 \otimes \cdots \otimes v_n \in V_1[z_1^{\pm 1}] \otimes \cdots \otimes V_n[z_n^{\pm 1}]$ by Proposition \ref{prop:int-exist} by
\begin{equation} \label{eq:big-int-def}
\wPhi^{v_1 \otimes \cdots \otimes v_n}_{\mu, k, a}(z_1, \ldots, z_n): M_{\mu, k, a} \to M_{\mu - \tau_1 - \cdots - \tau_n, k - k_1' - \cdots - k_n', a} \hotimes V_1[z_1^{\pm 1}] \otimes \cdots \otimes V_n[z_n^{\pm 1}],
\end{equation}
with the same convention that $z_i \equiv 1$ and $V_i[z_i^{\pm 1}] \equiv V_i$ if $V_i$ is a highest weight $U_q(\wgg)$-module.  If all $V_i$ are finite-dimensional, let its universal version be 
\begin{equation} \label{eq:uni-big-int-def}
\wPhi^{V_1, \ldots, V_n}_{\mu, k, a}(z_1, \ldots, z_n):= \sum_{v_1, \ldots, v_n} \wPhi^{v_1 \otimes \cdots \otimes v_n}_{\mu, k, a}(z_1, \ldots, z_n) \otimes v_n^* \otimes \cdots \otimes v_1^*.
\end{equation}
Notice that $\Phi^{v}_{\mu, k, a}(z) = \wPhi^v_{\mu, k, a}(z)$ and $\Phi^V_{\mu, k, a}(z) = \wPhi^V_{\mu, k, a}(z)$.  As maps of $U_q(\agg)$-modules, these intertwiners are independent of $a$; we denote by $\Phi^{v_1, \ldots, v_n}_{\mu, k}(z_1, \ldots, z_n)$, $\Phi^{V_1, \ldots, V_n}_{\mu, k}(z_1, \ldots, z_n)$, $\wPhi^{v_1, \ldots, v_n}_{\mu, k}(z_1, \ldots, z_n)$, and $\wPhi^{V_1, \ldots, V_n}_{\mu, k}(z_1, \ldots, z_n)$ the corresponding intertwiners with $a = 0$.

\begin{remark}
As defined, the composed intertwiners $\Phi_{\mu, k, a}^{v_1, \ldots, v_n}(z_1, \ldots, z_n)$ are formal series in $z_1, \ldots, z_n$.  It was shown in \cite{EFK, FR} that the matrix elements of these formal series converge for $z_1 \mgg \cdots \mgg z_n$ and admit extension to meromorphic functions. 
\end{remark}

\begin{remark}
For finite-dimensional $V_1, \ldots, V_n$, the intertwiners $\Phi_{\mu, k}^{v_1,\ldots, v_n}(z_1, \ldots, z_n)$ appeared in \cite{EFK, FR} with Weyl modules in the place of Verma modules.  The two constructions coincide for generic weight.
\end{remark}

\section{Fusion and exchange operators and ABRR equation} \label{sec:fus-ex}

In this section, we introduce the fusion and exchange operators as operators on representations via intertwiners and as universal elements via the ABRR equation.  After characterizing their basic properties, we give normalizations of the parameters for which evaluation of the universal operators in tensor products of representations coincides with the construction via intertwining operators.

\subsection{Dynamical notation} \label{sec:dyn}

Throughout the following sections, we will use dynamical notation to express the action of certain operators on tensor products of $U_q(\agg)$-modules.  Suppose that $f(\mu, k): \ahh^* \to U_q(\wgg)$ is a function and $V_1 \otimes \cdots V_n \otimes W_1^* \otimes W_m^*$ a tensor product of $U_q(\agg)$-representations.  We denote by $f((\mu, k) + a h^{(j)} + b h^{(*l)})$ the element of $\End(V_1 \otimes \cdots V_n \otimes W_1^* \otimes W_m^*)$ acting by
\[
f((\mu, k) + a h^{(j)} + b h^{(*l)})(v_1 \otimes \cdots \otimes v_n \otimes w_1^* \otimes \cdots \otimes w_m^*) := f((\mu, k) + a \mu_j + b\nu_l)(v_1 \otimes \cdots \otimes v_n \otimes w_1^* \otimes \cdots \otimes w_m^*)
\]
for $v_j \in V_j[\mu_j]$ and $w_l^* \in W_l^*[\nu_l]$, where we use $(\mu, k)$ to denote the element $\mu + k \Lambda_0 \in \ahh^*$. If $\mu_j, \nu_l \in \hh^*$, we will also denote this by the notation $f(\mu + a h^{(j)} + b h^{(*l)}, k)$. We denote by $f((\mu, k) + a \whh^{(j)})$ the element which acts by
\[
f_\tau((\mu, k) + a \whh^{(j)})(v_1 \otimes \cdots \otimes v_n \otimes w_1^* \otimes \cdots \otimes w_m^*) := f_\tau((\mu, k) + a \mu_j + a\tau)(v_1 \otimes \cdots \otimes v_n \otimes w_1^* \otimes \cdots \otimes w_m^*),
\]
where $f_\tau$ is the part of $f$ which shifts the weight in $V_j$ by $\tau$.

\subsection{Fusion and exchange operators in representations}

Let $V_1$ and $V_2$ be $U_q(\agg)$-representations of level $k_1$ and $k_2$ which are either finite-dimensional or highest weight $U_q(\wgg)$-representations.  As before, if $V_i$ is a highest weight $U_q(\wgg)$-representation, let $z_i \equiv 1$ and interpret $V_i[z_i^{\pm 1}] \equiv V_i$.  Suppose also that at most one of $V_i$ is highest weight.  The fusion operator $J_{V_1, V_2}(z_1, z_2; \mu, k): V_1[z_1^{\pm 1}]\otimes V_2[z_2^{\pm 2}] \to V_1[z_1^{\pm 1}] \otimes V_2[z_2^{\pm 1}]$ is defined by
\[
J_{V_1, V_2}(z_1, z_2; \mu, k)(v_1 \otimes v_2) := \Big\langle\Phi^{v_1, v_2}_{\mu, k}(z_1, z_2) v_{\mu, k}, v_{(\mu, k) - \wt(v_2) - \wt(v_1)}^*\Big\rangle
\]
on homogeneous $v_1 \otimes v_2 \in V_1 \otimes V_2$, where $v_{(\mu, k) - \wt(v_2) - \wt(v_1)}^*$ is the dual vector to the highest weight vector $v_{(\mu, k) - \wt(v_2) - \wt(v_1)}$ for $M_{(\mu, k) - \wt(v_2) - \wt(v_1)}$, and $\wt(v_i)$ denotes the $\ahh^*$-weight of $v_i$.  As defined, it is a formal series in $z_1/z_2$, and it was shown in \cite{EFK} that it converges to a meromorphic function on $z_1 \mgg z_2$ if $V_1$ and $V_2$ are finite-dimensional.  

For $U_q(\agg)$ representations $V_i$ of level $k_i$ which are either finite-dimensional or highest weight with at most one highest weight, define the iterated fusion operator by 
\[
J_{V_1, \ldots, V_n}(z_1, \ldots, z_n; \mu, k)(v_1 \otimes \cdots \otimes v_n)  := \Big \langle \Phi^{v_1, \ldots, v_n}_{\mu, k}(z_1, \ldots, z_n) v_{\mu, k}, v_{(\mu, k) - \wt(v_1) - \cdots - \wt(v_n)}^*\Big\rangle. 
\]
Define also its multicomponent version by 
\begin{multline*}
\wJ_{V_1, \ldots, V_i; V_{i+1}, \ldots, V_n}(z_1, \ldots, z_i; z_{i+1}, \ldots, z_n; \mu, k)(v_1 \otimes v_2 \otimes \cdots \otimes v_n)\\ :=\Big\langle \wPhi^{v_1 \otimes \cdots \otimes v_i}_{(\mu, k) - \wt(v_{i+1}) - \cdots - \wt(v_n)}(z_1, \ldots, z_i) \wPhi^{v_{i+1} \otimes \cdots \otimes v_n}_{\mu, k}(z_{i+1}, \ldots, z_n) v_{\mu, k}, v_{(\mu, k) - \wt(v_1) - \cdots - \wt(v_n)}^*\Big\rangle,
\end{multline*}
where we note that $\wJ_{V_1; V_2}(z_1; z_2; \mu, k) = J_{V_1, V_2}(z_1, z_2; \mu, k)$.  We may relate the two types of intertwiners via the multicomponent fusion operator.

\begin{lemma} \label{lem:fus-def}
For homogeneous vectors $v_1, \ldots, v_n$, we have that 
\[
\wPhi^{J_{V_1, \ldots, V_n}(z_1, \ldots, z_n; \mu, k)(v_1 \otimes  \cdots \otimes v_n)}_{\mu, k}(z_1, \ldots, z_n) = \Phi^{v_1, \ldots, v_n}_{\mu, k}(z_1, \ldots, z_n).
\]
\end{lemma}
\begin{proof}
The two intertwiners both have highest term
\[
J_{V_1, \ldots, V_n}(z_1, \ldots, z_n; \mu, k)(v_1 \otimes \cdots \otimes v_n),
\]
hence they coincide by Proposition \ref{prop:int-exist}.
\end{proof}

\begin{lemma} \label{lem:iter-fus}
The iterated fusion operator satisfies
\begin{multline*}
J_{V_1, \ldots, V_n}(z_1, \ldots, z_n; \mu, k)\\
 = \wJ_{V_1; V_2 \otimes \cdots \otimes V_n}(z_1; z_2, \ldots, z_n; \mu, k) \wJ_{V_2; V_3 \otimes \cdots \otimes V_n}(z_2; z_3, \ldots, z_n; \mu, k) \cdots \wJ_{V_{n - 1}; V_n}(z_{n - 1}; z_n; \mu, k).
\end{multline*}
\end{lemma}
\begin{proof}
These are two different ways of expressing the highest term of the intertwiner $\Phi^{v_1, \ldots, v_n}_{\mu, k}(z_1, \ldots, z_n)$, hence they are equal.
\end{proof}

Let $V_1, V_2$ be $U_q(\agg)$-representations which are either finite-dimensional or highest weight $U_q(\wgg)$-representations, with at most one being highest weight.  The exchange operator $R_{V_1V_2}(z_1, z_2; \mu, k): V_1[z_1^{\pm 1}] \otimes V_2[z_2^{\pm 1}] \to V_1((z_1)) \otimes V_2((z_2^{-1}))$ is defined as
\[
R_{V_1V_2}(z_1, z_2; \mu, k) := J_{V_1V_2}(z_1, z_2; \mu, k)^{-1} \cR^{21}_{V_1 V_2} J^{21}_{V_1V_2}(z_2, z_1; \mu, k).
\]

\begin{remark}
Both the fusion and exchange operators depend only on the $U_q(\agg)$-structure of the Verma module $M_{\mu, k, a}$, meaning that their definition is independent of the choice of normalization for the grading.  Therefore, their value remains the same if $M_{\mu, k}$, $\Phi^{v_1, \ldots, v_n}_{\mu, k}(z_1, \ldots, z_n)$, and $\wPhi^{v_1, \ldots, v_n}_{\mu, k}(z_1, \ldots, z_n)$ are replaced by $M_{\mu, k, a}$, $\Phi^{v_1, \ldots, v_n}_{\mu, k, a}(z_1, \ldots, z_n)$, and $\wPhi^{v_1, \ldots, v_n}_{\mu, k, a}(z_1, \ldots, z_n)$ in their definitions.
\end{remark}

\subsection{Universal fusion operators} \label{sec:fus-op}

In \cite{ES}, a universal fusion operator $\cJ(\mu, k)$ living in a completion of $U_q(\wgg) \hotimes U_q(\wgg)$ under the principal grading is defined; when evaluated in finite dimensional representations, $\cJ(\mu, k)$ yields the previously defined fusion operators.  We modify this definition by using the ABRR equation for the opposite coproduct and using the standard grading instead of the principal grading.  For this, we define the coefficient ring
\[
\AA_{\mu, k} := \CC[[q^{-2(\mu, \alpha_1)}, \ldots, q^{-2(\mu, \alpha_r)}, q^{-2k + 2(\mu, \theta)}]]
\]
and work formally over $\AA_{\mu, k}$. We then have the following analogue of \cite[Theorem 8.1]{ES}.

\begin{prop} \label{prop:es-abrr}
There exists a unique element $\cJ(\mu, k) \in 1 + (U_q(\hat{\bb}_-)_{< 0} \hat{\otimes} U_q(\hat{\bb}_+)_{>0})^{\hat{\hh}} \otimes \AA_{\mu, k}$ satisfying the ABRR equation
\begin{equation} \label{eq:abrr}
\cR^{21}q_1^{2\mu + 2kd} \cJ(\mu, k) = \cJ(\mu, k) q_1^{2\mu + 2kd} q^{-\Omega}.
\end{equation}
Moreover, the universal fusion operator $\cJ(\mu, k)$ satisfies
\begin{equation} \label{eq:ufus-cocycle}
\cJ^{12, 3}(\mu, k) \cJ^{12}((\mu, k) - h^{(3)}/2) = \cJ^{1, 23}(\mu, k) \cJ^{23}((\mu, k) + h^{(1)}/2).
\end{equation}
\end{prop}
\begin{proof}
Write $\cJ(\mu, k) = \sum_{i, j \geq 0} \cJ_{i, j}(\mu) q^{-2ki}$, where $\cJ_{i, j}(\mu)$ consists of terms with degree $-j$ in the first tensor component with respect to the standard grading, and write $\cR = \sum_{l \geq 0} \cR_l$, where $\cR_l$ consists of terms of degree $l$ in the first tensor factor and $\cR_0 = \cR_\trig$.  The ABRR equation (\ref{eq:abrr}) may be rewritten as
\[
\sum_{i, j, l} \cR^{21}_l q_1^{2\mu} \cJ_{i, j}(\mu) q_1^{-2\mu} q^\Omega q^{-2k(i + j)} = \sum_{i, j} \cJ_{i, j}(\mu) q^{-2ki}.
\]
Matching degree $j$ coefficients of $q^{-2k}$, the ABRR equation is equivalent to 
\[
\cJ_{i, j}(\mu) = \sum_{a = 0}^{\min\{i, j\}} \cR^{21}_{j - a} q_1^{2\mu} \cJ_{i - a, a}(\mu) q_1^{-2\mu} q^\Omega.
\]
For $j = 0$, this yields 
\[
\cJ_{i, 0}(\mu) = \cR^{21}_\trig q_1^{2\mu} \cJ_{i, 0}(\mu) q_1^{-2\mu} q^\Omega,
\]
which is the ABRR equation for $U_q(\gg)$.  By the existence and uniqueness of solutions to the ABRR equation for $U_q(\gg)$ given by \cite[Proposition 1]{ABRR}, we find that $\cJ_{0, 0}(\mu) = \cJ_\trig(\mu)$ and $\cJ_{i, 0}(\mu) = 0$ for $i > 0$.  For $j > 0$, we have a recursion relation expressing $\cJ_{i, j}(\mu)$ in terms of elements with smaller $i$ or $j$, yielding existence and uniqueness.

To show (\ref{eq:ufus-cocycle}), define the operators
\begin{align*}
A_L X &= \cR^{21} \cR^{31} q_1^{2\mu} X q_1^{-2\mu} q^{\Omega_{12}} q^{\Omega_{13}}\\
A_R X &= \cR^{32} \cR^{31} q_3^{-2\mu} X q_3^{2\mu} q^{\Omega_{12}} q^{\Omega_{23}}.
\end{align*}
First, we claim that $A_L$ and $A_R$ commute.  This is equivalent to the identity
\[
\cR^{21} \cR^{31} q_1^{2\mu + 2kd} \cR^{32} \cR^{31} q_3^{-2\mu - 2kd} = \cR^{32}\cR^{31} q_3^{-2 \mu - 2kd} \cR^{21} \cR^{31} q_1^{2\mu + 2kd},
\]
which follows from the Yang-Baxter equation after canceling $\cR^{31}$ and the factors of $q_1^{2\mu + 2kd}$ and $q_3^{-2\mu - 2kd}$.

Now, we claim that both sides are solutions to $A_L X = X$ and $A_R X = X$ which agree in degree $0$ terms under the principal grading in the first and third components.  Because such solutions are unique, they must be equal.  Notice first that
\begin{align*}
A_L \cJ^{1, 23}(\mu, k) \cJ^{23}((\mu, k) + h^{(1)}/2)
&= \cR^{23, 1} \Ad(q_1^{2 \mu + 2kd}) \cJ^{1, 23}(\mu, k) \cJ^{23}((\mu, k) + h^{(1)}/2) q^{\Omega_{1, 23}}\\
&= \cJ^{1, 23}(\mu, k) q^{-\Omega_{1, 23}} \cJ^{23}((\mu, k) + h^{(1)}/2) q^{\Omega_{1, 23}}\\
&= \cJ^{1, 23}(\mu, k) \cJ^{23}((\mu, k) + h^{(1)}/2).
\end{align*}
Because $A_L$ and $A_R$ commute, $X = A_R \cJ^{1, 23}(\mu, k) \cJ^{23}((\mu, k) + h^{(1)}/2)$ satisfies $A_LX = X$, so to check that $X = \cJ^{1, 23}(\mu, k) \cJ^{23}((\mu, k) + h^{(1)}/2)$, it suffices to note that the degree zero term of the first component of $\cJ^{1, 23}(\mu, k) \cJ^{23}((\mu, k) + h^{(1)}/2)$ under the principal grading is $\cJ^{23}((\mu, k) + h^{(1)}/2)$, the degree zero term of the first component of $X$ under the principal grading is 
\begin{align*}
\cR^{32} q^{-\Omega_{31}} q_3^{-2\mu - 2kd} \cJ^{23}((\mu, k) + h^{(1)}/2) q_3^{2\mu + 2kd} q^{\Omega_{13}} q^{\Omega_{23}} &= \cR^{23} \Ad(q_3^{-2\mu - 2kd - h^{(1)}})\cJ^{23}((\mu, k) + h^{(1)}/2) q^{\Omega_{23}}\\
& = \cR^{23} \Ad(q_2^{2\mu + 2kd + h^{(1)}})\cJ^{23}((\mu, k) + h^{(1)}/2) q^{\Omega_{23}},
\end{align*}
and that they are equal by the ABRR equation.  The computation for $\cJ^{12, 3}(\mu, k) \cJ^{12}((\mu, k) - h^{(3)}/2)$ is similar.
\end{proof}
We require also a renormalization of $\cJ(\mu, k)$ which will have good convergence properties when evaluated on $U_q(\wgg)$-representations which are locally nilpotent with respect to the induced $U_q(\gg)$-action.  Define the renormalized universal fusion operator $\cL(\mu, k) \in 1 + (U_q(\abb_-)_{<0} \hotimes U_q(\abb_+)_{> 0})^{\ahh} \otimes \AA_{\mu, k}$ by 
\[
\cL(\mu, k) := (\cR^{21})^{-1} \cJ(\mu, k).
\]
This renormalized operator satisfies the following formal properties.

\begin{prop} \label{prop:renorm-abrr}
The following hold for $\cL(\mu, k)$:
\begin{itemize}
\item[(a)] the element $\cL(\mu, k)$ satisfies the ABRR equation
\begin{equation} \label{eq:renorm-abrr}
q_1^{2\mu + 2kd} \cR^{21} \cL(\mu, k) = \cL(\mu, k) q_1^{2\mu + 2kd} q^{-\Omega};
\end{equation}

\item[(b)] each coefficient of $\cL(\mu, k)$ as a power series in $q^{-2k}$ has finite degree in the standard grading;

\item[(c)] the element $\cL(\mu, k)$ satisfies the shifted $2$-cocycle relation
\begin{equation} \label{eq:renorm-cocycle}
\cL^{21, 3}(\mu, k) \cL^{12}((\mu, k) - h^{(3)}/2) = \cL^{1, 32}(\mu, k)\cL^{23}((\mu, k) + h^{(1)}/2).
\end{equation}
\end{itemize}
\end{prop}
\begin{proof}
Claim (a) follows directly from the ABRR equation (\ref{eq:abrr}) for $\cJ(\mu, k)$.  For (b), let the power series expansion of $\cL(\mu, k)$ be 
\[
\cL(\mu, k) = \sum_{i, j \geq 0} \cL_{i, j}(\mu) q^{-2ki},
\]
where $\cL_{i, j}$ has terms of degree $-j$ in the first tensor factor, and write a series expansion
\[
\cR = \sum_{l \geq 0} \cR_l,
\]
where $\cR_l$ consists of terms of degree $l$ in the first tensor factor and $\cR_0 = \cR_\text{trig}$, the universal $\cR$-matrix of $U_q(\gg)$.  The (renormalized) ABRR equation (\ref{eq:renorm-abrr}) yields
\begin{align*}
\sum_{i, j\geq 0} \cL_{i, j}(\mu) q^{-2ki} q^{-\Omega} &= \sum_{i, j, l \geq 0} q_1^{2\mu + 2kd} \cR^{21}_l \cL_{i, j}(\mu) q_1^{-2\mu - 2kd} q^{-2ki}\\
&= \sum_{i, j, l \geq 0} q_1^{2\mu} \cR^{21}_l \cL_{i, j}(\mu) q_1^{-2\mu} q^{-2k(i + j + l)}.
\end{align*}
This implies that the constant term $\cL_0(\mu)$ satisfies 
\[
\cL_0(\mu) q^{-\Omega} = q_1^{2\mu}\cR^{21}_{\trig} \cL_0(\mu) q_1^{-2\mu}
\]
and is therefore equal to $(\cR^{21}_\trig)^{-1} \cJ_\trig(\mu)$, the analogous quantity for $U_q(\gg)$.  For the higher terms, matching coefficients yields
\[
\sum_{j \geq 0} \cL_{i, j}(\mu) q^{-\Omega} = \sum_{a + j + l = i} q_1^{2\mu} \cR^{21}_l \cL_{a, j}(\mu) q_1^{-2\mu}
\]
and therefore that
\[
\sum_{j \geq 0} \cL_{i, j}(\mu) q^{-\Omega} - q_1^{2\mu} \cR^{21}_0 \cL_{i, 0}(\mu) q^{-2\mu} = \sum_{\substack{a + j + l = i\\ a < i}} q_1^{2\mu} \cR^{21}_l \cL_{a, j}(\mu) q_1^{-2\mu}.
\]
Induction on the power of $q^{-2k}$ yields (b).  For (c), the $2$-cocycle relation (\ref{eq:ufus-cocycle}) for the fusion operator implies that 
\[
\cJ^{12, 3}(\mu, k) \cR^{21} \cL^{12}((\mu, k) - h^{(3)}/2) = \cJ^{1, 23}(\mu, k) \cR^{32} \cL^{23}((\mu, k) + h^{(1)}/2).
\]
Noting the relations $\cJ^{12, 3}(\mu, k) \cR^{21} = \cR^{21} \cJ^{21, 3}(\mu, k)$ and $\cJ^{1, 23}(\mu, k) \cR^{32} = \cR^{32} \cJ^{1, 32}(\mu, k)$ transforms this into
\[
\cR^{21} \cR^{3, 21} \cL^{21, 3}(\mu, k) \cL^{12}((\mu, k) - h^{(3)}/2) = \cR^{32, 1} \cR^{32}\cJ^{1, 32}(\mu, k) \cL^{23}((\mu, k) + h^{(1)}/2).
\]
The result follows from noting that $\cR^{21} \cR^{3, 21} = \cR^{32, 1} \cR^{32}$ by the Yang-Baxter equation.
\end{proof}
\begin{remark}
A consequence of Proposition \ref{prop:renorm-abrr}(b) is that $\cL(\mu, k)$ may be evaluated on the tensor product of any two representations which are locally nilpotent with respect to the action of $U_q(\gg)$.
\end{remark}

Define the shifted universal fusion operators
\[
\JJ(\mu, k) := \cJ\Big((\mu, k) - h^{(1)}/2 - h^{(2)}/2\Big) \qquad \text{ and } \qquad \LL(\mu, k) := \cL\Big((\mu, k) - h^{(1)}/2 - h^{(2)}/2\Big).
\]
For finite-dimensional $U_q(\agg)$-representations $V_1, \ldots, V_n$, consider the corresponding evaluation representations
\[
V := V_1[z_1^{\pm 1}] \otimes \cdots \otimes V_n[z_n^{\pm 1}] \qquad \text{ and } \qquad \hV := V_1((z_1) \otimes \cdots \otimes V_n((z_n)).
\]
If $W$ is a highest weight $U_q(\wgg)$-representation, then evaluation of $\JJ(\mu, k)$ gives linear maps $V \otimes W \to V \otimes W$ and $W \otimes V \to W \otimes \hV$, which we denote by $\JJ_{VW}(z_1, \ldots, z_n; 1; \mu, k)$ and $\JJ_{WV}(1; z_1, \ldots, z_n; \mu, k)$.  Evaluation of $\LL(\mu, k)$ gives linear maps $V \otimes W \to V \otimes W$ and $W \otimes V \to W \otimes V$, which we denote by $\LL_{VW}(z_1, \ldots, z_n; 1; \mu, k)$ and $\LL_{WV}(1; z_1, \ldots, z_n; \mu, k)$.

\subsection{Universal exchange operators} \label{sec:ex-op-def}

Define the universal exchange operator in $U_q(\wgg) \hotimes U_q(\wgg)$ by
\begin{equation} \label{eq:def-ex}
\RR(\mu, k) := \JJ(\mu, k)^{-1} \cR^{21} \JJ^{21}(\mu, k)
\end{equation}
In terms of the renormalized universal fusion operator, we have 
\begin{equation} \label{eq:renorm-ex}
\RR(\mu, k) = \LL(\mu, k)^{-1} \cR \LL^{21}(\mu, k).
\end{equation}
Evaluation of $\RR(\mu, k)$ gives linear maps $V \otimes W \to \hV \otimes W$ and $W \otimes V \to W \otimes \hV$, which we denote by $\RR_{VW}(z_1, \ldots, z_n; 1; \mu, k)$ and $\RR_{WV}(1; z_1, \ldots, z_n; \mu, k)$. 

\begin{prop} \label{prop:qdybe}
The universal exchange operator satisfies the quantum dynamical Yang-Baxter equation 
\[
\RR^{23}(\mu, k) \RR^{13}((\mu, k) - h^{(2)}) \RR^{12}(\mu, k) = \RR^{12}((\mu, k) - h^{(3)})\RR^{13}(\mu, k)\RR^{23}((\mu, k) - h^{(1)}).
\]
\end{prop}
\begin{proof}
By Proposition \ref{prop:es-abrr}, we obtain that 
\begin{align*}
\JJ^{13}((\mu, k) - h^{(2)})^{-1} &= \JJ^{32}(\mu, k)^{-1} \JJ^{1, 32}(\mu, k)^{-1} \JJ^{13, 2}(\mu, k)\\
\JJ^{31}((\mu, k) - h^{(2)}) &= \JJ^{31, 2}(\mu, k)^{-1} \JJ^{3, 12}(\mu, k) \JJ^{12}(\mu, k).
\end{align*}
Substituting these relations into the definition of the universal exchange operator, we find that
\begin{align*}
\RR^{23}(\mu, k) &\RR^{13}((\mu, k) - h^{(2)}) \RR^{12}(\mu, k) \\
&= \JJ^{23}(\mu, k)^{-1} \cR^{32} \JJ^{32}(\mu, k) \JJ^{13}((\mu, k) - h^{(2)})^{-1} \cR^{31} \JJ^{31}((\mu, k) - h^{(2)}) \JJ^{12}(\mu, k)^{-1} \cR^{21} \JJ^{21}(\mu, k) \\
&= \JJ^{23}(\mu, k)^{-1} \cR^{32} \JJ^{1, 32}(\mu, k)^{-1} \JJ^{13, 2}(\mu, k) \cR^{31} \JJ^{31, 2}(\mu, k)^{-1} \JJ^{3, 12}(\mu, k) \cR^{21} \JJ^{21}(\mu, k) \\
&= \JJ^{23}(\mu, k)^{-1} \JJ^{1, 23}(\mu, k)^{-1} \cR^{32} \cR^{31} \cR^{21} \JJ^{3, 21}(\mu, k) \JJ^{21}(\mu, k)\\
&= \JJ^{23}(\mu, k)^{-1} \JJ^{1, 23}(\mu, k)^{-1} \cR^{21} \cR^{31} \cR^{32} \JJ^{3, 21}(\mu, k) \JJ^{21}(\mu, k).
\end{align*}
On the other hand, using the relations
\begin{align*}
\JJ^{12}((\mu, k) - h^{(3)})^{-1} &= \JJ^{23}(\mu, k)^{-1} \JJ^{1, 23}(\mu, k)^{-1} \JJ^{12, 3}(\mu, k)\\
\JJ^{21}((\mu, k) - h^{(3)}) &= \JJ^{21, 3}(\mu, k)^{-1} \JJ^{2, 13}(\mu, k) \JJ^{13}(\mu, k)\\
\JJ^{23}((\mu, k) - h^{(1)})^{-1} &= \JJ^{31}(\mu, k)^{-1} \JJ^{2, 31}(\mu, k)^{-1} \JJ^{23, 1}(\mu, k)\\
\JJ^{32}((\mu, k) - h^{(1)}) &= \JJ^{32, 1}(\mu, k)^{-1} \JJ^{3, 21}(\mu, k) \JJ^{21}(\mu, k)
\end{align*}
from Proposition \ref{prop:es-abrr}, we obtain
\begin{align*}
&\RR^{12}((\mu, k) - h^{(3)})\RR^{13}(\mu, k)\RR^{23}((\mu, k) - h^{(1)}) \\
&= \JJ^{12}((\mu, k) - h^{(3)})^{-1} \cR^{21} \JJ^{21}((\mu, k) - h^{(3)}) \JJ^{13}(\mu, k)^{-1} \cR^{31} \JJ^{31}(\mu, k) \JJ^{23}((\mu, k) - h^{(1)})^{-1} \cR^{32} \JJ^{32}((\mu, k) - h^{(1)})\\
&= \JJ^{23}(\mu, k)^{-1} \JJ^{1, 23}(\mu, k)^{-1} \JJ^{12, 3}(\mu, k) \cR^{21} \JJ^{21, 3}(\mu, k)^{-1}\\
&\phantom{====} \JJ^{2, 13}(\mu, k) \cR^{31} \JJ^{2, 31}(\mu, k)^{-1} \JJ^{23, 1}(\mu, k) \cR^{32} \JJ^{32, 1}(\mu, k)^{-1} \JJ^{3, 21}(\mu, k) \JJ^{21}(\mu, k)\\
&= \JJ^{23}(\mu, k)^{-1} \JJ^{1, 23}(\mu, k)^{-1} \cR^{21} \cR^{31} \cR^{32} \JJ^{3, 21}(\mu, k) \JJ^{21}(\mu, k),
\end{align*}
which yields the desired.
\end{proof}

\subsection{Evaluation of universal fusion operators}

Let $V_1, \ldots, V_n$ be $U_q(\agg)$-representations which are either finite-dimensional or highest weight $U_q(\wgg)$-representations.  Let $z_i$ be a variable if $V_i$ is finite-dimensional or $1$ otherwise, and let $V_i[z_i^{\pm 1}]$ be an evaluation representation of $U_q(\wgg)$ if $V_i$ is finite-dimensional or $V_i$ itself otherwise.  We now relate the multicomponent fusion operators $\wJ_{V_1, \ldots, V_i; V_{i+1}, \ldots, V_n}(z_1, \ldots, z_i; z_{i+1}, \ldots, z_n; \mu, k)$ to the shifted universal fusion operator $\JJ(\mu + \rho, k + \ch)$ evaluated in a representation.  Define the representations
\[
W_1 = V_1[z_1^{\pm 1}] \otimes \cdots \otimes V_i[z_i^{\pm 1}] \qquad \text{ and } \qquad W_2 = V_{i+1}[z_{i+1}^{\pm 1}] \otimes \cdots \otimes V_n[z_n^{\pm 1}],
\]
and denote by $\JJ_{W_1 W_2}(z_1, \ldots, z_i; z_{i+1}, \ldots, z_n; \mu + \rho, k + \ch)$ the evaluation of $\JJ(\mu, k)$ on $W_1 \otimes W_2$.

\begin{prop} \label{prop:fus-eval}
Suppose that at most one $V_i$ is a highest weight $U_q(\wgg)$-representation.  We have that
\[
\JJ_{W_1 W_2}(z_1, \ldots, z_i; z_{i+1}, \ldots, z_n; \mu + \rho, k + \ch) = \wJ_{V_1, \ldots, V_i; V_{i+1}, \ldots, V_n}(z_1, \ldots, z_i; z_{i+1}, \ldots, z_n; \mu, k).
\] 
\end{prop}
\begin{proof}
By Lemmas \ref{lem:drinfeld} and \ref{lem:drinfeld-verma}, the element $C = e^{2\wrho} u$ is central and satisfies $C|_{M_{\lambda, k}} = q^{(\lambda + k\Lambda_0, \lambda + k\Lambda_0 + 2\wrho)} \cdot \id$.  Consider a series expansion
\[
\cR = \sum_i a_i \otimes b_i
\]
so that $u = \sum_i S(b_i) a_i$.  Choose homogeneous vectors $w_1 \in W_1$ and $w_2 \in W_2$, and define the quantity
\begin{equation} \label{eq:f-def}
L = \Big\langle v_{(\mu, k) - \wt(w_1) - \wt(w_2)}^* , \wPhi^{w_1}_{(\mu, k) - \wt(w_1)}(z_1, \ldots, z_i) \circ u|_{M_{(\mu, k) - \wt(w_2)}} \circ \wPhi^{w_2}_{\mu, k}(z_{i+1}, \ldots, z_n) v_{\mu, k} \Big\rangle,
\end{equation}
where $\wt(w_i)$ denotes the weight of $w_i$. We compute $L$ in two different ways.  First, computing the action of $C$ on $M_{\mu, k}$, we find for weight vectors $v \in V$ and $w \in W$ that  
\begin{align*}
L & = q^{(\mu + k \Lambda_0 - \wt(w_2), \mu + k \Lambda_0 - \wt(w_2) + 2 \wrho)}\\
&\phantom{=========} \Big\langle v_{(\mu, k) - \wt(w_1) - \wt(w_2)}^* , \wPhi^{w_1}_{(\mu, k) - \wt(w_1)}(z_1, \ldots, z_i) \circ q^{2\wrho}|_{W_2} \circ \wPhi^{w_2}_{\mu, k}(z_{i+1}, \ldots, z_n) q^{-2\wrho}v_{\mu, k} \Big\rangle\\
&= q^{(\mu + k \Lambda_0 - \wt(w_2), \mu + k \Lambda_0 - \wt(w_2) + 2 \wrho) - 2(\mu + k \Lambda_0, \wrho)} q^{2\wrho}|_{W_2}\\
&\phantom{=========} \wJ_{V_1, \ldots, V_i; V_{i+1}, \ldots, V_n}(z_1, \ldots, z_i; z_{i+1}, \ldots, z_n; \mu, k)(w_1 \otimes w_2)\\
&= q^{(\mu + k \Lambda_0 - \wt(w_2), \mu + k \Lambda_0 - \wt(w_2)) - 2(\wt(w_2), \wrho)} q^{2\wrho}|_{W_2}\\
&\phantom{=========} \wJ_{V_1, \ldots, V_i; V_{i+1}, \ldots, V_n}(z_1, \ldots, z_i; z_{i+1}, \ldots, z_n; \mu, k)(w_1 \otimes w_2).
\end{align*}
For the second way, we have
\[
\Delta_2(S_2(\cR)) = (1 \otimes S \otimes S) \Delta_2^{21}(\cR) = (1 \otimes S \otimes S)(\cR^{12}\cR^{13}),
\]
which implies that 
\[
\sum_i a_i \otimes S(b_i)_{(1)} \otimes S(b_i)_{(2)} = \sum_{i, j} a_ia_j \otimes S(b_i) \otimes S(b_j).
\]
We conclude that 
\begin{align*}
L &= \sum_i \Big\langle v_{(\mu, k) - \wt(w_1) - \wt(w_2)}^* , \wPhi^{w_1}_{(\mu, k) - \wt(w_1)}(z_1, \ldots, z_i) \circ S(b_i) a_i \circ \wPhi^{w_2}_{\mu, k}(z_{i+1}, \ldots, z_n) v_{\mu, k} \Big\rangle\\
&= \sum_{i, j} \Big\langle v_{(\mu, k) - \wt(w_1) - \wt(w_2)}^* , (S(b_i) \otimes S(b_j) \otimes 1) \wPhi^{w_1}_{(\mu, k) - \wt(w_1)}(z_1, \ldots, z_i) \circ a_ia_j \circ \wPhi^{w_2}_{\mu, k}(z_{i+1}, \ldots, z_n) v_{\mu, k} \Big\rangle.
\end{align*}
Recall now that $\cR \in q^{-\Omega} (1 + U_q(\hat{\bb}_+)_{> 0} \hat{\otimes} U_q(\hat{\bb}_-)_{>0})$, meaning that all $i$-indexed terms aside from the ones corresponding to $q^{-\Omega}$ are zero in the sum.  This implies that
\begin{align*}
L &= \sum_{j} S(b_j)|_{W_1} \Big\langle v_{(\mu, k) - \wt(w_1) - \wt(v_2)}^*, \\
&\phantom{========} ((m_{31} \otimes 1)(\Delta_1 \otimes S)(q^{-\Omega})) \wPhi^{w_1}_{(\mu, k) - \wt(w_2)}(z_1, \ldots, z_i) a_j\wPhi^{w_2}_{\mu, k}(z_{i+1}, \ldots, z_n) v_{\mu, k} \Big\rangle\\
&= q^{(\mu - \wt(w_1) - \wt(w_2))^2} \sum_{j} (S(b_j) q^{\mu - \wt(w_1) - \wt(w_2)})|_{W_1}\\
&\phantom{========} \Big\langle v_{(\mu, k) - \wt(w_1) - \wt(w_2)}^*, \wPhi^{w_1}_{(\mu, k) - \wt(w_1)}(z_1, \ldots, z_i)a_j\wPhi^{w_2}_{\mu, k}(z_{i+1}, \ldots, z_n) v_{\mu, k} \Big\rangle,
\end{align*}
where we note that 
\[
(m_{31} \otimes 1) \circ \Big((\Delta \otimes S)q^{-\Omega}\Big) = q_1^{\sum_i x_i^2 + 2 cd} q^{\Omega}.
\]
Observe now that 
\[
S_3(\Delta_1(\cR)) = S_3(\cR^{13} \cR^{23}) = \sum_{i, j} a_i \otimes a_j \otimes S(b_j) S(b_i) = \Big(\sum_j 1 \otimes a_j \otimes S(b_j)\Big)\Big(\sum_i a_i \otimes 1 \otimes S(b_i)\Big),
\]
which we may rearrange to obtain
\[
\sum_i a_i \otimes 1 \otimes S(b_i) = \Big(\sum_j 1 \otimes a_j \otimes S(b_j)\Big)^{-1} S_3(\Delta_1(\cR))).
\]
This implies that 
\begin{align*}
&\sum_{j} (S(b_j) q^{\mu - \wt(w_1) - \wt(w_2)})|_{W_1} \Big\langle v_{(\mu, k) - \wt(w_1) - \wt(w_2)}^*, \wPhi^{w_1}_{(\mu, k) - \wt(w_2)}(z_1, \ldots, z_i) a_j\wPhi^{w_2}_{\mu, k}(z_{i+1}, \ldots, z_n) v_{\mu, k} \Big\rangle\\
&= \Big(\sum_j S(b_j) \otimes a_j\Big)^{-1} \sum_j S(b_j)|_{W_1} q_{W_1}^{\mu - \wt(w_1) - \wt(w_2)} \\
&\phantom{===========} \left\langle v_{(\mu, k) - \wt(w_1) - \wt(w_2)}^*, \wPhi^{w_1}_{(\mu, k) - \wt(w_2)}(z_1, \ldots, z_i) \wPhi^{w_2}_{\mu, k}(z_{i + 1}, \ldots, z_n)a_jv_{\mu, k}\right\rangle\\
&= \Big(\sum_j S(b_j) \otimes a_j\Big)^{-1} q_{W_1}^{2\mu - \wt(w_1) - \wt(w_2)}\\
&\phantom{===========} \left\langle v_{(\mu, k) - \wt(w_1) - \wt(w_2)}^*, \wPhi^{w_1}_{(\mu, k) - \wt(w_2)}(z_1, \ldots, z_i) \wPhi^{w_2}_{\mu, k}(z_{i+1}, \ldots, z_n)v_{\mu, k}\right\rangle\\
&= \Big(\sum_j S(b_j) \otimes a_j \Big)^{-1} q_{W_1}^{2\mu - \wt(w_1) - \wt(w_2)} \wJ_{V_1, \ldots, V_i; V_{i+1}, \ldots, V_n}(z_1, \ldots, z_i; z_{i+1}, \ldots, z_n; \mu, k)(w_1 \otimes w_2),
\end{align*}
where in the second equality we notice that all terms in the sum over $j$ aside from those of $q^{-\Omega}$ vanish.  Now by Lemma \ref{lem:r-mat-s} we have that
\[
S_2(\cR) = q_2^{-2\wrho} \cR^{-1} q_2^{2\wrho},
\]
meaning that 
\[
\Big(\sum_j S(b_j) \otimes a_j \Big)^{-1} = q_{W_1}^{-2\wrho} \cR^{21}_{W_1 W_2} q_{W_1}^{2\wrho}.
\]
We conclude that 
\begin{multline*}
L = q^{(\mu - \wt(w_1) - \wt(w_2))^2} q_{W_1}^{-2\wrho} \cR^{21}_{W_1 W_2} q_{W_1}^{2\mu + 2\wrho - \wt(w_1) - \wt(w_2)}\\ \wJ_{V_1, \ldots, V_i; V_{i+1}, \ldots, V_n}(z_1, \ldots, z_i; z_{i+1}, \ldots, z_n; \mu, k) (w_1 \otimes w_2).
\end{multline*}
Combined with our previous expression, we find that 
\begin{multline*}
\cR^{21}_{W_1W_2} q_{W_1}^{2\mu + 2\wrho - \wt(w_1) - \wt(w_2)} \wJ_{V_1, \ldots, V_i; V_{i+1}, \ldots, V_n}(z_1, \ldots, z_i; z_{i+1}, \ldots, z_n; \mu, k) (w_1 \otimes w_2)\\
= q^{(\mu + k \Lambda_0 - \wt(w_2), \mu + k \Lambda_0 - \wt(w_2)) - 2(\wt(w_2), \wrho) - (\mu - \wt(w_2) - \wt(w_1))^2} q_{W_1}^{2\wrho} q^{2\wrho}_{W_2}\\ \wJ_{V_1, \ldots, V_i; V_{i+1}, \ldots, V_n}(z_1, \ldots, z_i; z_{i+1}, \ldots, z_n; \mu, k)(w_1 \otimes w_2)\\
= \wJ_{V_1, \ldots, V_i; V_{i+1}, \ldots, V_n}(z_1, \ldots, z_i; z_{i+1}, \ldots, z_n; \mu, k) q^{2\mu + 2 \wrho - \wt(w_1) - \wt(w_2)}_{W_1} q^{-\Omega} (w_1 \otimes w_2).
\end{multline*}
This coincides with the evaluation of the ABRR equation 
\[
\cR^{21}_{W_1 W_2} q_{W_1}^{2\mu + 2\wrho - \wt(w_1) - \wt(w_2)}\JJ(\mu + \rho, k + \ch) = \JJ(\mu + \rho, k + \ch) q_{W_1}^{2\mu + 2\wrho - \wt(w_1) - \wt(w_2)} q^{-\Omega}
\]
for $\JJ(\mu + \rho, k + \ch)$ in $W_1 \otimes W_2$.  Since the solution of the ABRR equation in $\End(W_1 \otimes W_2)$ is unique, we conclude the desired equality
\[
\wJ_{V_1, \ldots, V_i; V_{i+1}, \ldots, V_n}(z_1, \ldots, z_i; z_{i+1}, \ldots, z_n; \mu, k) = \JJ_{W_1W_2}(z_1, \ldots, z_i; z_{i+1}, \ldots, z_n; \mu + \rho, k + \ch). \qedhere
\]
\end{proof}

\begin{corr} \label{corr:fus-factor}
We have that
\begin{multline*}
J_{V_1, \ldots, V_n}(z_1, \ldots, z_n; \mu, k)\\
 = \JJ_{V_1[z_1^{\pm 1}], V_2[z_2^{\pm 1}] \otimes \cdots \otimes V_n[z_n^{\pm 1}]}(z_1; z_2, \ldots, z_n; \mu + \rho, k + \ch) \cdots \JJ_{V_{n - 1}[z_{n-1}^{\pm 1}], V_n[z_n^{\pm 1}]}(z_{n-1}; z_n; \mu + \rho, k + \ch).
\end{multline*}
\end{corr}
\begin{proof}
This follows by repeatedly applying Proposition \ref{prop:fus-eval} and Lemma \ref{lem:iter-fus}.
\end{proof}

\subsection{Adjoints of fusion and exchange operators}

In what follows, we will often consider adjoints of fusion and exchange operators evaluated in representations.  For vector spaces $W_1, \ldots, W_m$ and an operator 
\[
T \in \End(W_1) \otimes \End(W_m) \otimes \CC((z_2/z_1, \ldots, z_n/z_{n-1})),
\]
we denote by 
\begin{align*}
T^{*{W_i}} &\in \End(W_1) \otimes \cdots \otimes \End(W_i^*) \otimes \cdots \End(W_m) \otimes \CC((z_2/z_1, \ldots, z_n/z_{n-1}))\\
T^{*{W_i} *{W_j}} &\in \End(W_1) \otimes \cdots \otimes \End(W_i^*) \otimes \cdots \otimes \End(W_j^*) \otimes \cdots \End(W_m) \otimes \CC((z_1/z_2, \ldots, z_{n-1}/z_n))
\end{align*}
the operators with adjoints taken in the corresponding vector spaces.  We will use the notation $T^* := T^{*{W_1} \cdots *{W_m}}$ for the case where adjoints are taken in all vector spaces.  For example, the adjoint of the iterated fusion operator is denoted by 
\[
J_{V_1, \ldots, V_n}(z_1, \ldots, z_n; \mu, k)^* \in \End(V_1^*) \otimes \cdots \otimes \End(V_n^*) \otimes \CC((z_1/z_2, \ldots, z_{n-1}/z_n))
\]
and the adjoint in $V$ of the universal exchange operator evaluated in $W \otimes V$ is denoted by 
\[
\RR_{WV}(1; z_1, \ldots, z_n; \mu, k)^{*V} \in \End(W) \otimes \End(V^*) \otimes \CC((z_1/z_2,\ldots, z_{n_1}/z_n)).
\]

\section{Normalized trace functions} \label{sec:tr}

In this section, we define the trace function for $U_q(\agg)$ and give it a normalization under which it will satisfy four systems of $q$-difference equations.  We note that the parameter shifts in the normalization differ from those of \cite{ESV} due to our different choice of coproduct.

\subsection{Unnormalized traces}

Let $V_1, \ldots, V_n$ be finite dimensional $U_q(\agg)$-representations.  For $v_i \in V_i[\tau_i]$, define the trace function $\Psi^{v_1, \ldots, v_n}(z_1, \ldots, z_n; \lambda, \omega, \mu, k)$ by
\[
\Psi^{v_1,\ldots, v_n}(z_1, \ldots, z_n; \lambda, \omega, \mu, k) := \Tr|_{M_{\mu, k}}\Big(\Phi^{v_1, \ldots, v_n}_{\mu, k}(z_1, \ldots, z_n)q^{2\lambda + 2 \omega d}\Big).
\]
Define the universal trace function $\Psi^{V_1,\ldots, V_n}(z_1, \ldots, z_n; \lambda, \omega, \mu, k)$ with values in 
\[
q^{2(\lambda, \mu)} \AA_{\mu, k} \otimes \AA_{\lambda, \omega} \otimes \CC((z_1/z_2, \ldots, z_{n-1}/z_n)) \otimes (V_1 \otimes \cdots \otimes V_n)[0] \otimes (V_n^* \otimes \cdots \otimes V_1^*)[0]
\]
by the expression
\[
\Psi^{V_1, \ldots, V_n}(z_1, \ldots, z_n; \lambda, \omega, \mu, k) := \sum_{v_i \in V_i} \Psi^{v_1,\ldots, v_n}(z_1, \ldots, z_n; \lambda, \omega, \mu, k) \otimes (v_n^* \otimes \cdots \otimes v_1^*),
\]
where $v_i$ and $v_i^*$ range over dual bases of $V_i$.  Define also the single-step versions
\[
\wPsi^{v_1 \otimes \cdots \otimes v_n}(z_1, \ldots, z_n; \lambda, \omega, \mu, k) := \Tr|_{M_{\mu, k}}\Big(\wPhi^{v_1 \otimes \cdots \otimes v_n}_{\mu, k}(z_1, \ldots, z_n)q^{2\lambda + 2 \omega d}\Big)
\]
and
\[
\wPsi^{V_1, \ldots, V_n}(z_1, \ldots, z_n; \lambda, \omega, \mu, k) := \Tr|_{M_{\mu, k}}\Big(\wPhi^{V_1, \ldots, V_n}_{\mu, k}(z_1, \ldots, z_n)q^{2\lambda + 2 \omega d}\Big).
\]
The two versions of the universal trace function are related by the adjoint of the iterated fusion operator.

\begin{lemma} \label{lem:trace-adj}
We have that 
\[
\Psi^{V_1, \ldots, V_n}(z_1, \ldots, z_n; \lambda, \omega, \mu, k) = J_{V_1, \ldots, V_n}(z_1, \ldots, z_n; \mu, k)^* \wPsi^{V_1, \ldots, V_n}(z_1, \ldots, z_n; \lambda, \omega, \mu, k).
\]
\end{lemma}
\begin{proof}
This is a consequence of Lemma \ref{lem:fus-def}.
\end{proof}

\subsection{Normalization factors} \label{sec:fus-col}

To obtain interesting difference equations on traces, it will be convenient to introduce the normalization factor 
\[
\QQ(\mu, k) := m_{21}\Big(S_2(\LL(\mu, k))\Big).
\]
Consider the series expansions
\begin{multline} \label{eq:ser-exp}
\cR = \sum_i a_i \otimes b_i, \qquad \cR^{-1} = \sum_i a_i' \otimes b_i', \qquad \cJ(\mu, k) = \sum_i c_i \otimes d_i(\mu, k)\\
\cJ(\mu, k)^{-1} = \sum_i c_i' \otimes d_i'(\mu, k), \qquad \cL(\mu, k) = \sum_i e_i \otimes f_i(\mu, k), \qquad \cL(\mu, k)^{-1} = \sum_i e_i' \otimes f_i'(\mu, k).
\end{multline}
Notice that 
\[
\QQ(\mu, k) = \sum_i S(f_i(\mu, k)) e_i \qquad \text{ and } \qquad S(\QQ(\mu, k)) = \sum_i S(e_i) S^2(f_i(\mu, k)).
\]
Recall by Proposition \ref{prop:renorm-abrr} that $\LL(\mu, k)$ is a power series in $q^{-2k}$ whose coefficients have finite degree in the standard grading, which implies that $\QQ(\mu, k)$ may be evaluated in any representation which is locally nilpotent with respect to the action of $U_q(\gg)$.

\begin{remark}
As noted in \cite[Remark 2]{ESV}, the definition of $\QQ(\mu, k)$ differs from that used in \cite{EV, ES} and is related to the definition of \cite{ES} by viewing $\QQ(\mu, k)$ as a renormalization of the product $S(u)^{-1} \QQ^{\text{\cite{ES}}}(\mu, k)$ which involves divergent sums.
\end{remark}

\subsection{Weyl denominator} \label{sec:wd-sec}

Define the Weyl denominator $\delta_q(\lambda, \omega)$ by 
\[
\delta_q(\lambda, \omega) := \Tr|_{M_{-\wrho}}(q^{2 \lambda + 2 \omega d})^{-1}.
\]
In \cite[Lemma 3.3]{ES3}, an explicit product formula is given for $\delta_q(\lambda, \omega)$.  In our notation, it is given by 
\begin{equation} \label{eq:delta-val}
\delta_q(\lambda, \omega) = e^{(\wrho, 2 \lambda + 2 \omega d)} \prod_{\alpha > 0} (1 - e^{-(\alpha, 2\lambda + 2 \omega d)})^{-1},
\end{equation}
where the product is over positive roots $\alpha$.

\subsection{Combined fusion operator}

Define the combined fusion operator by 
\[
\JJ^{1 \cdots n}(\mu, k) := \JJ^{1, 2\cdots n}(\mu, k) \cdots \JJ^{n - 1, n}(\mu, k)
\]
and denote its evaluation in $V = V_1[z_1^{\pm 1}] \otimes \cdots \otimes V_n[z_n^{\pm 1}]$ by $\JJ^{1 \cdots n}_V(z_1, \ldots, z_n; \mu, k)$.  These combined fusion operators commute with exchange operators in the following way.

\begin{lemma} \label{lem:c-fuse}
We have that 
\[
\RR^{0, 1\cdots n}(\mu, k)\JJ^{1 \cdots n}((\mu, k) - h^{(0)}) = \JJ^{1 \cdots n}(\mu, k) \RR^{01}((\mu, k) - h^{(2 \cdots n)}) \cdots \RR^{0n}(\mu, k).
\]
\end{lemma}
\begin{proof}
We induct on $n$.  The base case $n = 1$ is trivial.  For the inductive step, consider $V_{n - 1} \otimes V_n$ as one representation to obtain by the inductive hypothesis that
\begin{align*}
\RR^{0, 1\cdots n}&(\mu, k)\JJ^{1 \cdots n}((\mu, k) - h^{(0)})\\
 &= \RR^{0, 1 \cdots n}(\mu, k) \JJ^{1, 2\cdots n}((\mu, k) - h^{(0)}) \cdots \JJ^{n - 2, (n - 1) \cdots n}((\mu, k) - h^{(0)}) \JJ^{n - 1, n}((\mu, k) - h^{(0)})\\
&= \JJ^{1, 2\cdots n}(\mu, k) \cdots \JJ^{n - 2, (n - 1) \cdots n}(\mu, k) \RR^{01}((\mu, k) - h^{(2 \cdots n)}) \cdots \RR^{0, n-1,n}(\mu, k)\JJ^{n - 1, n}((\mu, k) - h^{(0)}).
\end{align*}
Notice now that 
\begin{align*}
\RR^{0, 12}(\mu, k)\JJ^{12}((\mu, k) - h^{(0)}) &= \JJ^{0, 12}(\mu, k)^{-1} \cR^{12, 0} \JJ^{12, 0}(\mu, k) \JJ^{12}((\mu, k) - h^{(0)})\\
&= \JJ^{0, 12}(\mu, k)^{-1} \cR^{12, 0} \JJ^{1, 20}(\mu, k) \JJ^{20}(\mu, k)\\
&= \JJ^{12}(\mu, k) \JJ^{01}((\mu, k) - h^{(2)})^{-1} \JJ^{01, 2}(\mu, k)^{-1} \cR^{10} \cR^{20} \JJ^{1, 20}(\mu, k) \JJ^{20}(\mu, k)\\
&= \JJ^{12}(\mu, k) \JJ^{01}((\mu, k) - h^{(2)})^{-1}\cR^{10} \JJ^{10, 2}(\mu, k)^{-1} \JJ^{1, 02}(\mu, k) \cR^{20}  \JJ^{20}(\mu, k)\\
&= \JJ^{12}(\mu, k) \JJ^{01}((\mu, k) - h^{(2)})^{-1}\cR^{10} \JJ^{10}((\mu, k) - h^{(2)}) \JJ^{02}(\mu, k)^{-1} \cR^{20}  \JJ^{20}(\mu, k).
\end{align*}
We conclude that 
\[
\RR^{0, 1\cdots n}(\mu, k)\JJ^{1 \cdots n}((\mu, k) - h^{(0)}) = \JJ^{1 \cdots n}(\mu, k) \RR^{01}((\mu, k) - h^{(2 \cdots n)}) \cdots \RR^{0, n-1}((\mu, k) - h^{(n)}) \RR^{0n}(\mu, k). \qedhere
\]
\end{proof}

\subsection{Normalized trace function}

We now put together the normalization factors to define the desired normalized trace function as
\begin{multline} \label{eq:norm-tr}
F^{V_1, \ldots, V_n}(z_1, \ldots, z_n; \lambda, \omega, \mu, k) := \QQ_{V_n^*}((\mu, k) - h^{(* 1 \cdots * n)})^{-1} \cdots \QQ_{V_1^*}((\mu, k) - h^{(*1)})^{-1} \\
 \JJ^{1 \cdots n}_{V_1[z_1^{\pm 1}] \otimes \cdots \otimes V_n[z_n^{\pm 1}]}(z_1, \ldots, z_n; \lambda, \omega)^{-1} \Psi^{V_1, \ldots, V_n}(z_1, \ldots, z_n; \lambda, \omega, \mu - \rho, k - h) \delta_q(\lambda, \omega).
\end{multline}

\section{Macdonald-Ruijsenaars equations} \label{sec:mr}

In this section, we prove the Macdonald-Ruijsenaars equations for $F^{V_1, \ldots, V_n}(z_1, \ldots, z_n; \lambda, \omega, \mu, k)$ by explicitly computing the radial part of an explicit central element in a completion of $U_q(\wgg)$.  The bulk of this section is devoted to this computation.

\subsection{The statement}

Let $W$ be an integrable lowest weight $U_q(\wgg)$-module of non-positive integer level $k_W$, and let $V_1, \ldots, V_n$ be finite-dimensional $U_q(\agg)$-modules. Define the difference operator 
\[
\DD_W(\omega, k) := \sum_{\nu \in \hh^*, a \in \CC} \Tr|_{W[\nu + a \delta + k_W \Lambda_0]}\Big(\RR_{WV_1}(1, z_1; (\lambda, \omega) - h^{(2\cdots n)}) \cdots \RR_{W V_n}(1, z_n; \lambda, \omega)\Big) q^{-2ka} T^{\lambda, \omega}_{\nu, k_W},
\]
where $T^{\lambda, \omega}_{\nu, k_W} f(\lambda, \omega) = f(\lambda - \nu, \omega - k_W)$. Let this operator act on functions valued in 
\[
(V_1[z_1^{\pm}] \otimes \cdots \otimes V_n[z_n^{\pm} 1]) \otimes (V_n^* \otimes \cdots \otimes V_1^*).
\]
The Macdonald-Ruijsenaars equations, whose proof occupies the remainder of this section, state that $\DD_W$ is diagonalized on the normalized trace functions.

\begin{theorem}[Macdonald-Ruijsenaars equation] \label{thm:mr}
For any integrable lowest weight representation $W$ of non-positive integer level $\eta$, we have
\[
\DD_W(\omega, k) F^{V_1, \ldots, V_n}(z_1, \ldots, z_n; \lambda, \omega, \mu, k) = \chi_W(q^{-2\mu - 2kd}) F^{V_1, \ldots, V_n}(z_1, \ldots, z_n; \lambda, \omega, \mu, k),
\]
where $\chi_W$ is the character of $W$.
\end{theorem}

\subsection{Difference equations from radial parts of central elements}

The proof of Theorem \ref{thm:mr} is based on the computation of radial parts of certain central elements of $U_q(\agg)$.  Their existence for quantum affine algebras was shown in \cite{Eti}. Let $W$ be an integrable lowest weight $U_q(\wgg)$-representation of non-positive integer level $k_W$. Consider the element
\[
C_W := (\id \otimes \Tr|_W)(\cR^{21} \cR)(1 \otimes q^{-2 \wrho}).
\]
It lies in a certain completion of $U_q(\wgg)$, and its properties were characterized in \cite{Eti}.

\begin{prop}[{\cite[Theorem 2]{Eti}}] \label{prop:eti-cent}
The following is true of the elements $C_W$:
\begin{itemize}
\item[(a)] $C_W$ is central in (a completion of) $U_q(\wgg)$;

\item[(b)] $C_W$ acts by $\chi_W(q^{-2\mu - 2kd - 2 \wrho})$ on $M_{\mu, k}$.
\end{itemize}
\end{prop}

\begin{remark}
The action given in Proposition \ref{prop:eti-cent}(b) has a change in sign in the character because our coproduct and antipode are different from that of \cite{Eti}.
\end{remark}

\begin{prop}[{\cite[Theorem 5]{Eti}}] \label{prop:eti-diff}
For any $W$, there is a unique difference operator $\MM_W$ in $\lambda$ so that 
\[
\Tr|_{M_{\mu, k}}\Big(\wPhi_{\mu, k}^{V_1, \ldots, V_n}(z_1, \ldots, z_n) C_W q^{2\lambda + 2 \omega d}\Big) = \MM_W \Tr|_{M_{\mu, k}}\Big(\wPhi_{\mu, k}^{V_1, \ldots, V_n}(z_1, \ldots, z_n) q^{2\lambda + 2 \omega d}\Big).
\]
\end{prop}

\begin{remark}
The result of \cite[Theorem 5]{Eti} does not have a spectral parameter and is stated for a single evaluation representation, but the statement of Proposition \ref{prop:eti-diff} follows from the same argument.
\end{remark}

\begin{prop} \label{prop:rad-diag}
The operators $\MM_W$ satisfy:
\begin{itemize}
\item[(a)] for $W, W'$, we have $[\MM_W, \MM_{W'}] = 0$;

\item[(b)] we have 
\[
\MM_W \wPsi^{V_1, \ldots, V_n}(z_1, \ldots, z_n; \lambda, \omega, \mu, k) = \chi_W(q^{-2 \mu - 2kd - 2\wrho}) \wPsi^{V_1, \ldots, V_n}(z_1, \ldots, z_n; \lambda, \omega, \mu, k).
\]
\end{itemize}
\end{prop}

To prove Theorem \ref{thm:mr}, we now identify the operators $\MM_W$ and $\DD_W$ via the following steps:
\begin{enumerate}
\item[1.] compute universal versions of the operators $\MM_W$;

\item[2.] express $\MM_W$ in terms of an unknown coefficient $\GG(\lambda, \omega)$ by evaluating this universal computation;

\item[3.] characterize $\GG(\lambda, \omega)$ as the solution to a coproduct identity;

\item[4.] solve the coproduct identity;

\item[5.] conclude the Macdonald-Ruijsenaars equations.
\end{enumerate}
We carry out these steps in the following subsections.

\subsection{Computing universal versions of the operators $\MM_W$} 

In this subsection, we express the trace function 
\[
\Tr|_{M_{\mu, k}}\Big(\wPhi_{\mu, k}^{V_1, \ldots, V_n}(z_1, \ldots, z_n) C_W q^{2\lambda + 2 \omega d}\Big)
\]
in terms of the evaluation of a certain universal expression involving fusion matrices.  Define the representations
\[
V := V_1[z_1^{\pm 1}] \otimes \cdots \otimes V_n[z_n^{\pm 1}] \qquad \text{ and } \qquad V^* := V_n^* \otimes \cdots \otimes V_1^*.
\]
Label tensor factors of the tensor product $M_{\mu, k} \otimes V \otimes V^* \otimes U_q(\wgg) \otimes U_q(\wgg)$ by $0$, $1$, $1*$, $2$, and $3$ in that order.

\begin{lemma} \label{lem:trace-rmat}
The trace function is given by 
\begin{multline*}
\Tr|_{M_{\mu, k}}\Big(\wPhi_{\mu, k}^{V_1, \ldots, V_n}(z_1, \ldots, z_n) C_W q^{2\lambda + 2 \omega d}\Big)\\ = \Tr|_{W} \circ m_{23}\Big(S_3 \Big(\Tr|_{M_{\mu, k}}\Big(\wPhi_{\mu, k}^{V_1, \ldots, V_n}(z_1, \ldots, z_n) \cR^{20} (\cR^{03})^{-1} q_0^{2\lambda + 2 \omega d}\Big) \Big)q_3^{-2\wrho}\Big).
\end{multline*}
\end{lemma}
\begin{proof}
By definition the trace function is given by the trace over $M_{\mu, k} \otimes W$ of the composition
\[
M_{\mu, k} \otimes W \overset{q^{2\lambda + 2 \omega d} \otimes q^{-2 \wrho}} \to M_{\mu, k} \otimes W \overset{\cR^{20} \cR^{02}} \to M_{\mu, k} \otimes W \overset{\wPhi_{\mu, k}^{V_1, \ldots, V_n}(z_1, \ldots, z_n)} \to M_{\mu, k} \otimes V(z) \otimes W.
\]
The composition above is given by 
\[
\pi_{W, 3} \circ m_{23} \circ \Big(\wPhi_{\mu, k}^{V_1, \ldots, V_n}(z_1, \ldots, z_n) \cR^{20} \cR^{03} q_3^{-2\wrho} q_0^{2\lambda + 2 \omega d}\Big),
\]
so we conclude that 
\begin{align*}
\Tr|_{M_{\mu, k}}&\Big(\wPhi_{\mu, k}^{V_1, \ldots, V_n}(z_1, \ldots, z_n) C_W q^{2\lambda + 2 \omega d}\Big)\\
 &= \Tr|_{M_{\mu, k}} \Tr|_W \Big(m_{23} \circ \Big(\wPhi_{\mu, 1}^{V_1, \ldots, V_n}(z_1, \ldots, z_n) \cR^{20} \cR^{03} q_3^{-2\wrho} q_0^{2\lambda + 2 \omega d}\Big)\Big)\\
&= \Tr|_{M_{\mu, k}} \Tr|_W \Big(m_{23} \circ \Big(\wPhi_{\mu, k}^{V_1, \ldots, V_n}(z_1, \ldots, z_n) \cR^{20} S_3((\cR^{03})^{-1}) q_3^{-2\wrho} q_0^{2\lambda + 2 \omega d}\Big)\Big)\\
&= \Tr|_{W} \circ m_{23}\Big(S_3 \Big(\Tr|_{M_{\mu, k}}\Big(\wPhi_{\mu, k}^{V_1, \ldots, V_n}(z_1, \ldots, z_n) \cR^{20} (\cR^{03})^{-1} q_0^{2\lambda + 2 \omega d}\Big) \Big)q_3^{-2\wrho}\Big). \qedhere
\end{align*}
\end{proof}

We will compute universal versions of the outcome of Lemma \ref{lem:trace-rmat} in two steps.  Define the quantities
\begin{align*}
Z_V(z_1, \ldots, z_n; \lambda, \omega, \mu, k) &:= \Tr|_{M_{\mu, k}}\Big(\wPhi_{\mu, k}^{V_1, \ldots, V_n}(z_1, \ldots, z_n) \cR^{20} q_0^{2\lambda + 2 \omega d}\Big) \\
X_V(z_1, \ldots, z_n; \lambda, \omega, \mu, k) &:= \Tr|_{M_{\mu, k}}\Big(\wPhi_{\mu, k}^{V_1, \ldots, V_n}(z_1, \ldots, z_n) \cR^{20} (\cR^{03})^{-1} q_0^{2\lambda + 2 \omega d}\Big),
\end{align*}
and notice that $X_V(z_1, \ldots, z_n; \lambda, \omega, \mu, k) = Z_V(z_1, \ldots, z_n; (\lambda, \omega) - h^{(3)}/2, \mu, k) + (\text{l.o.t.})$.  We compute these quantities in the following two lemmas, whose proofs are deferred to Subsection \ref{sec:mr-comps}.

\begin{lemma} \label{lem:z-val}
We have the identity
\[
Z_V(z_1, \ldots, z_n; \lambda, \omega, \mu, k) = \cJ^{12}(\lambda, \omega) q_2^{-kd} \wPsi^{V_1, \ldots, V_n}(z_1, \ldots, z_n; (\lambda, \omega) - h^{(2)}/2, \mu, k).
\]
\end{lemma}

\begin{lemma} \label{lem:x-val}
We have the identity
\begin{multline*}
X_V(z_1, \ldots, z_n; \lambda, \omega, \mu, k) = q_3^{2\lambda + 2 \omega d} \cJ^{3, 12}(\lambda, \omega) \cJ^{12}((\lambda, \omega) + h^{(3)}/2) q_2^{-kd}q_3^{kd}\\ \wPsi^{V_1, \ldots, V_n}(z_1, \ldots, z_n; (\lambda, \omega) + h^{(3)}/2 - h^{(2)}/2, \mu, k) \cJ^{32}(\lambda, \omega)^{-1} q_3^{-2\lambda - 2\omega d}.
\end{multline*}
\end{lemma}

\subsection{Evaluating the universal computations}

Define the quantity
\[
\GG(\lambda, \omega) := \QQ((\lambda, \omega) - h^{(1)})^{-1} S(\QQ(\lambda, \omega)) q^{-2\rho}.
\]
The goal of this subsection is to evaluate the result of Lemma \ref{lem:trace-rmat} using Lemma \ref{lem:x-val} to obtain the following expression for the trace function in terms of $\GG(\lambda, \omega)$.

\begin{prop} \label{prop:trace-coeff}
We have the identity
\begin{multline} \label{eq:tr-coeff}
\Tr|_{M_{\mu, k}}\Big(\wPhi_{\mu, k}^{V_1, \ldots, V_n}(z_1, \ldots, z_n) C_W q^{2\lambda + 2 \omega d}\Big)\\  = \sum_{\nu \in \hh^*, a \in \CC} \Tr|_{W[\nu + a \delta + k_W \Lambda_0]}\Big(q^{-2ka - 2ha}\GG(\lambda, \omega) \RR_{WV}(1; z_1, \ldots, z_n; \lambda, \omega)\Big)\\ \wPsi^{V_1, \ldots, V_n}(z_1, \ldots, z_n; \lambda - \nu, \omega - k_W, \mu, k) .
\end{multline}
\end{prop}
\begin{proof}
Recall the notations of (\ref{eq:ser-exp}) for expansions of $\cR$, $\cL(\lambda, \omega)$, and their inverses. We rewrite Lemma \ref{lem:x-val} in terms of the renormalized fusion operator as
\begin{multline*}
X_V(z_1, \ldots, z_n; \lambda, \omega, \mu, k) = q_3^{2\lambda + 2 \omega d} \cR^{12, 3} \cL^{3, 12}(\lambda, \omega) \cR^{21}((\lambda, \omega) + h^{(3)}/2) q_2^{-kd} q_3^{kd}\\
\wPsi^{V_1, \ldots, V_n}(z_1, \ldots, z_n; (\lambda, \omega) + h^{(3)}/2 - h^{(2)}/2, \mu, k) \cL^{32}(\lambda, \omega)^{-1} (\cR^{23})^{-1} q_3^{-2\lambda - 2\omega d}.
\end{multline*}
In coordinates, this means that 
\begin{align*}
X_V&(z_1, \ldots, z_n; \lambda, \omega, \mu, k) \\
&= \sum_{i, j, m, n, r, s} \Big(a_i^{(1)} f_j^{(1)}(\lambda, \omega) b_m e_n \otimes a_i^{(2)} f_j^{(2)}(\lambda, \omega) a_m f_n((\lambda, \omega) + h^{(3)}/2)q_2^{-kd} \otimes q^{2\lambda + 2 \omega d} b_i e_j q_3^{kd}\Big) \\
&\phantom{=============} \wPsi^{V_1, \ldots, V_n}(z_1, \ldots, z_n; (\lambda, \omega) + h^{(3)}/2 - h^{(2)}/2, \mu, k) \Big(1 \otimes f_r'(\lambda, \omega) a_s' \otimes e_r' b_s' q^{-2\lambda - 2 \omega d}\Big).
\end{align*}
Therefore, rewriting Lemma \ref{lem:trace-rmat} as
\[
\Tr|_{M_{\mu, k}}\Big(\wPhi_{\mu, k}^{V_1, \ldots, V_n}(z_1, \ldots, z_n) C_W q^{2\lambda + 2 \omega d}\Big) = \Tr|_{W} \circ m_{23}\Big(S_3 (X_V(z_1, \ldots, z_n; \lambda, \omega, \mu, k))q_3^{-2\wrho}\Big),
\]
we obtain by substitution and cyclic permutation of the trace that
\begin{align*}
\Tr&|_{M_{\mu, k}}\Big(\wPhi_{\mu, k}^{V_1, \ldots, V_n}(z_1, \ldots, z_n) C_W q^{2\lambda + 2 \omega d}\Big)\\
&= \sum_{i, j, m, n, r, s} a_i^{(1)} f_j^{(1)}(\lambda, \omega) b_m e_n \Tr|_W\Big(a_i^{(2)} f_j^{(2)}(\lambda, \omega) a_m f_n((\lambda, \omega) - h^{(2)}/2) \\
&\phantom{===} \wPsi^{V_1, \ldots, V_n}(z_1, \ldots, z_n; (\lambda, \omega) - h^{(2)}, \mu, k) q^{-kd} f_r'(\lambda, \omega) a_s' q^{2\lambda + 2\omega d} S(b_s') S(e_r') q^{-kd} S(e_j) S(b_i) q^{-2\lambda - 2\omega d - 2\wrho}\Big)\\
&= \sum_{\nu, a} \sum_{i, j, m, n, r, s} a_i^{(1)} f_j^{(1)}(\lambda, \omega) b_m e_n \wPsi^{V_1, \ldots, V_n}(z_1, \ldots, z_n; \lambda - \nu, \omega - k_W, \mu, k)\\
&\phantom{===} \Tr|_{W[\nu + k_W \Lambda_0 + a \delta]}\Big(q^{-kd} f_r'(\lambda, \omega) a_s'q^{2\lambda + 2\omega d} S(b_s') S(e_r')\\
&\phantom{==================} q^{-kd} S(e_j) S(b_i) q^{-2\lambda - 2\omega d - 2 \wrho} a_i^{(2)} f_j^{(2)}(\lambda, \omega) a_m f_n(\lambda - \nu/2, \omega - k_W/2)\Big).
\end{align*}
We now evaluate the sum over $r, s$ explicitly in terms of $\QQ(\lambda, \omega)^{-1}$.
\begin{lemma} \label{lem:rs-sum}
We have the sum
\[
\sum_{r, s} f_r'(\lambda, \omega) a_s' q^{2\lambda + 2 \omega d} S(b_s') S(e_r') = \QQ((\lambda, \omega) - h^{(1)})^{-1} q^{2\lambda + 2 \omega d} q^{-\sum_i x_i^2 - 2cd}.
\]
\end{lemma}
\begin{proof}
Apply $m_{21} \circ S_1$ to 
\[
\cL(\lambda, \omega)^{-1} (\cR^{21})^{-1} q_1^{-2\lambda - 2 \omega d} = q^{\Omega} q_1^{-2\lambda - 2\omega d} \cL(\lambda, \omega)^{-1}
\]
to obtain
\[
\sum_{r, s} f_r'(\lambda, \omega) a_s' q^{2\lambda + 2 \omega d} S(b_s') S(e_r') = \sum_i f_i'(\lambda, \omega) S(e_i') q^{2\lambda + 2 \omega d} q^{-\sum_i x_i^2 - 2 cd}.
\]
Now, apply $m_{321} \circ S_2$ to 
\[
\cL^{21, 3}(\lambda, \omega) \cL^{12}((\lambda, \omega) - h^{(3)}/2) \cL^{23}((\lambda, \omega) + h^{(1)}/2)^{-1}  = \cL^{1, 32}(\lambda, \omega).
\]
The right hand side yields $1$, while the left yields
\begin{multline*}
\sum_{i, j, l} f_i(\lambda, \omega) f_l'((\lambda, \omega) + h^{(1)}/2) S(e_l') S(f_j((\lambda, \omega) - h^{(1)}/2)) S(e_i^{(1)}) e_i^{(2)} e_j\\ = \sum_l f_l'((\lambda, \omega) + h^{(1)}/2) S(e_l') \QQ((\lambda, \omega) - h^{(1)}/2).
\end{multline*}
We conclude that 
\[
\sum_l f_l'(\lambda, \omega) S(e_l') = \QQ((\lambda, \omega) - h^{(1)})^{-1},
\]
which yields the desired upon substitution.
\end{proof}
From Lemma \ref{lem:rs-sum}, we conclude that 
\begin{multline*}
\Tr|_{M_{\mu, k}}\Big(\wPhi_{\mu, k}^{V_1, \ldots, V_n}(z_1, \ldots, z_n) C_W q^{2\lambda + 2 \omega d}\Big)\\
= \sum_{\nu, a} \sum_{i, j, m, n} a_i^{(1)} f_j^{(1)}(\lambda, \omega) b_m e_n \wPsi^{V_1, \ldots, V_n}(z_1, \ldots, z_n; \lambda - \nu, \omega - k_W, \mu, k)\\
\Tr|_{W[\nu + k_W \Lambda_0 + a \delta]}\Big(q^{-2ka} \QQ(\lambda - \nu, \omega - k_W)^{-1} q^{2\lambda + 2\omega d} q^{-\sum_i x_i^2 - 2cd}\\ S(e_j) S(b_i) q^{-2\lambda - 2\omega d - 2\wrho} a_i^{(2)} f_j^{(2)}(\lambda, \omega) a_m f_n(\lambda - \nu/2, \omega - k_W/2)\Big).
\end{multline*}
We now simplify the sum over $i, j$.  Apply $m_{32} \circ S_3$ to the identity
\[
q_3^{2\lambda + 2 \omega d + 2 \wrho} \cR^{12, 3} \cL^{3, 12}(\lambda, \omega) =  q_3^{2\wrho} \cL^{3, 12}(\lambda, \omega) q_3^{2 \lambda + 2 \omega d} q^{-\Omega_{23} - \Omega_{13}}
\]
from Proposition \ref{prop:renorm-abrr} to obtain
\begin{multline*}
\sum_{i, j} a_i^{(1)} f_j^{(1)}(\lambda, \omega) \otimes S(e_j) S(b_i) q^{-2\lambda - 2\omega d - 2\wrho} a_i^{(2)} f_j^{(2)}(\lambda, \omega) \\ = \sum_i q^\Omega \Big(f_i^{(1)}(\lambda, \omega) \otimes q^{\sum_i x_i^2 + 2cd} q^{-2\lambda - 2\omega d} S(e_i) q^{-2\wrho} f_i^{(2)}(\lambda, \omega) \Big).
\end{multline*}
We may substitute this into the previous expression to obtain
\begin{multline*}
\Tr|_{M_{\mu, k}}\Big(\wPhi_{\mu, k}^{V_1, \ldots, V_n}(z_1, \ldots, z_n) C_W q^{2\lambda + 2 \omega d}\Big)
= \sum_{\nu, a} \sum_{i, m, n} f_i^{(1)}(\lambda, \omega) b_m e_n\\ \wPsi^{V_1, \ldots, V_n}(z_1, \ldots, z_n; \lambda - \nu, \omega - k_W, \mu, k)
\Tr|_{W[\nu + k_W \Lambda_0 + a \delta]}\Big(q^{-2ka} \QQ(\lambda - \nu, \omega - k_W)^{-1} \\S(e_i) q^{-2\wrho}f_i^{(2)}(\lambda, \omega) a_m f_n(\lambda - \nu/2, \omega - k_W/2)\Big).
\end{multline*}
Apply $P^{12} \circ m_{12} \circ S_1$ to the shifted $2$-cocycle relation (\ref{eq:renorm-cocycle}) rewritten in the form
\[
q_2^{-2\wrho}\cL^{1, 32}(\lambda, \omega) = q_2^{-2\wrho}\cL^{21, 3}(\lambda, \omega) \cL^{12}((\lambda, \omega) - h^{(3)}/2) \cL^{23}((\lambda, \omega) + h^{(1)}/2)^{-1}
\]
to obtain
\begin{align*}
\sum_i f_i^{(1)}(\lambda, \omega) \otimes S(e_i) q^{-2\wrho} f_i^{(2)}(\lambda, \omega) 
&= \sum_{i, j, l} f_i(\lambda, \omega) f_l'((\lambda, \omega) - h^{(2)}/2) \otimes S(e_j) S(e_i^{(2)}) q^{-2\wrho} e_i^{(1)} f_j((\lambda, \omega) - h^{(1)}/2) e_l' \\
&= \sum_{j, l} f_l'((\lambda, \omega) - h^{(2)}/2) \otimes S(e_j) q^{-2\wrho} f_j((\lambda, \omega) - h^{(1)}/2) e_l'.
\end{align*}
Substituting this into the previous expression, we obtain
\begin{align*}
\Tr&|_{M_{\mu, k}}\Big(\wPhi_{\mu, k}^{V_1, \ldots, V_n}(z_1, \ldots, z_n) C_W q^{2\lambda + 2 \omega d}\Big)\\
&= \sum_{\nu, a} \sum_{i, j, m, n} f_i'(\lambda - \nu/2, \omega - k_W/2) b_m e_n \wPsi^{V_1, \ldots, V_n}(z_1, \ldots, z_n; \lambda - \nu, \omega - k_W, \mu, k)\\
&\phantom{====} \Tr|_{W[\nu + k_W \Lambda_0 + a \delta]}\Big(q^{-2ka} \QQ(\lambda - \nu, \omega - k_W)^{-1}S(e_j) q^{-2\wrho} f_j(\lambda, \omega) e_i' a_m f_n(\lambda - \nu/2, \omega - k_W/2)\Big)\\
&= \sum_{\nu, a} \sum_{i, j, m, n} f_i'(\lambda - \nu/2, \omega - k_W/2) b_m e_n \Tr|_{W[\nu + k_W \Lambda_0 + a \delta]}\Big(q^{-2ka} \QQ(\lambda - \nu, \omega - k_W)^{-1} S(e_j) q^{-2\wrho} f_j(\lambda, \omega)\\
&\phantom{====}  e_i' a_m f_n(\lambda - \nu/2, \omega - k_W/2)\Big)\wPsi^{V_1, \ldots, V_n}(z_1, \ldots, z_n; \lambda - \nu, \omega - k_W, \mu, k)\\
&= \sum_{\nu, a} \sum_{i, m, n} f_i'(\lambda - \nu/2, \omega - k_W/2) b_m e_n \Tr|_{W[\nu + k_W \Lambda_0 + a \delta]}\Big(q^{-2ka} \QQ(\lambda - \nu, \omega - k_W)^{-1} S(\QQ(\lambda, \omega)) \\ 
&\phantom{====}  q^{-2\wrho} e_i' a_m f_n(\lambda - \nu/2, \omega - k_W/2)\Big)\wPsi^{V_1, \ldots, V_n}(z_1, \ldots, z_n; \lambda - \nu, \omega - k_W, \mu, k)\\
&= \sum_{\nu, a} \Tr|_{W[\nu + k_W \Lambda_0 + a \delta]}\Big(q^{-2ka} \QQ(\lambda - \nu, \omega - k_W)^{-1} S(\QQ(\lambda, \omega)) q^{-2\wrho} \cL_{WV}(\lambda - \nu/2, \omega - k_W/2)^{-1}\\
&\phantom{====} \cR_{WV} \cL_{WV}^{21}(\lambda - \nu/2, \omega - k_W/2)\Big)\wPsi^{V_1, \ldots, V_n}(z_1, \ldots, z_n; \lambda - \nu, \omega - k_W, \mu, k)\\
&= \sum_{\nu, a} q^{-2(k+\ch)a}\Tr|_{W[\nu + k_W \Lambda_0 + a \delta]}\Big(\GG(\lambda, \omega) \RR_{WV}(1; z_1, \ldots, z_n; \lambda, \omega) \Big)\\
&\phantom{====}\wPsi^{V_1, \ldots, V_n}(z_1, \ldots, z_n; \lambda - \nu, \omega - k_W, \mu, k). \qedhere
\end{align*}
\end{proof}

\subsection{The coproduct identity for $\GG(\lambda, \omega)$}

The goal of this subsection is to show that $\GG(\lambda, \omega)$ satisfies a coproduct identity, meaning that it is related in a simple way to its coproduct.  We first prove an analogous result for $\QQ(\lambda, \omega)$.

\begin{lemma} \label{lem:q-coprod}
We have
\[
\Delta(\QQ(\lambda, \omega)) = ((S \otimes S) \cL((\lambda, \omega) + h^{(1)}/2 + h^{(2)}/2)^{-1} \Big(\QQ((\lambda, \omega) + h^{(2)}) \otimes \QQ(\lambda, \omega)\Big) \cL^{21}((\lambda, \omega) + h^{(1)}/2 + h^{(2)}/2)^{-1}. 
\]
\end{lemma}
\begin{proof}
Apply $(m_{42} \otimes m_{31}) \circ S_4 \circ S_3$ to the shifted $2$-cocycle relation
\[
\cL^{21, 3}(\lambda, \omega) \cL^{12}((\lambda, \omega) - h^{(3)}/2) = \cL^{1, 32}(\lambda, \omega) \cL^{32}((\lambda, \omega) + h^{(1)}/2)
\]
from Proposition \ref{prop:renorm-abrr}(c).  From the left hand side, we obtain
\begin{multline} \label{eq:four-eval-1}
\sum_{i, j} \Big(S(f_i^{(2)}(\lambda, \omega)) e_i^{(1)} \otimes S(f_i^{(1)}(\lambda, \omega)) e_i^{(2)}\Big) \Big(f_j((\lambda, \omega) + h^{(1)}/2 + h^{(2)}/2) \otimes e_j\Big)\\ = \Delta(\QQ(\lambda, \omega)) \cL^{21}((\lambda, \omega) + h^{(1)}/2 + h^{(2)}/2).
\end{multline}
From the right hand side, we obtain
\begin{align} \nonumber
\sum_{i, j} & S(f_j^{(2)}((\lambda, \omega) + h^{(2)}/2)) S(f_i^{(1)(2)}(\lambda, \omega)) f_i^{(2)}(\lambda, \omega) e_j \otimes S(f_j^{(1)}((\lambda, \omega) + h^{(2)}/2)) S(f_i^{(1)(1)}(\lambda, \omega)) e_i \\ \nonumber
&= \sum_{i, j} S(f_j^{(2)}((\lambda, \omega) + h^{(2)}/2)) S(f_i^{(2)(1)}(\lambda, \omega)) f_i^{(2)(2)}(\lambda, \omega) e_j \otimes S(f_j^{(1)}((\lambda, \omega) + h^{(2)}/2)) S(f_i^{(1)}(\lambda, \omega)) e_i \\ \nonumber
&= \sum_{i, j} S(f_j^{(2)}((\lambda, \omega) + h^{(2)}/2)) e_j \otimes S(f_j^{(1)}((\lambda, \omega) + h^{(2)}/2)) S(f_i(\lambda, \omega)) e_i \\
&= \sum_j S(f_j^{(2)}((\lambda, \omega) + h^{(2)}/2)) e_j \otimes S(f_j^{(1)}((\lambda, \omega) + h^{(2)}/2)) \QQ(\lambda, \omega), \label{eq:four-eval-2}
\end{align}
where the first equality follows by coassociativity and the second by the relation $m_{12} \circ S_1 \circ \Delta = \eta \circ \eps$ in a Hopf algebra. Now, applying $m_{21} \circ S_2 \circ S_3$ to the cocycle relation, we find for the left hand side
\begin{equation} \label{eq:three-eval-1}
\sum_{i, j} S(f_j((\lambda, \omega) + h^{(2)}/2)) S(e_i^{(1)}) e_i^{(2)} e_j \otimes S(f_i(\lambda, \omega)) = \QQ((\lambda, \omega) + h^{(2)}/2) \otimes 1
\end{equation}
and for the right hand side
\begin{multline} \label{eq:three-eval-2}
\sum_{i, j} S(e_j) S(f_i^{(2)}(\lambda, \omega)) e_i \otimes S(f_j((\lambda, \omega) + h^{(1)}/2)) S(f_i^{(1)}(\lambda, \omega))\\ = ((S \otimes S) \cL((\lambda, \omega) + h^{(1)}/2)) \sum_i S(f_i^{(2)}(\lambda, \omega)) e_i \otimes S(f_i^{(1)}(\lambda, \omega)).
\end{multline}
Equating (\ref{eq:three-eval-1}) and (\ref{eq:three-eval-2}), we obtain
\begin{equation} \label{eq:three-eval-conc}
\sum_i S(f_i^{(2)}(\lambda, \omega)) e_i \otimes S(f_i^{(1)}(\lambda, \omega)) = ((S \otimes S)\cL((\lambda, \omega) + h^{(1)}/2))^{-1} (\QQ((\lambda, \omega) + h^{(2)}/2) \otimes 1).
\end{equation}
Equating (\ref{eq:four-eval-1}) and (\ref{eq:four-eval-2}) and substituting (\ref{eq:three-eval-conc}), we conclude that
\begin{multline*}
\Delta(\QQ(\lambda, \omega)) = ((S \otimes S) \cL((\lambda, \omega) + h^{(1)}/2 + h^{(2)}/2)^{-1}\\ \Big(\QQ((\lambda, \omega) + h^{(2)}) \otimes \QQ(\lambda, \omega)\Big) \cL^{21}((\lambda, \omega) + h^{(1)}/2 + h^{(2)}/2)^{-1}. \qedhere
\end{multline*}
\end{proof}

We are now ready to prove the following coproduct identity for $\GG(\lambda, \omega)$.

\begin{lemma} \label{lem:g-coprod}
We have 
\[
\Delta(\GG(\lambda, \omega)) = \LL^{21}(\lambda, \omega) \Big(\GG(\lambda, \omega) \otimes \GG((\lambda, \omega) - h^{(1)})\Big)\LL^{21}(\lambda, \omega)^{-1}.
\]
\end{lemma}
\begin{proof}
Recall that $\GG(\lambda, \omega) = \QQ(\lambda - h^{(1)}, \omega)^{-1} S(\QQ(\lambda, \omega)) q^{-2\rho}$, so we have that 
\begin{align*}
\Delta(\GG(\lambda, \omega)) &= \Delta(\QQ((\lambda, \omega) - h^{(1)})^{-1} \Big((S \otimes S)(\Delta^{21}(\QQ(\lambda, \omega))\Big) (q^{-2\rho} \otimes q^{-2\rho}).
\end{align*}
By Lemma \ref{lem:q-coprod}, we have that 
\begin{align*}
\Delta(\QQ((\lambda, \omega) - h^{(1)})^{-1} &=  \cL^{21}((\lambda, \omega) - h^{(1)}/2 - h^{(2)}/2) \Big(\QQ((\lambda, \omega) - h^{(1)})^{-1} \otimes \QQ((\lambda, \omega) - h^{(1)} - h^{(2)})^{-1}\Big)\\
&\phantom{====================} ((S \otimes S) \cL((\lambda, \omega) - h^{(1)}/2 - h^{(2)}/2)\\
&= \LL^{21}(\lambda, \omega) \Big(\QQ((\lambda, \omega) - h^{(1)})^{-1} \otimes \QQ((\lambda, \omega) - h^{(1)} - h^{(2)})^{-1}\Big) ((S \otimes S)\LL(\lambda, \omega))
\end{align*}
and that
\begin{align*}
(S \otimes S) (\Delta^{21}(\QQ(\lambda, \omega))
 &= \Big((S \otimes S)\cL((\lambda, \omega) - h^{(1)}/2 - h^{(2)}/2)^{-1}\Big)\\ 
&\phantom{===}\Big(S(\QQ(\lambda, \omega)) \otimes S(\QQ((\lambda, \omega) - h^{(1)}))\Big) (S^2 \otimes S^2) \cL^{21}((\lambda, \omega) - h^{(1)}/2 - h^{(2)}/2)^{-1}\\
&= \Big((S \otimes S) \LL(\lambda, \omega)^{-1}\Big) \Big(S(\QQ(\lambda, \omega)) \otimes S(\QQ((\lambda, \omega) - h^{(1)}))\Big) \LL^{21}(\lambda, \omega)^{-1},
\end{align*}
where we use that $\LL^{21}(\lambda, \omega)^{-1}$ has weight $0$.  We conclude that 
\begin{align*}
\Delta(\GG(\lambda, \omega)) &= \LL^{21}(\lambda, \omega) \Big(\GG(\lambda, \omega) \otimes \GG((\lambda, \omega) - h^{(1)})\Big)\LL^{21}(\lambda, \omega)^{-1}. \qedhere
\end{align*}
\end{proof}

\subsection{Solving the coproduct identity}

The goal of this subsection is to prove Lemma \ref{lem:g-comp} on the value of the $\GG(\lambda, \omega)$.  Recall that $\GG(\lambda, \omega)$ has weight $0$ and may be evaluated on any representation which is locally nilpotent with respect to the $U_q(\gg)$-action.  In particular, letting $L_{\mu, k}$ denote the irreducible integrable representation of $U_q(\wgg)$ with highest weight $\mu + k \Lambda_0$, we may define the function $\eta_{\mu, k}(\lambda, \omega)$ to be the eigenvalue of $\GG(\lambda, \omega)$ on the highest weight vector of $L_{\mu, k}$.  Our method proceeds by finding $\eta_{\mu, k}(\lambda, \omega)$ and showing that $\GG(\lambda, \omega)$ is determined by it.  First, we derive a compatibility relation for $\eta_{\mu, k}(\lambda, \omega)$.

\begin{lemma} \label{lem:zero-curve}
For dominant integral weights $\mu_1 + k_1 \Lambda_0$ and $\mu_2 + k_2 \Lambda_0$, the function $\eta_{\mu, k}(\lambda, \omega)$ satisfies the zero-curvature relation
\begin{equation} \label{eq:zero-curve}
\eta_{\mu_1, k_1}(\lambda, \omega) \eta_{\mu_2, k_2}(\lambda - \mu_1, \omega - k_1) = \eta_{\mu_2, k_2}(\lambda, \omega) \eta_{\mu_1, k_1}(\lambda - \mu_2, \omega - k_2).
\end{equation}
\end{lemma}
\begin{proof}
By \cite[Theorem 6.2.2]{Lus}, the category of integrable highest-weight representations of $U_q(\wgg)$ is semisimple, meaning in particular that $L_{\mu_1 + \mu_2, k_1 + k_2}$ is the subrepresentation in $L_{\mu_1, k_1} \otimes L_{\mu_2, k_2}$ and $L_{\mu_2, k_2} \otimes L_{\mu_1, k_1}$ generated by the highest weight vector.  Therefore, by Lemma \ref{lem:g-coprod} and the fact that $\LL^{21}(\lambda, \omega) \in q^{\Omega}\Big(1 + U_q(\abb_-)_{>0} \hotimes U_q(\abb_+)_{>0}\Big)$, both sides of the desired are equal to $\eta_{\mu_1 + \mu_2, k_1 + k_2}(\lambda, \omega)$, hence equal.
\end{proof}

We now work with the formal expansion of $\GG(\lambda, \omega)$.  Let $q = e^\hbar$ for the formal parameter $\hbar$, and work over the ring $\CC[[\hbar]]$.  From the compatibility relation, we now constrain the formal expansion of
\[
\weta_{\mu, k}(\lambda, \omega) := \eta_{\mu, k}\Big(\frac{\lambda}{\hbar}, \frac{\omega}{\hbar}\Big) \in \CC[[\hbar]].
\]
Notice that $\lim_{\hbar \to 0} \weta_{\mu, k}(\lambda, \omega) = 1$, since $\lim_{\hbar \to 0} \GG(\lambda, \omega) = 1$.  The proof of \cite[Lemma 7.56]{EL} applies verbatim to the affine setting, hence Lemma \ref{lem:zero-curve} implies the following lemma, which allows us to constrain the form of $\GG(\lambda, \omega)$.

\begin{lemma} \label{lem:eta-quot}
We may find some function $f(\lambda, \omega) \in \CC[[\hbar]]$ so that 
\[
\weta_{\mu, k}(\lambda, \omega) = \frac{f(\lambda - \hbar\mu, \omega - \hbar k)}{f(\lambda, \omega)}
\]
as a formal series in $\hbar$.
\end{lemma}

\begin{lemma} \label{lem:g-form}
As formal series in $\hbar$, we have
\[
\GG\Big(\frac{\lambda}{\hbar}, \frac{\omega}{\hbar}\Big) = \frac{f\Big((\lambda, \omega) - \hbar h^{(1)}\Big)}{f(\lambda, \omega)}.
\]
\end{lemma}
\begin{proof}
By Lemma \ref{lem:eta-quot}, the renormalized element
\[
\wGG(\lambda, \omega) := \GG\Big(\frac{\lambda}{\hbar}, \frac{\omega}{\hbar}\Big) \frac{f(\lambda, \omega)}{f\Big((\lambda, \omega) - \hbar h^{(1)}\Big)}
\]
acts by $1$ on the highest weight vector of any highest weight irreducible integrable representation.  If $\wGG(\lambda, \omega) \neq 1$, let its formal expansion take the form
\[
\wGG(\lambda, \omega) = 1 + \hbar^n g(\lambda, \omega) + O(\hbar^{n+1})
\]
for some non-zero $g(\lambda, \omega) \in U(\agg)$.  By Lemma \ref{lem:g-coprod}, we have that 
\begin{multline*}
\Delta\Big(1 + \hbar^n g(\lambda, \omega) + O(\hbar^{n+1})\Big)
\\= \LL^{21}\Big(\frac{\lambda}{\hbar}, \frac{\omega}{\hbar}\Big) \Big((1 + \hbar^n g(\lambda, \omega) + O(\hbar^{n+1})) \otimes (1 + \hbar^n g(\lambda, \omega) + O(\hbar^{n+1}))\Big) \LL^{21}\Big(\frac{\lambda}{\hbar}, \frac{\omega}{\hbar}\Big)^{-1},
\end{multline*}
hence canceling terms, dividing by $\hbar^n$, and looking modulo $\hbar U_q(\agg)$ yields
\[
\Delta_0(g(\lambda, \omega)) = g(\lambda, \omega) \otimes 1 + 1 \otimes g(\lambda, \omega),
\]
where $\Delta_0$ is the coproduct of $U(\agg)$ and both sides are considered modulo $\hbar U_q(\agg)$.  Let $\overline{g}(\lambda, \omega)$ be the class of $g(\lambda, \omega)$ modulo $\hbar U_q(\agg)$; this implies that $\overline{g}(\lambda, \omega) \in \agg$.  On the other hand, $g(\lambda, \omega)$ has zero weight, so $\overline{g}(\lambda, \omega) \in \ahh$, which implies that it is $0$ since $g(\lambda, \omega)$ and hence $\overline{g}(\lambda, \omega)$ vanishes on all highest weight vectors of highest weight irreducible integrable representations.  This is a contradiction, so $\wGG(\lambda, \omega) = 1$, as desired.
\end{proof}

We are finally ready to compute $\GG(\lambda, \omega)$ by computing the value of $f(\lambda, \omega)$.

\begin{lemma} \label{lem:g-comp}
The value of $\GG(\lambda, \omega)$ is
\[
\GG(\lambda, \omega) = \frac{\delta_q((\lambda, \omega) - h^{(1)})}{\delta_q(\lambda, \omega)}.
\]
\end{lemma}
\begin{proof}
It suffices to check that 
\[
\frac{f\Big((\lambda, \omega) - h^{(1)}\Big)}{f(\lambda, \omega)} = \frac{\delta_q\Big(\Big(\frac{\lambda}{\hbar}, \frac{\omega}{\hbar}\Big) - h^{(1)}\Big)}{\delta_q\Big(\frac{\lambda}{\hbar}, \frac{\omega}{\hbar}\Big)},
\]
as this formal equality implies the desired equality at $\hbar = 1$. Applying Proposition \ref{prop:trace-coeff} to $V = \CC$ and noting that $C_W$ acts on $M_{\mu, k}$ by $\chi_W(q^{-2\mu - 2kd - 2 \wrho})$ by Proposition \ref{prop:eti-cent}, we obtain as formal series in $\hbar$ that
\begin{align*}
\chi_W(&q^{-2\mu - 2kd - 2 \wrho}) \wPsi^\CC(z; \lambda/\hbar, \omega/\hbar, \mu, k)\\
&= \sum_{\nu \in \hh^*, a \in \CC} \Tr|_{W[\nu + a \delta + k_W \Lambda_0]}\Big(q^{-2ka - 2ha}\GG(\lambda/\hbar, \omega/\hbar) \RR_{W\CC}(1, z; \lambda/\hbar, \omega/\hbar)\Big) \wPsi^\CC(z; \lambda/\hbar - \nu, \omega/\hbar - k_W, \mu, k)\\
&= \sum_{\nu \in \hh^*, a \in \CC} \Tr|_{W[\nu + a \delta + k_W \Lambda_0]}\Big(q^{-2ka - 2ha}\GG(\lambda/\hbar, \omega/\hbar)\Big) \wPsi^\CC(z; \lambda/\hbar - \nu, \omega/\hbar - k_W, \mu, k)\\
&= \sum_{\nu \in \hh^*, a \in \CC} \Tr|_{W[\nu + a \delta + k_W \Lambda_0]}\Big(q^{-2ka - 2ha}\frac{f(\lambda - \hbar \nu, \omega - \hbar k_W)}{f(\lambda, \omega)}\Big) \wPsi^\CC(z; \lambda/\hbar - \nu, \omega/\hbar - k_W, \mu, k).
\end{align*}
Notice now that 
\[
\wPsi^\CC(z; \lambda, \omega, \mu, k) = \Tr|_{M_{\mu, k}}(q^{2\lambda + 2 \omega d}) = \frac{q^{2(\mu + kd + \wrho, \lambda + \omega d)}}{\delta_q(\lambda, \omega)}
\]
and therefore that 
\begin{multline*}
\chi_W(q^{-2\mu - 2kd - 2 \wrho}) = \sum_{\nu \in \hh^*, a \in \CC} \dim W[\nu + a \delta + k_W \Lambda_0]\\ \frac{f(\lambda - \hbar \nu, \omega - \hbar k_W)}{f(\lambda, \omega)} \frac{\delta_q(\lambda/\hbar, \omega/\hbar)}{\delta_q(\lambda/\hbar - \nu, \omega/\hbar - k_W)} q^{- 2(\mu + k \Lambda_0 + \wrho, \nu + a \delta + k_W \Lambda_0)},
\end{multline*}
so equating coefficients of power series in $q^{-2\mu - 2kd - 2\wrho}$ yields the desired
\[
\frac{f\Big((\lambda, \omega) - h^{(1)}\Big)}{f(\lambda, \omega)} = \frac{\delta_q\Big(\Big(\frac{\lambda}{\hbar}, \frac{\omega}{\hbar}\Big) - h^{(1)}\Big)}{\delta_q\Big(\frac{\lambda}{\hbar}, \frac{\omega}{\hbar}\Big)}. \qedhere
\]
\end{proof}

\subsection{The proof of the Macdonald-Ruijsenaars equations}

We are now ready to deduce the Macdonald-Ruijsenaars equations from Proposition \ref{prop:trace-coeff} and Lemma \ref{lem:g-comp}.  Recalling the action of $C_W$ on $M_{\mu, k}$ from Proposition \ref{prop:eti-cent}, we find that
\begin{multline*}
\chi_W(q^{-2\mu - 2kd - 2\wrho}) \wPsi^{V_1, \ldots, V_n}(z_1, \ldots, z_n; \lambda, \omega, \mu, k)
= \sum_{\nu \in \hh^*, a \in \CC} q^{-2(k + h)a} \frac{\delta_q(\lambda - \nu, \omega - k_W)}{\delta_q(\lambda, \omega)}\\ \Tr|_{W[\nu + a \delta + k_W \Lambda_0]} \Big(\RR_{WV}(1; z_1, \ldots, z_n; \lambda, \omega)\Big) \wPsi^{V_1, \ldots, V_n}(z_1, \ldots, z_n; \lambda - \nu, \omega - k_W, \mu, k).
\end{multline*}
By Lemma \ref{lem:c-fuse}, we have that
\begin{multline*}
\RR_{WV}(1; z_1, \ldots, z_n; \lambda, \omega)\\ = \JJ_{V_1[z_1^{\pm 1}], V_2[z_2^{\pm 1}] \otimes \cdots \otimes V_n[z_n^{\pm 1}]}(z_1; z_2, \ldots, z_n; \lambda, \omega) \cdots \JJ_{V_{n - 1}[z_{n-1}^{\pm 1}], V_n[z_n^{\pm 1}]}(z_{n-1}; z_n; \lambda, \omega)\\
\RR_{WV_1}\Big(1, z_1; (\lambda, \omega) - h^{(2\cdots n)}\Big)\cdots \RR_{WV_n}(1, z_n; \lambda, \omega)\\ \Big(\JJ_{V_1[z_1^{\pm 1}], V_2[z_2^{\pm 1}] \otimes \cdots \otimes V_n[z_n^{\pm 1}]}(z_1; z_2, \ldots, z_n; \lambda - \nu, \omega - k_W) \\ \cdots  \JJ_{V_{n - 1}[z_{n-1}^{\pm 1}], V_n[z_n^{\pm 1}]}(z_{n-1}; z_n; \lambda - \nu, \omega - k_W)\Big)^{-1}.
\end{multline*}
Multiplying both sides by $J_{V_1, \ldots, V_n}(z_1, \ldots, z_n; \mu, k)^*$, applying Lemma \ref{lem:trace-adj}, and substituting in, we obtain
\begin{multline*}
\chi_W(q^{-2\mu - 2kd - 2\wrho}) \JJ^{1 \cdots n}_V(z_1, \ldots, z_n; \lambda, \omega)^{-1}\delta_q(\lambda, \omega) \Psi^{V_1, \ldots, V_n}(z_1, \ldots, z_n; \lambda, \omega, \mu, k) \\
= \sum_{\nu \in \hh^*, a \in \CC} q^{-2(k + h)a} \Tr|_{W[\nu + a \delta + k_W \Lambda_0]}\Big(\RR_{WV_1}\Big(1, z_1; (\lambda, \omega) - h^{(2\cdots n)}\Big) \cdots \RR_{WV_n}(1, z_n; \lambda, \omega)\Big) \\ 
\JJ^{1 \cdots n}_V(z_1, \ldots, z_n; \lambda - \nu, \omega - k_W)^{-1}\delta_q(\lambda - \nu, \omega - k_W)  \Psi^{V_1, \ldots, V_n}(z_1, \ldots, z_n; \lambda - \nu, \omega - k_W, \mu, k).
\end{multline*}
Recalling the definition of the normalized trace $F^{V_1, \ldots, V_n}(z_1, \ldots, z_n; \lambda, \omega, \mu, k)$ now yields
\[
\DD_W(\omega, k) F^{V_1, \ldots, V_n}(z_1, \ldots, z_n; \lambda, \omega, \mu, k) = \chi_W(q^{-2\mu - 2kd}) F^{V_1, \ldots, V_n}(z_1, \ldots, z_n; \lambda, \omega, \mu, k).
\]

\subsection{Computations for the Macdonald-Ruijsenaars equations} \label{sec:mr-comps}

In this subsection we give proofs of Lemmas \ref{lem:z-val} and \ref{lem:x-val}.

\begin{proof}[Proof of Lemma \ref{lem:z-val}]
Label the tensor factors of $M_{\mu, k} \otimes V \otimes V^* \otimes U_q(\wgg) \otimes U_q(\wgg)$ by $0$, $1$, $1*$, $2$, and $3$ in that order. By moving $\cR^{20}$ around the trace, we have
\begin{align*}
Z_{V_1, \ldots, V_n}(z_1, \ldots, z_n; \lambda, \omega, \mu, k) &= \Tr|_{M_{\mu, k}}\Big(\wPhi_{\mu, k}^{V_1, \ldots, V_n}(z_1, \ldots, z_n) \cR^{20} q_0^{2\lambda + 2\omega d}\Big) \\
&= q_2^{2\lambda + 2 \omega d} \Tr|_{M_{\mu, k}}\Big(\wPhi_{\mu, k}^{V_1, \ldots, V_n}(z_1, \ldots, z_n) q_0^{2\lambda + 2 \omega d} \cR^{20}\Big) q_2^{-2\lambda + 2 \omega d}\\
&= q_2^{2\lambda + 2 \omega d} \Tr|_{M_{\mu, k}}\Big(\cR^{20}\wPhi_{\mu, k}^{V_1, \ldots, V_n}(z_1, \ldots, z_n) q_0^{2\lambda + 2 \omega d}\Big) q_2^{-2\lambda + 2 \omega d}\\
&= q_2^{2\lambda + 2 \omega d} \Tr|_{M_{\mu, k}}\Big((\cR^{21})^{-1}\wPhi_{\mu, k}^{V_1, \ldots, V_n}(z_1, \ldots, z_n) \cR^{20} q_0^{2\lambda + 2 \omega d}\Big) q_2^{-2\lambda - 2 \omega d}\\
&= q_2^{2\lambda + 2 \omega d}(\cR^{21})^{-1} Z_{V_1, \ldots, V_n}(z_1, \ldots, z_n; \lambda, \omega, \mu, k) q_2^{-2\lambda - 2 \omega d}.
\end{align*}
where we note that $(\cR^{21})^{-1} \wPhi_{\mu, k}^{V_1, \ldots, V_n}(z_1, \ldots, z_n) \cR^{20} = \cR^{20} \wPhi_{\mu, k}^{V_1, \ldots, V_n}(z_1, \ldots, z_n)$.  Denote the claimed expression for $Z_{V_1, \ldots, V_n}(z_1, \ldots, z_n; \lambda, \omega, \mu, k)$ by $Z_{V_1, \ldots, V_n}'(z_1, \ldots, z_n; \lambda, \omega, \mu, k)$. Notice that 
\begin{align*}
q_2^{2\lambda + 2 \omega d}&(\cR^{21})^{-1} Z_{V_1, \ldots, V_n}'(z_1, \ldots, z_n; \lambda, \omega, \mu, k) q_2^{-2\lambda - 2 \omega d}\\ &= q_2^{2\lambda + 2\omega d}(\cR^{21})^{-1}\cJ^{12}(\lambda, \omega) q_2^{-kd} \wPsi^{V_1, \ldots, V_n}(z_1, \ldots, z_n; (\lambda, \omega) - h^{(2)}/2, \mu, k) q_2^{-2\lambda - 2\omega d}\\
&= q_2^{2\lambda + 2 \omega d}q_1^{2\lambda + 2 \omega d}\cJ^{12}(\lambda, \omega) q_1^{-2\lambda - 2 \omega d} q^{\Omega_{12}}q_2^{-kd}\wPsi^{V_1, \ldots, V_n}(z_1, \ldots, z_n; (\lambda, \omega) - h^{(2)}/2, \mu, k) q_2^{-2\lambda - 2 \omega d}\\
&= \cJ^{12}(\lambda, \omega) q^{\Omega_{12}}q_2^{-kd} \wPsi^{V_1, \ldots, V_n}(z_1, \ldots, z_n; (\lambda, \omega) - h^{(2)}/2, \mu, k)\\
&= Z_{V_1, \ldots, V_n}'(z_1, \ldots, z_n; \lambda, \omega, \mu, k),
\end{align*}
where we used the ABRR equation and applied
\[
q^{\Omega_{12}} \wPsi^{V_1, \ldots, V_n}(z_1, \ldots, z_n; (\lambda, \omega) - h^{(2)}/2, \mu, k) = \wPsi^{V_1, \ldots, V_n}(z_1, \ldots, z_n; (\lambda, \omega) - h^{(2)}/2, \mu, k).
\]
Notice now that the weight $0$ term in tensor factor $2$ for $Z_{V_1, \ldots, V_n}(z_1, \ldots, z_n; \lambda, \omega, \mu, k)$ is given by 
\[
\Tr|_{M_{\mu, k}}\Big(\wPhi_{\mu, k}^{V_1, \ldots, V_n}(z_1, \ldots z_n) q_{02}^{-\Omega} q_0^{-2\lambda - 2\omega d}\Big) = \wPsi^{V_1, \ldots, V_n}(z_1, \ldots, z_n; (\lambda, \omega) - h^{(2)}/2, \mu, k) q_2^{-kd}.
\]
We conclude that both $Z_{V_1, \ldots, V_n}(z_1, \ldots, z_n; \lambda, \omega, \mu, k)$ and $Z_{V_1, \ldots, V_n}'(z_1, \ldots, z_n; \lambda, \omega, \mu, k)$ are solutions to 
\[
q_2^{2\lambda + 2 \omega d} (\cR^{21})^{-1} Z = Z q_2^{2\lambda + 2 \omega d}
\]
whose weight $0$ term in tensor factor $2$ is equal to $\wPsi^{V_1, \ldots, V_n}(z_1, \ldots, z_n; (\lambda, \omega) - h^{(2)}/2, \mu, k) q_2^{-kd}$, hence they are equal.
\end{proof}

\begin{proof}[Proof of Lemma \ref{lem:x-val}]
By moving $(\cR^{03})^{-1}$ around the trace, we obtain
\begin{align*}
X_{V_1, \ldots, V_n}(z_1, \ldots, z_n; \lambda, \omega, \mu, k) &= q_3^{2\lambda + 2 \omega d} \Tr|_{M_{\mu, k}}\Big(\wPhi_{\mu, k}^{V_1, \ldots, V_n}(z_1, \ldots, z_n) \cR^{20} q_0^{2\lambda + 2 \omega d} (\cR^{03})^{-1}\Big) q_3^{-2\lambda - 2\omega d}\\
&= q_3^{2\lambda + 2\omega d} \Tr|_{M_{\mu, k}}\Big((\cR^{03})^{-1}\wPhi_{\mu, k}^{V_1, \ldots, V_n}(z_1, \ldots, z_n) \cR^{20} q_0^{2\lambda + 2 \omega d}\Big) q_3^{-2\lambda - 2 \omega d}\\
&= q_3^{2\lambda + 2\omega d} \Tr|_{M_{\mu, k}}\Big(\cR^{13} \wPhi_{\mu, k}^{V_1, \ldots, V_n}(z_1, \ldots, z_n) (\cR^{03})^{-1} \cR^{20} q_0^{2\lambda + 2\omega d}\Big) q_3^{-2\lambda - 2\omega d}\\
&= q_3^{2\lambda + 2\omega d}\cR^{13} \cR^{23} X_{V_1, \ldots, V_n}(z_1, \ldots, z_n; \lambda, \omega, \mu, k) (\cR^{23})^{-1} q_3^{-2\lambda - 2\omega d}
\end{align*}
where we note that
\[
(\cR^{03})^{-1} \wPhi_{\mu, k}^{V_1, \ldots, V_n}(z_1, \ldots, z_n) = \cR^{13} \wPhi_{\mu, k}^{V_1, \ldots, V_n}(z_1, \ldots, z_n) (\cR^{03})^{-1} \text{ and } (\cR^{03})^{-1} \cR^{20} = \cR^{23} \cR^{20} (\cR^{03})^{-1} (\cR^{23})^{-1}.
\]
Now, denote by $X_{V_1, \ldots, V_n}'(z_1, \ldots, z_n; \lambda, \omega, \mu, k)$ the claimed expression for $X_{V_1, \ldots, V_n}(z_1, \ldots, z_n; \lambda, \omega, \mu, k)$.  We have that 
\begin{align*}
&q_3^{2\lambda + 2\omega d} \cR^{13} \cR^{23} X_{V_1, \ldots, V_n}'(z_1, \ldots, z_n; \lambda, \omega, \mu, k) (\cR^{23})^{-1} q_3^{-2\lambda - 2\omega d} \\
&= q_3^{2\lambda + 2\omega d} \cR^{12, 3} q_3^{2\lambda + 2 \omega d} \cJ^{3, 12}(\lambda, \omega)  \cJ^{12}((\lambda, \omega) + h^{(3)}/2) q_2^{-kd}q_3^{kd}\\
&\phantom{=========} \wPsi^{V_1, \ldots, V_n}(z_1, \ldots, z_n; (\lambda, \omega) + (h^{(3)} - h^{(2)})/2, \mu, k)  \cJ^{32}(\lambda, \omega)^{-1} q_3^{-2\lambda - 2\omega d} (\cR^{23})^{-1} q_3^{-2\lambda - 2\omega d}\\
&= q_3^{2\lambda + 2\omega d} \cJ^{3, 12}(\lambda, \omega) q_3^{2 \lambda + 2 \omega d} q^{-\Omega_{12, 3}} \cJ^{12}((\lambda, \omega) + h^{(3)}/2)q_2^{-kd}q_3^{kd}\\
&\phantom{=========} \wPsi^{V_1, \ldots, V_n}(z_1, \ldots, z_n; (\lambda, \omega) + (h^{(3)} - h^{(2)})/2, \mu, k) q^{\Omega_{32}} q_3^{-2\lambda - 2 \omega d} \cJ^{32}(\lambda, \omega)^{-1} q_3^{-2\lambda - 2 \omega d}\\
&= q_3^{2\lambda + 2 \omega d} \cJ^{3, 12}(\lambda, \omega) \cJ^{12}((\lambda, \omega) + h^{(3)}/2) q_2^{-kd}q_3^{kd}\\
&\phantom{=========} \wPsi^{V_1, \ldots, V_n}(z_1, \ldots, z_n; (\lambda, \omega) + (h^{(3)} - h^{(2)})/2, \mu, k) \cJ^{32}(\lambda, \omega)^{-1} q_3^{-2\lambda - 2\omega d}\\
&= X_{V_1, \ldots, V_n}'(z_1, \ldots, z_n; \lambda, \omega, \mu, k),
\end{align*}
where we used that $\cR^{13} \cR^{23} = \cR^{12, 3}$ and 
\[
q^{-\Omega_{1, 3}} \wPsi^{V_1, \ldots, V_n}(z_1, \ldots, z_n; (\lambda, \omega) + (h^{(3)} - h^{(2)})/2, \mu, k) = \wPsi^{V_1, \ldots, V_n}(z_1, \ldots, z_n; (\lambda, \omega) + (h^{(3)} - h^{(2)})/2, \mu, k).
\]
Therefore, both $X_{V_1, \ldots, V_n}(z_1, \ldots, z_n; \lambda, \omega, \mu, k)$ and $X_{V_1, \ldots, V_n}'(z_1, \ldots, z_n; \lambda, \omega, \mu, k)$ are solutions to 
\[
q_3^{2\lambda + 2\omega d}\cR^{13} \cR^{23} Z = Z q_3^{2\lambda + 2\omega d} \cR^{23}.
\]
Define the quantities
\begin{align*}
Y_{V_1, \ldots, V_n}(z_1, \ldots, z_n; \lambda, \omega, \mu, k) &= \cJ^{3, 12}(\lambda, \omega)^{-1} q_3^{-2\lambda - 2\omega d} X_{V_1, \ldots, V_n}(z_1, \ldots, z_n; \lambda, \omega, \mu, k) q_3^{2\lambda + 2\omega d} \cJ^{32}(\lambda, \omega)\\
Y_{V_1, \ldots, V_n}'(z_1, \ldots, z_n; \lambda, \omega, \mu, k) &= \cJ^{3, 12}(\lambda, \omega)^{-1} q_3^{-2\lambda - 2\omega d} X_{V_1, \ldots, V_n}'(z_1, \ldots, z_n; \lambda, \omega, \mu, k) q_3^{2\lambda + 2\omega d} \cJ^{32}(\lambda, \omega)
\end{align*}
so that both $Y_{V_1, \ldots, V_n}(z_1, \ldots, z_n; \lambda, \omega, \mu, k)$ and $Y_{V_1, \ldots, V_n}'(z_1, \ldots, z_n; \lambda, \omega, \mu, k)$ are solutions to 
\[
q_3^{2\lambda + 2\omega d}q^{-\Omega_{12, 3}} Z = Z q_3^{2\lambda + 2\omega d}q^{-\Omega_{32}}
\]
with weight $0$ term in the third tensor factor given by
\[
\cJ^{12}((\lambda, \omega) + h^{(3)}/2)q_2^{-kd}q_3^{kd} \wPhi^{V_1, \ldots, V_n}(z_1, \ldots, z_n; (\lambda, \omega) + (h^{(3)} - h^{(2)})/2, \mu, k)
\]
by Lemma \ref{lem:z-val}, yielding the conclusion.
\end{proof}

\section{Dual Macdonald-Ruijsenaars equations} \label{sec:dmr}

In this section we prove the dual Macdonald-Ruijsenaars equations for $F^{V_1, \ldots, V_n}(z_1, \ldots, z_n; \lambda, \omega, \mu, k)$.  Our method proceeds by showing that two naturally defined intertwiners are related by an application of a dynamical $R$-matrix, after which the result follows by computing the trace of a single intertwiner in two different ways.

\subsection{The statement}

Let $W$ be an integrable lowest-weight $U_q(\wgg)$-module of level $k_W$, and define the difference operator
\[
\DD_W^\vee(\omega, k) = \sum_{\nu \in \hh^*, a \in \CC} \Tr|_{W[\nu + a \delta + k_W \Lambda_0]}\Big(\RR_{WV_n^*}(1, z_n; (\mu, k) - h^{(* 1 \cdots * (n-1))}) \cdots \RR_{WV_1^*}(1, z_1; \mu, k)\Big) q^{-2\omega a} T^{\mu, k}_{\nu, k_W},
\]
where $T^{\mu, k}_{\nu, k_W}f(\mu, k) = f(\mu - \nu, k - k_W)$. Let this operator act on functions valued in 
\[
(V_1[z_1^{\pm}] \otimes \cdots \otimes V_n[z_n^{\pm} 1]) \otimes (V_n^* \otimes \cdots \otimes V_1^*),
\]
where we interpret $\RR_{WV_m^*}(1, z_m; \mu, k)$ as the evaluation of the universal fusion matrix on $W \otimes V_m^*[z_m^{\pm 1}]$.  The dual Macdonald-Ruijsenaars equations state that $\DD_W^\vee(\omega, k)$ are diagonalized on renormalized trace functions.

\begin{theorem}[dual Macdonald-Ruijsenaars equation] \label{thm:dmr}
For any integrable lowest weight representation $W$ of non-positive integer level $k_W$, we have
\[
\DD_W^\vee(\omega, k) F^{V_1, \ldots, V_n}(z_1, \ldots, z_n; \lambda, \omega, \mu, k) = \chi_W(q^{-2\lambda - 2 \omega d}) F^{V_1, \ldots, V_n}(z_1, \ldots, z_n; \lambda, \omega, \mu, k),
\]
where $\chi_W$ is the character of $W$.
\end{theorem}

\subsection{Computing an intertwiner in two different ways}

Suppose that $(\mu, k)$ is generic so that all Verma modules $M_{\lambda, k}$ with highest weight given by a shift of $\mu + k \Lambda_0$ by an integral weight are irreducible.  Let $W$ be a highest weight irreducible integrable module of level $k_W$.  By Proposition \ref{prop:int-exist}, we have an isomorphism of $U_q(\wgg)$-modules
\[
\eta: \bigoplus_{\lambda, a} W[\lambda - \mu + k_W \Lambda_0 - a \delta] \otimes M_{\lambda + (k + k_W)\Lambda_0 - a \delta} \to M_{\mu, k} \otimes W.
\]
For finite-dimensional $U_q(\agg)$-representations $V_1, \ldots, V_n$, consider the representations
\[
V := V_1[z_1^{\pm 1}] \otimes \cdots \otimes V_n[z_n^{\pm 1}] \qquad \text{ and } \qquad \hV := V_1((z_1)) \otimes \cdots \otimes V_n((z_n)).
\]
We abuse notation to also define the vector space
\[
V^* := V_n^* \otimes \cdots \otimes V_1^*.
\]
This isomorphism will provide two natural ways of constructing an intertwiner
\[
M_{\mu, k} \otimes W \to M_{\mu, k} \otimes W \hotimes \hV \otimes V^*,
\]
which will be related by the application of a dynamical $R$-matrix.

\begin{prop} \label{prop:l-comm}
For any finite-dimensional $U_q(\agg)$-representation $V$, the following diagram commutes.
\begin{diagram}
M_{\mu, k} \otimes W & \rTo^{\wPhi^{V_1, \ldots, V_n}_{\mu, k}(z_1, \ldots, z_n)} & M_{\mu, k} \hotimes V \otimes W \otimes V^*\\
\uTo_{\eta} & &  \\
\bigoplus_{\nu, a} W[\nu - \mu + k_W \Lambda_0 - a \delta] \otimes M_{\nu, k + k_W, -a} & &  \\
\dTo_{\wPhi^{V_1, \ldots, V_n}_{\nu, k + k_W, -a}(z_1, \ldots, z_n)} & &\dTo^{P_{\hV W} \cR_{VW}}\\
\bigoplus_{\nu, a} W[\nu - \mu + k_W \Lambda_0 - a \delta] \otimes M_{\nu, k + k_W, -a} \hotimes V \otimes V^* &  & \\
\dTo_{\RR_{WV}(1; z_1, \ldots, z_n; \nu + \rho, k + k_W + \ch)^{*V}} & & \\
\bigoplus_{\nu, a} W[\nu - \mu + k_W\Lambda_0 - a \delta] \otimes M_{\nu, k + k_W, -a} \hotimes \hV \otimes V^* & \rTo^{\eta} & M_{\mu, k} \otimes W \hotimes \hV \otimes V^*\\ 
\end{diagram}
\end{prop}
\begin{proof}
For each choice of $(\nu, a)$ and $w \in W[\nu - \mu + k_W\Lambda_0 - a \delta]$, restricting each branch of the diagram to $w \otimes M_{\nu, k + k_W, -a}$ gives two intertwiners $M_{\nu, k + k_W, -a} \to M_{\mu, k} \otimes W \hotimes \hV \otimes V^*$, where we note that the top branch is an intertwiner by Lemma \ref{lem:r-inter}.  To check that these intertwiners are equal, it suffices by Proposition \ref{prop:int-exist} to check that they have the same highest term.  Suppose that 
\[
\RR_{WV}(1; z_1, \ldots, z_n; \nu + \rho, k + k_W + \ch)^{*V} = \sum_j p_j \otimes q_j,
\]
where $p_j \in \End(W)$ and $q_j(z) \in \End(V^*)((z_1, \ldots, z_n))$, and let $\{v_i\}$ and $\{v_i^*\}$ be dual bases of $V_1 \otimes \cdots \otimes V_n$ and $V^*$.  Applying Proposition \ref{prop:fus-eval}, the highest term of the top branch is given by
\begin{align} \nonumber
&\left\langle P_{VW} \cR_{VW} \wPhi^{V_1, \ldots, V_n}_{\mu, k}(z_1, \ldots, z_n) \Phi^w_{\nu, k + k_W, -a}\right\rangle \\ \nonumber
 &\phantom{==========}= \sum_i P_{VW}\cR_{VW} \JJ_{VW}(z_1, \ldots, z_n; 1 ;\nu + \rho, k + k_W + \ch) (v_i \otimes w) \otimes v_i^* \\\label{eq:high-top}
&\phantom{==========}= \sum_i \cR_{WV}^{21} \JJ_{WV}^{21}(1; z_1, \ldots, z_n; \nu + \rho, k + k_W + \ch) (w \otimes v_i) \otimes v_i^*.
\end{align}
On the other hand, the highest term of the top branch is given by 
\begin{align}\nonumber
&\sum_{i, j} \left\langle \Phi^{p_j w}_{\mu, k} \wPhi^{v_i}_{\nu, k + k_W, -a}(z_1, \ldots, z_n)\right\rangle \otimes q_j v_i^*\\ \nonumber
 &\phantom{=======}= \sum_j \JJ_{WV}(1; z_1, \ldots, z_n; \nu + \rho, k + k_W + \ch) \RR_{WV}(1; z_1, \ldots, z_n; \nu, k + k_W) (w \otimes v_i) \otimes v_i^*\\  \label{eq:high-bottom}
 &\phantom{=======}= \sum_i \cR^{21}_{WV} \JJ^{21}_{WV}(1; z_1, \ldots, z_n; \nu + \rho, k + k_W + \ch) (w \otimes v_i) \otimes v_i^*.
\end{align}
Comparing (\ref{eq:high-top}) and (\ref{eq:high-bottom}) yields the desired.
\end{proof}

\subsection{Computing the double dual of the dynamical $R$-matrix}

We require also the following computation of the double dual of the dynamical $R$-matrix.

\begin{lemma} \label{lem:r-trans-val}
The value of $\RR_{WV}(1; z_1, \ldots, z_n; \mu, k)^{*W*V}$ is given by 
\begin{multline*}
\RR_{WV}(1; z_1, \ldots, z_n; \mu, k)^{*W*V} = \Big(\QQ(\mu, k) \otimes  \QQ((\mu, k) + h^{(1)}) \Big) \\ \RR_{W^* V^*}(1; z_1, \ldots, z_n; (\mu, k) + h^{(1)} + h^{(2)}) \Big(\QQ((\mu, k) + h^{(2)})^{-1}\otimes \QQ(\mu, k)^{-1}\Big).
\end{multline*}
\end{lemma}
\begin{proof}
Recalling that 
\[
\RR(\mu, k) = \LL(\mu, k)^{-1} \cR \LL^{21}(\mu, k),
\]
we see that 
\begin{multline*}
\RR_{WV}(1; z_1, \ldots, z_n; \mu, k)^{*W*V} = (S^{-1} \otimes S^{-1})(\LL^{21}(\mu, k))|_{W^* \otimes (V_n^*[z_n^{\pm 1}] \otimes \cdots \otimes V_1^*[z_1^{\pm 1}])}\\
 (S^{-1} \otimes S^{-1})(\cR_{W^* V^*}) (S^{-1} \otimes S^{-1})(\LL(\mu, k))^{-1}|_{W^* \otimes (V_n^*[z_n^{\pm 1}] \otimes \cdots \otimes V_1^*[z_1^{\pm 1}])}.
\end{multline*}
Because $\LL(\mu, k)$ and $\cR$ are weight zero, we may replace $S^{-1} \otimes S^{-1}$ with $S \otimes S$ in the equation above.  By Lemma \ref{lem:q-coprod}, we have that 
\begin{multline*}
(S \otimes S)(\LL(\mu, k)) = (S \otimes S)(\cL((\mu, k) + h^{(1)}/2 + h^{(2)}/2))\\
= \Big(\QQ((\mu, k) + h^{(2)}) \otimes \QQ(\mu, k)\Big) \cL((\mu, k) + h^{(1)}/2 + h^{(2)}/2)^{-1} \Delta^{21}(\QQ(\mu, k)^{-1}).
\end{multline*}
Substituting in, we conclude that
\begin{align*}
\RR_{WV}(1; z_1, \ldots, z_n; &\mu, k)^{*W*V}\\
 &= \Big(\QQ_{W^*}(\mu, k) \otimes \QQ_{V^*}((\mu, k) + h^{(1)})\Big) \cL_{W^*V^*}((\mu, k) + h^{(1)}/2 + h^{(2)}/2)^{-1} \\
&\phantom{===} \Delta^{21}_{W^*V^*}(\QQ(\mu, k)^{-1}) \cR_{W^* V^*} \Delta_{W^*V^*}(\QQ(\mu, k))\\
&\phantom{===} \cL^{21}_{W^*V^*}((\mu, k) + h^{(1)}/2 + h^{(2)}/2) \Big(\QQ_{W^*}((\mu, k) + h^{(2)})^{-1} \otimes \QQ_{V^*}(\mu, k)^{-1}\Big)\\
&= \Big(\QQ_{W^*}(\mu, k) \otimes \QQ_{V^*}((\mu, k) + h^{(1)})\Big) \RR_{W^*V^*}(1; z_1, \ldots, z_n; (\mu, k) + h^{(1)} + h^{(2)})\\
&\phantom{===} \Big(\QQ_{W^*}((\mu, k) + h^{(2)})^{-1} \otimes \QQ_{V^*}(\mu, k)^{-1}\Big). \qedhere
\end{align*}
\end{proof}

\subsection{Proof of the dual Macdonald-Ruijsenaars identities}

We are now ready to prove Theorem \ref{thm:dmr}.  If $W$ is a lowest weight integrable module of level $k_W$, then $W^*$ is a highest weight integrable module of level $-k_W$, for which we may apply Proposition \ref{prop:l-comm} to obtain the equality
\begin{multline} \label{eq:conj-inter-eq}
P_{VW^*}\cR_{VW^*} \wPhi^V_{\mu, k}(z_1, \ldots, z_n) \\ = \eta \circ \RR_{W^*V}(1; z_1, \ldots, z_n;\nu + \rho, k - k_W + \ch)^{*V} \circ \wPhi^V_{\nu, k - k_W, -a}(z_1, \ldots, z_n) \circ \eta^{-1}
\end{multline}
of intertwiners $M_{\mu, k} \otimes W^* \to M_{\mu, k} \otimes W^* \hotimes \hV \otimes V^*$.  Consider the trace of both sides of (\ref{eq:conj-inter-eq}) precomposed with $q^{2\lambda + 2\omega d}$ and postcomposed with $q^{\Omega_{VW^*}}$.  Computing using the left hand expression of (\ref{eq:conj-inter-eq}), we note that only terms involving the diagonal term of $\cR_{VW^*}(z)$ contribute; since this diagonal term is $q^{-\Omega_{VW^*}}$, we conclude the trace is equal to
\begin{equation} \label{eq:tr-first}
\chi_{W^*}(q^{2\lambda + 2 \omega d}) \wPsi^V(z_1, \ldots, z_n; \lambda, \omega, \mu, k).
\end{equation}
Computing using the right hand expression of (\ref{eq:conj-inter-eq}), we note that the value of the trace has zero weight in $\hV$, hence $q^{-\Omega_{VW^*}}$ evaluates to $1$.  Therefore, in computing the trace we may ignore both $q^{-\Omega_{VW^*}}$ and the conjugation by $\eta$, obtaining
\begin{multline} \label{eq:tr-second}
\sum_{\nu, a} \Tr|_{W^*[\nu - \mu - k_W\Lambda_0 - a \delta]}\Big(\RR_{W^*V}(1; z_1, \ldots, z_n; \nu + \rho, k - k_W + \ch)^{*V}\Big)\\ q^{-2\omega a} \wPsi^V(z_1, \ldots, z_n; \lambda, \omega, \nu, k - k_W) \\ = \sum_{\nu, a} \Tr|_{W^*[\nu - k_W\Lambda_0 - a\delta]}\Big(\RR_{W^*V}(1; z_1, \ldots, z_n; \mu + \nu + \rho, k - k_W + \ch)^{*V}\Big)\\ q^{-2\omega a} \wPsi^V(z_1, \ldots, z_n; \lambda, \omega, \mu + \nu, k - k_W).
\end{multline}
Recall now that $V$ is a tensor product of evaluation representations.  Equating (\ref{eq:tr-first}) and (\ref{eq:tr-second}), multiplying on the left by $J_{V_1, \ldots, V_n}(z_1, \ldots, z_n; \mu, k)^*$, and applying Lemma \ref{lem:trace-adj}, we find that 
\begin{multline*}
\chi_W(q^{-2\lambda - 2\omega d}) \Psi^{V_1, \ldots, V_n}(z_1, \ldots, z_n; \lambda, \omega, \mu, k)\\
= \sum_{\nu, a} J_{V_1, \ldots, V_n}(z_1, \ldots, z_n; \mu, k)^* \Tr|_{W^*[\nu - k_W \Lambda_0 - a \delta]}\Big(\RR_{W^*V}(1; z_1, \ldots, z_n; \mu + \nu + \rho, k - k_W + \ch)^{*V}\Big)\\ q^{-2\omega a} \wPsi^V(z_1, \ldots, z_n; \lambda, \omega, \mu + \nu, k - k_W).
\end{multline*}
Now, by Corollary \ref{corr:fus-factor}, we see that 
\[
J_{V_1, \ldots, V_n}(z_1, \ldots, z_n; \mu, k)^* = \JJ_{V_{n - 1}[z_{n - 1}^{\pm 1}], V_n[z_n^{\pm 1}]}(\mu + \rho, k + \ch)^* \cdots \JJ_{V_1[z_1^{\pm 1}]; V_2[z_2^{\pm 1}] \otimes \cdots \otimes V_n[z_n^{\pm 1}]}(\mu + \rho, k + \ch)^*
\]
and therefore by Lemma \ref{lem:c-fuse} that on $W^*[\nu - k_W \Lambda_0 - a\delta]$ we have
\begin{multline*}
J_{V_1, \ldots, V_n}(z_1, \ldots, z_n; \mu, k)^* \RR_{W^*V}(1; z_1, \ldots, z_n; \mu + \nu + \rho, k - k_W + \ch)^{*V}\\
= \RR_{W^*V_n}(1, z_n; \mu + \nu + \rho, k - k_W + \ch)^{*V_n} \cdots \RR_{W^*V_1}(1, z_1; (\mu + \nu + \rho, k - k_W + \ch) - h^{(2 \cdots n)})^{*V_1}\\ J_{V_1, \ldots, V_n}(z_1, \ldots, z_n; \mu + \nu, k - k_W)^*.
\end{multline*}
We conclude that 
\begin{align*}
&\chi_W(q^{-2\lambda - 2\omega d}) \Psi^{V_1, \ldots, V_n}(z_1, \ldots, z_n; \lambda, \omega, \mu, k)\\
&= \sum_{\nu, a} \Tr|_{W^*[\nu - k_W \Lambda_0 - a \delta]}\Big(\RR_{W^*V_n}(1, z_n; \mu + \nu + \rho, k - k_W + \ch)^{*V_n} \cdots\\
&\phantom{====} \RR_{W^*V_1}(1, z_1; (\mu + \nu + \rho, k - k_W + \ch) - h^{(2 \cdots n)})^{*V_1}\Big) q^{-2\omega a} J_{V_1, \ldots, V_n}(z_1, \ldots, z_n; \mu + \nu, k - k_W)^* \\
&\phantom{====}\wPsi^V(z_1, \ldots, z_n; \lambda, \omega, \mu + \nu, k - k_W)\\
&= \sum_{\nu, a} \Tr|_{W^{**}[-\nu + k_W \Lambda_0 + a \delta]}\Big(\RR_{W^*V_n}(1, z_n;\mu + \nu + \rho, k - k_W + \ch)^{*W^* *V_n} \cdots\\
&\phantom{====} \RR_{W^*V_1}(1, z_1; (\mu + \nu + \rho, k - k_W + \ch) - h^{(2 \cdots n)})^{*W^* *V_1}\Big) q^{-2\omega a} \Psi^{V_1, \ldots, V_n}(z_1, \ldots, z_n; \lambda, \omega, \mu + \nu, k - k_W).
\end{align*}
By Lemma \ref{lem:r-trans-val}, we see that 
\begin{align*}
&\RR_{W^*V_i}(1, z_i; (\mu + \nu + \rho, k - k_W + \ch) - h^{(i+1 \cdots n)})^{*W^* *V_i} \\
&= \Big(\QQ_{W^{**}}((\mu + \nu + \rho, k - k_W + \ch) - h^{(i+1 \cdots n)}) \QQ_{V_i^*}((\mu + \rho, k + \ch) - h^{(i+1 \cdots n)})\Big)\\
&\phantom{==} \RR_{W^{**}V_i^*}(1, z_i; (\mu + \rho, k + \ch) + h^{(* i)} - h^{(i+1\cdots n)})\\
&\phantom{==} \Big(\QQ_{W^{**}}((\mu + \nu + \rho, k - k_W + \ch) + h^{(* i)} - h^{(i+1\cdots n)})^{-1} \QQ_{V_i^*}((\mu + \nu + \rho, k - k_W + \ch) - h^{(i+1\cdots n)})^{-1}\Big).
\end{align*}
Substituting in, noting that $h^{(* i)} = - h^{(i)}$ and $h^{(1 \cdots n)} = 0$, and canceling common terms in $W^*$, we obtain that 
\begin{align*}
&\chi_W(q^{-2\lambda - 2\omega d}) \Psi^{V_1, \ldots, V_n}(z_1, \ldots, z_n; \lambda, \omega, \mu, k)\\
&= \sum_{\nu, a} \Tr|_{W^{**}[-\nu + k_W \Lambda_0 + a\delta]}\Big(\QQ_{W^{**}}(\mu + \nu + \rho, k - k_W + \ch) \QQ_{V_n^*}(\mu + \rho, k + \ch) \cdots \QQ_{V_1^*}((\mu + \rho, k +\ch) - h^{(2 \cdots n)}) \\
&\phantom{=====}\RR_{W^{**}V_n^*}(1, z_n; (\mu + \rho, k + \ch) + h^{(* n)}) \cdots \RR_{W^{**} V_1^*}(1, z_1; (\mu + \rho, k + \ch) + h^{(* 1)} - h^{(2 \cdots n)})\\
&\phantom{=====}\QQ^{-1}_{W^{**}}((\mu + \nu + \rho, k - k_W + \ch) + h^{(* 1)} - h^{(2\cdots n)}) \QQ^{-1}_{V_n^*}(\mu + \nu + \rho, k - k_W + \ch)\\
&\phantom{=====} \cdots \QQ^{-1}_{V_1^*}((\mu + \nu + \rho, k - k_W + \ch) - h^{(2 \cdots n)})\Big) q^{-2\omega a} \Psi^{V_1, \ldots, V_n}(z_1, \ldots, z_n; \lambda, \omega, \mu + \nu, k - k_W).
\end{align*}
Noting that $W^{**} \simeq W$ and that we may ignore the conjugation by the $\QQ_{W^{**}}$-term in computing the trace, we may simplify this to
\begin{align*}
\chi_W(&q^{-2\lambda - 2\omega d}) \Psi^{V_1, \ldots, V_n}(z_1, \ldots, z_n; \lambda, \omega, \mu, k)\\
&= \sum_{\nu, a} \QQ_{V_n^*}((\mu + \rho, k + \ch) - h^{(* 1 \cdots * n)}) \cdots \QQ_{V_1^*}((\mu + \rho, k + \ch) - h^{(* 1)}) \\
&\phantom{=====}\Tr|_{W[-\nu + k_W \Lambda_0 + a\delta]}\Big(\RR_{WV_n^*}(1, z_n; (\mu + \rho, k + \ch) + h^{(* n)}) \cdots \RR_{W V_1^*}(1, z_1; (\mu + \rho, k + \ch) + h^{(* 1 \cdots * n)})\Big)\\
&\phantom{=====} \QQ_{V_n^*}((\mu + \nu + \rho, k - k_W + \ch) - h^{(* 1 \cdots * n)})^{-1} \cdots \QQ_{V_1^*}((\mu + \nu + \rho, k - k_W + \ch) - h^{(*1)})^{-1}\\
&\phantom{=====} q^{-2\omega a} \Psi^{V_1, \ldots, V_n}(z_1, \ldots, z_n; \lambda, \omega, \mu + \nu, k - k_W).
\end{align*}
Substituting in the definition of $F$, we find the desired
\begin{multline*}
\chi_W(q^{-2\lambda - 2\omega d}) F^{V_1, \ldots, V_n}(z_1, \ldots, z_n; \lambda, \omega, \mu, k) = \sum_{\nu, a}\Tr|_{W[\nu + k_W\Lambda_0 + a \delta]}\Big(\RR_{WV_n^*}(1, z_n; (\mu, k) + h^{(* n)})\cdots \\ \RR_{WV_1^*}(1, z_1; (\mu, k) + h^{(* 1 \cdots * n)})\Big) q^{-2\omega a} F^{V_1, \cdots V_n}(z_1, \ldots, z_n; \lambda, \omega, \mu - \nu, k - k_W).
\end{multline*}

\section{Macdonald symmetry identity} \label{sec:sym}

In this section we prove the Macdonald symmetry identity for $F^{V_1, \ldots, V_n}(z_1, \ldots, z_n; \lambda, \omega, \mu, k)$ under interchange of $(\lambda, \omega)$ and $(\mu, k)$.  Our method uses the observation that the Macdonald-Ruijsenaars and dual Macdonald-Ruijsenaars operators $\DD_W(\lambda, \omega)$ and $\DD_W^\vee(\mu, k)$ are exchanged under this interchange and the fact that the Macdonald-Ruijsenaars equations admit a unique formal solution.

\subsection{The statement}

Theorems \ref{thm:mr} and \ref{thm:dmr} show that $F^{V_1, \ldots V_n}(z_1, \ldots, z_n; \lambda, \omega, \mu, k)$ satisfies dual systems of difference equations. Define the function $F_\star^{V_n^*, \ldots, V_1^*}$ to be the result of interchanging $V_i$ and $V_{n + 1 - i}^*$ in the definition of $F^{V_1, \ldots, V_n}$.  This section is devoted to proving the following symmetry relation.

\begin{theorem}[Macdonald symmetry identity] \label{thm:mrs}
The functions $F^{V_1, \ldots, V_n}$ and $F_\star^{V_n^*, \ldots, V_1^*}$ satisfy the symmetry relation
\[
F^{V_1, \ldots, V_n}(z_1, \ldots, z_n; \lambda, \omega, \mu, k) = F_\star^{V_n^*, \ldots, V_1^*}(z_n, \ldots, z_1; \mu, k, \lambda, \omega).
\]
\end{theorem}

Recall the coefficient rings 
\[
\AA_{\lambda, \omega} = \CC[[q^{-2(\lambda, \alpha_1)}, \ldots, q^{-2(\lambda, \alpha_r)}, q^{-2\omega + 2(\lambda, \theta)}]] \qquad \text{ and } \qquad \AA_{\mu, k} = \CC[[q^{-2(\mu, \alpha_1)}, \ldots, q^{-2(\mu, \alpha_r)}, q^{-2k + 2(\mu, \theta)}]].
\]
Our strategy will be to show that the Macdonald-Ruijsenaars equations admit unique formal solutions over $\AA_{\lambda, \omega}$ and $\AA_{\mu, k}$ with specified leading term.  The fact that $F^{V_1, \ldots, V_n}(z_1, \ldots, z_n; \lambda, \omega, \mu, k)$ satisfies these equations in both sets of variables will then give the conclusion.

\subsection{Formal expansion properties of $F^{V_1, \ldots, V_n}(z_1, \ldots, z_n; \lambda, \omega, \mu, k)$}

In this subsection, we show that the renormalized trace functions admit a formal expansion in a certain coefficient ring.

\begin{lemma} \label{lem:tr-exp-ring}
The renormalized trace function $F^{V_1, \ldots, V_n}(z_1, \ldots, z_n; \lambda, \omega, \mu, k)$ has formal expansion lying in
\[
q^{2(\lambda, \mu)} \AA_{\lambda, \omega} \otimes \AA_{\mu, k} \otimes \CC((z_2/z_1, \ldots, z_n/z_{n-1})) \otimes (V_1 \otimes \cdots \otimes V_n)[0] \otimes (V_n^* \otimes \cdots \otimes V_1^*)[0].
\]
\end{lemma}
\begin{proof}
For $\Psi^{V_1, \ldots, V_n}(z_1, \ldots, z_n; \lambda, \omega, \mu, k)$, this follows from an argument analogous to that of \cite[Proposition 2.6]{Sun:qafftr}.  Now, the normalization factors lie in
\[
 \AA_{\lambda, \omega} \otimes \AA_{\mu, k} \otimes \CC((z_2/z_1, \ldots, z_n/z_{n-1})) \otimes \End((V_1 \otimes \cdots \otimes V_n)[0]) \otimes \End((V_n^* \otimes \cdots \otimes V_1^*)[0])
\]
by definition, so combining these facts yields the desired.
\end{proof}

\subsection{Uniqueness of formal solutions to the Macdonald-Ruijsenaars equations}

We now prove a uniqueness property for formal solutions to the Macdonald-Ruijsenaars equations over $\AA_{\lambda, \omega}$.

\begin{lemma} \label{lem:form-sol}
For each $v \in (V_1 \otimes \cdots \otimes V_n)[0]$, the system of equations
\[
\DD_W(\omega, k) F(\lambda, \omega) = \chi_W(q^{-2\mu - 2kd}) F(\lambda, \omega)
\]
for $W$ ranging over all lowest weight integrable representations has a unique solution $F(\lambda, \omega)$ valued in 
\[
q^{2(\lambda, \mu)} \AA_{\lambda, \omega} \otimes \AA_{\mu, k} \otimes \CC((z_2/z_1, \ldots, z_n/z_{n-1})) \otimes (V_1 \otimes \cdots \otimes V_n)[0]
\]
with leading term $q^{2(\lambda, \mu)} v$ as a series in $\AA_{\lambda, \omega}$.
\end{lemma}
\begin{proof}
Existence follows by taking $\langle v^*, F^{V_1, \ldots, V_n}(z_1, \ldots, z_n; \lambda, \omega, \mu, k)\rangle$ and applying Lemma \ref{lem:tr-exp-ring}, where $v^*$ is dual to $v$.  For uniqueness, suppose that $F_1(\lambda, \omega)$ and $F_2(\lambda, \omega)$ are two such solutions.  If $F'(\lambda, \omega) := F_1(\lambda, \omega) - F_2(\lambda, \omega)$ is non-zero, it is another solution which must contain a (possibly non-unique) leading monomial term of the form
\[
c q^{2(\lambda, \mu)} q^{-\sum_i 2n_i (\lambda, \alpha_i)} q^{-2 m \omega + 2 m (\lambda, \theta)} v' = cq^{(\lambda, 2\mu - 2\sum_i n_i \alpha_i + 2m\theta) - 2m\omega} v'
\]
for some $n_i \geq 0, m \geq 0$ with at least one $n_i, m$ non-zero and $v' \in (V_1 \otimes \cdots \otimes V_n)[0]$.  Notice that $q^{(\lambda, 2\mu - 2\sum_i n_i \alpha_i + 2m\theta) - 2m\omega} v'$ will again be a leading monomial term of $\DD_W(\omega, k) F'(\lambda, \omega)$ with coefficient given by 
\[
c\sum_{\nu, a} \dim W[\nu + k_W \Lambda_0 + a \delta] \, q^{-2ka} q^{-(\nu, 2\mu - 2\sum_i n_i \alpha_i + 2m\theta)} q^{2m k_W} = c\,\chi_W(q^{-2\mu -2kd + 2\sum_i n_i \alpha_i + 2m \alpha_0}).
\]
Since $2 \sum_i n_i \alpha_i + 2m\alpha_0 \neq 0$, the fact that this holds for all lowest weight integrable $W$ contradicts the fact that $\DD_W(\omega, k) F'(\lambda, \omega) = \chi_W(q^{-2\mu - 2kd})F'(\lambda, \omega)$.  We conclude that $F'(\lambda, \omega) = 0$ and the desired solution is unique.
\end{proof}

\subsection{Proof of the symmetry identity}

We are now ready to prove Theorem \ref{thm:mrs}.  Notice that the Macdonald-Ruijsenaars equations for $F^{V_1, \ldots, V_n}(z_1, \ldots, z_n; \lambda, \omega, \mu, k)$ correspond under variable exchange to the dual Macdonald-Ruijsenaars equations for $F_\star^{V_n^*, \ldots, V_1^*}(z_n, \ldots, z_1; \mu, k, \lambda, \omega)$.  By Theorems \ref{thm:mr} and \ref{thm:dmr}, both $F^{V_1, \ldots, V_n}(z_1, \ldots, z_n; \lambda, \omega, \mu, k)$ and $F_\star^{V_n^*, \ldots, V_1^*}(z_n, \ldots, z_1; \mu, k, \lambda, \omega)$ are solutions to
\[
\DD_W(\omega, k) F(\lambda, \omega) = \chi_W(q^{-2\mu - 2kd}) F(\lambda, \omega)
\]
for all lowest weight integrable $W$.  Now, their leading terms with respect to $\AA_{\lambda, \omega}$ are related by an element $M(\mu, k) \in \AA_{\mu, k} \otimes \End((V_n^* \otimes \cdots \otimes V_1^*)[0])$, meaning that
\begin{equation} \label{eq:m-rel}
F^{V_1, \ldots, V_n}(z_1, \ldots, z_n; \lambda, \omega, \mu, k) = M(\mu, k) F_\star^{V_n^*, \ldots, V_1^*}(z_n, \ldots, z_1; \mu, k, \lambda, \omega).
\end{equation}
Repeating this argument with $(\mu, k)$, we find that 
\[
F_\star^{V_n^*, \ldots, V_1^*}(z_n, \ldots, z_1; \mu, k, \lambda, \omega) = M'(\lambda, \omega) F^{V_1, \ldots, V_n}(z_1, \ldots, z_n; \lambda, \omega, \mu, k)
\]
for some $M'(\lambda, \omega) \in \AA_{\lambda, \omega} \otimes \End((V_1 \otimes \cdots \otimes V_n)[0])$.  This implies that $M(\mu, k) M'(\lambda, \omega) = 1$, hence $M(\mu, k)$ lies in $\End((V_n^* \otimes \cdots \otimes V_1^*)[0])$.  Now, comparing leading terms in $\AA_{\lambda, \omega}$ in (\ref{eq:m-rel}) implies that $M(\mu, k) = 1$, yielding the desired.

\section{$q$-KZB and dual $q$-KZB equations} \label{sec:kzb}

In this section we prove the $q$-KZB and dual $q$-KZB equations on behavior of $F^{V_1, \ldots, V_n}(z_1, \ldots, z_n; \lambda, \omega, \mu, k)$ under shifts of the spectral parameters by the modular parameters $q^{-2\omega}$ and $q^{-2k}$.  We directly establish the dual $q$-KZB equations, after which the $q$-KZB equations follow by the symmetry relation of Theorem \ref{thm:mrs}.

\subsection{The statements}

The $q$-KZB operators are defined by 
\begin{align*}
K_j(z_1, \ldots, z_n;& \lambda, \omega, k) := \RR_{V_{j+1}V_j}(z_{j+1}, q^{2k} z_j; (\lambda, \omega) - h^{((j + 2) \cdots n)})^{-1} \cdots \RR_{V_nV_j}(z_n, q^{2k}z_j; \lambda, \omega)^{-1} \Gamma_j\\
&\phantom{=} \RR_{V_jV_1}(z_j, z_1; (\lambda, \omega) - h^{(2 \cdots (j-1))} - h^{((j+1) \cdots n)}) \cdots \RR_{V_j V_{j-1}}(z_j, z_{j-1}; (\lambda, \omega) - h^{((j + 1) \cdots n)})\\
&\phantom{=}D_j(\mu) := q_{*j}^{-2\mu + \sum_i x_i^2} q^{\Omega_{*j, *(j-1)}} \cdots q^{\Omega_{*j*1}},
\end{align*}
where $\Gamma_j f(\lambda, \omega) = f\Big((\lambda, \omega) + h^{(j)}\Big)$.  The dual $q$-KZB operators are defined by
\begin{align*}
K_j^\vee(&z_1, \ldots, z_n; \mu, k, \omega) := \RR_{V_{j-1}^*, V_j^*}(z_{j-1}, q^{2\omega} z_j; (\mu, k) - h^{(*1 \cdots *(j-2))})^{-1} \cdots \RR_{V_1^*, V_j^*}(z_1, q^{2\omega}z_j; \mu, k)^{-1} \Gamma_{*j} \\
&\phantom{=} \RR_{V_j^* V_n^*}(z_j, z_n; (\mu, k) - h^{(*1 \cdots *(j-1))} - h^{(*(j+1) \cdots *(n-1))}) \cdots \RR_{V_j^* V_{j+1}^*}(z_j, z_{j+1}; (\mu, k) - h^{(*1  \cdots *(j-1))})\\
&\phantom{======!}D_j^\vee(\lambda) := q_j^{-2\lambda + \sum_i x_i^2} q^{\Omega_{j, j + 1}} \cdots q^{\Omega_{j, n}},
\end{align*}
where $\Gamma_{*j} f(\mu, k) = f\Big((\mu, k) + h^{(*j)}\Big)$.  Note that $K_j(z_1, \ldots, z_n; \lambda, \omega, k)$ and $K_j^\vee(z_1, \ldots, z_n; \mu, k, \omega)$ are difference operators in $(\lambda, \omega)$ and $(\mu, k)$ whose coefficients are linear operators on $V$ and $V^*$ and that $D_j(\mu)$ and $D_j^\vee(\lambda)$ are linear operators on $V^*$ and $V$.  It is known that the $K_j(z_1, \ldots, z_n; \lambda, \omega, k)$ and the $K_j^\vee(z_1, \ldots, z_n; \mu, k, \omega)$ commute and form the $q$-KZB and dual $q$-KZB integrable systems.  

The $q$-KZB equations relate a spectral shift by the modular parameter $q^{-2k}$ associated to $\mu$ to the action of the difference operator $K_j(z_1, \ldots, z_n; \lambda, \omega, k)$ in $\lambda$.  Symmetrically, the dual $q$-KZB equations relates a spectral shift by the modular parameter $q^{-2\omega}$ associated to $\lambda$ to the action of the difference operator $K_j^\vee(z_1, \ldots, z_n; \mu, k, \omega)$ in $\mu$. In a different form, they were introduced by Felder in \cite{Fel} and studied by Felder-Tarasov-Varchenko in \cite{FTV, FTV2}. The remainder of this section will be devoted to the proof of these two equations, stated below.  We will prove Theorem \ref{thm:dual-qkzb} directly, after which Theorem \ref{thm:qkzb} follows from the symmetry property of Theorem \ref{thm:mrs}.

\begin{theorem}[$q$-KZB equation] \label{thm:qkzb}
For $j = 1, \ldots, n$, we have
\begin{multline*}
F^{V_1, \ldots, V_n}(z_1, \ldots, q^{2k}z_j,\ldots, z_n; \lambda, \omega, \mu, k)\\ = \Big(K_j(z_1, \ldots, z_n;\lambda, \omega, k) \otimes D_j(\mu)\Big) F^{V_1, \ldots, V_n}(z_1, \ldots, z_j, \ldots, z_n; \lambda, \omega, \mu, k).
\end{multline*}
\end{theorem}

\begin{theorem}[dual $q$-KZB equation] \label{thm:dual-qkzb}
For $j = 1, \ldots, n$, we have
\begin{multline*}
F^{V_1, \ldots, V_n}(z_1, \ldots, q^{2\omega}z_j,\ldots, z_n; \lambda, \omega, \mu, k)\\ = \Big(D_j^\vee(\lambda) \otimes K_j^\vee(z_1, \ldots, z_n; \mu, k, \omega)\Big) F^{V_1, \ldots, V_n}(z_1, \ldots, z_j, \ldots, z_n; \lambda, \omega, \mu, k).
\end{multline*}
\end{theorem}

\subsection{Commutation relation for intertwiners}

The fundamental operation in the proof of Theorem \ref{thm:dual-qkzb} is the application of the following commutation relation for intertwiners.

\begin{lemma} \label{lem:inter-comm}
For finite-dimensional $U_q(\agg)$-representations $V$ and $W$, we have the relation
\[
P_{V W} \cR_{V W} \Phi_{\mu, k}^{V, W}(z_1, z_2) = \RR_{W V}(z_2, z_1; \mu + \rho, k + \ch)^* \Phi_{\mu, k}^{W, V}(z_2, z_1).
\]
\end{lemma}
\begin{proof}
Both sides of the desired equality are intertwiners $M_{\mu, k} \to W[z_2^{\pm 1}] \otimes V[z_1^{\pm 1}] \otimes W^* \otimes V^*$.  Let $\{v_i\}$ and $\{v_j\}$ be bases of $V, W$, and let $\{v_i^*\}, \{v_j^*\}$ be the dual bases. The highest term of the left side is given by 
\[
\sum_{i, j} v_j \otimes v_i \otimes J^{21}_{WV}(z_1, z_2; \mu, k)^* (\cR_{WV}^{21})^*(v_j^* \otimes v_i^*)
\]
and the highest term of the right hand side is given by 
\[
\sum_{i, j} v_j \otimes v_i \otimes \RR_{W V}(z_2, z_1; \mu + \rho, k + \ch)^* J_{W V}(z_1, z_2; \mu, k)^* (v_j^* \otimes v_i^*),
\]
so the result follows by noting that Proposition \ref{prop:fus-eval} implies these are equal.
\end{proof}

\subsection{Proof of the dual $q$-KZB equation}

We are now ready to prove the dual $q$-KZB equation.  Rewrite the conclusion of Lemma \ref{lem:inter-comm} as
\[
\Phi_{(\mu, k) - h^{(W)}}^V(z_1) \circ \Phi_{\mu, k}^W(z_2) = \cR_{VW}^{-1} P_{VW} \RR_{WV}(z_2, z_1; \mu + \rho, k + \ch)^*  \Phi_{(\mu, k) - h^{(V)}}^W(z_2) \circ \Phi_{\mu, k}^V(z_1).
\]
Swapping the roles of $V$ and $W$, we obtain also that 
\[
\Phi_{(\mu, k) - h^{(W)}}^V(z_1) \circ \Phi_{\mu, k}^W(z_2) = \RR_{VW}(z_1, z_2; \mu + \rho, k + \ch)^{-*} P_{VW} \cR_{WV} \Phi_{(\mu, k) - h^{(V)}}^W(z_2) \circ \Phi_{\mu, k}^V(z_1).
\]
Apply these commutation relations to commute $\Phi^{V_j}_{(\mu, k) - h^{(j + 1 \cdots n)}}(z_j)$ to the left, apply the cyclic property of the trace to move it to the right, and then apply the commutation relations to commute it back to its original position in
\[
\Psi^{V_1, \ldots, V_n}(z_1, \ldots, z_n; \lambda, \omega, \mu, k) := \Tr|_{M_{\mu, k}}\Big(\Phi^{V_1}_{(\mu, k) - h^{(2 \cdots n)}}(z_1) \cdots \Phi^{V_n}_{\mu, k}(z_n) q^{2\lambda + 2 \omega d}\Big)
\]
to obtain
\begin{align*}
&\Psi^{V_1, \ldots, V_n}(z_1, \ldots, z_n; \lambda, \omega, \mu, k)\\
&= \RR_{V_{j-1} V_j}(z_{j-1}, z_j; (\mu + \rho, k + \ch) - h^{(j + 1 \cdots n)})^{-*} P_{V_{j-1}V_j} \cR_{V_{j}V_{j-1}} \cdots\\
&\phantom{===} \RR_{V_1 V_j}(z_1, z_j; (\mu + \rho, k + \ch) - h^{(2 \cdots j - 1) + (j + 1 \cdots n)})^{-*} P_{V_jV_1} \cR_{V_1 V_j}\\
&\phantom{===} q_j^{-2\lambda - 2\omega d} \Gamma_{*j} \cR_{V_nV_j}^{-1} P_{V_jV_n} \RR_{V_j V_n}(z_j, z_n; \mu + \rho, k + \ch)^* \cdots \cR_{V_{j+1}V_j}^{-1}\\
&\phantom{===} P_{V_{j+1}V_j}\RR_{V_{j}V_{j+1}}(z_j, z_{j+1}; (\mu + \rho, k + \ch) - h^{(j + 2 \cdots n)})^* q_j^{2\omega d} \Psi^{V_1, \ldots, V_n}(z_1, \ldots, z_j, \ldots, z_n; \lambda, \omega, \mu, k)\\
&= \cR_{V_{j}V_{j-1}} \cdots \cR_{V_j V_1} q_j^{-2\lambda - 2 \omega d} \cR_{V_nV_j}^{-1}\cdots \cR_{V_{j+1}V_j}^{-1} q_j^{2\omega d}\\
&\phantom{===} \RR_{V_{j-1} V_j}(z_{j-1}, z_j; (\mu + \rho, k + \ch) - h^{(j + 1 \cdots n)})^{-*} \cdots \RR_{V_1 V_j}(z_1, z_j; (\mu + \rho, k + \ch) - h^{(2 \cdots j - 1) + (j + 1 \cdots n)})^{-*}\\ 
&\phantom{===}\Gamma_{*j} \RR_{V_j V_n}(q^{-2\omega} z_j, z_n; \mu + \rho, k + \ch)^* \cdots \RR_{V_{j}V_{j+1}}(q^{-2\omega} z_j, z_{j+1}; (\mu + \rho, k + \ch) - h^{(j + 2 \cdots n)})^*\\
&\phantom{===} \Psi^{V_1, \ldots, V_n}(z_1, \ldots, q^{-2\omega}z_j, \ldots,  z_n; \lambda, \omega, \mu, k).
\end{align*}
Now, apply Lemma \ref{lem:r-trans-val} and cancel terms to see that 
\begin{multline} \label{eq:k-op-val}
\RR_{V_{j-1} V_j}(z_{j-1}, z_j; (\mu + \rho, k + \ch) - h^{(j + 1 \cdots n)})^{-*} \cdots \\ \RR_{V_1 V_j}(z_1, z_j; (\mu + \rho, k + \ch) - h^{(2 \cdots j - 1) + (j + 1 \cdots n)})^{-*}\\ \Gamma_{*j} \RR_{V_j V_n}(q^{-2\omega} z_j, z_n; \mu + \rho, k + \ch)^* \cdots \RR_{V_{j}V_{j+1}}(q^{-2\omega} z_j, z_{j+1}; (\mu + \rho, k + \ch) - h^{(j + 2 \cdots n)})^*\\
= \Big(\QQ_{V_j^*}((\mu + \rho, k + \ch) + h^{(*(j+1) \cdots *n)}) \QQ_{V_{j-1}^*}((\mu + \rho, k + \ch) + h^{(*j \cdots *n)}) \cdots \QQ_{V_1^*}((\mu + \rho, k + \ch) + h^{(*2 \cdots *n)})\\
 \QQ_{V_n^*}(\mu + \rho, k + \ch) \cdots \QQ_{V_{j + 1}^*}((\mu + \rho, k + \ch) + h^{(*(j+2) \cdots *n)})\Big) K_j^\vee(z_1, \ldots, q^{-2\omega}z_j, \ldots, z_n; \mu + \rho, k + \ch, \omega)\\
\Big(\QQ_{V_{j-1}^*}((\mu + \rho, k + \ch) + h^{(*j\cdots *n)})^{-1} \cdots \QQ_{V_1^*}((\mu + \rho, k + \ch) + h^{(*2 \cdots *n)})^{-1}\\
 \QQ_{V_n^*}(\mu + \rho, k + \ch)^{-1} \cdots \QQ_{V_{j+1}^*}((\mu + \rho, k + \ch) + h^{(*(j+2) \cdots *n)})^{-1} \QQ_{V_j^*}((\mu + \rho, k + \ch) + h^{(*(j+1)\cdots *n)})^{-1}\Big).
\end{multline}
Finally, we claim that on $V[0] \otimes V^*[0]$, we have
\begin{multline} \label{eq:d-op-val}
\JJ^{1 \cdots n}(\lambda, \omega)^{-1} \cR^{j, {j-1}} \cdots \cR^{j1} q_j^{-2\lambda - 2 \omega d} (\cR^{nj})^{-1} \cdots (\cR^{j+1,j})^{-1} \JJ^{1 \cdots n}(\lambda, \omega) q^{2\omega d}\\ = q_j^{-2\lambda - \sum_i x_i^2} q^{-\Omega_{j, j+1}} \cdots q^{-\Omega_{j, n}} = D_j^\vee(\lambda).
\end{multline}
Notice that $\cR^{j, {j-1}} \cdots \cR^{j1} = \cR^{j, 1 \cdots j-1}$ and $(\cR^{nj})^{-1} \cdots (\cR^{j+1,j})^{-1} = (\cR^{j+1\cdots n, j})^{-1}$.  Therefore, to prove (\ref{eq:d-op-val}), it suffices to check it for $n = 3$ and $n = 2$.  For $n = 3$, the product of both sides for $j = 1, 2, 3$ is $1$, hence it suffices to check for $j = 1, 3$, in which case it reduces to the $n = 2$ case.  For $n = 2$, (\ref{eq:d-op-val}) follows by rearranging the ABRR equation of Proposition \ref{prop:es-abrr}. Substituting (\ref{eq:k-op-val}) and (\ref{eq:d-op-val}) into our previous relation and then applying the normalizations of (\ref{eq:norm-tr}) yields the result.

\bibliographystyle{alpha}
\bibliography{qafftr-bib}
\end{document}